\documentclass[11pt,reqno]{amsart} 

\usepackage{graphicx}

\usepackage[top=1.2in,bottom=1.2in,left=1in,right=1in]{geometry}

\usepackage{xcolor}

\usepackage{enumitem}

\usepackage[labelfont=bf,font=sf]{caption}

\usepackage[font={sf,small},labelfont=md]{subfig}

\usepackage[noadjust,nosort]{cite}

\usepackage{amsmath, amsthm, amssymb, tikz, mathtools, array, mathrsfs, tensor}
\usepackage{relsize}

\usetikzlibrary{shapes.geometric}
\tikzset{cross/.style={cross out, draw=black, minimum size=2.5*(#1-\pgflinewidth), inner sep=2pt, outer sep=0.5pt},
	cross/.default={1pt}}
\usetikzlibrary{calc,arrows}
\usepackage{pgfplots}
\usepackage{polynom}

\usepackage{esint}

\usepackage{multicol}

\usepackage{dsfont}

\usepackage{framed}

\usepackage{upref}
\usepackage{subfig}

\usepackage{hyperref}
\hypersetup{colorlinks=true,linkcolor=blue,citecolor=blue,urlcolor=blue} 

\allowdisplaybreaks


\newcommand{\E}{\mathbb{E}}

\newcommand{\N}{\mathbb{N}}
\renewcommand{\P}{\mathbb{P}}
\newcommand{\Q}{\mathbb{Q}}
\newcommand{\R}{\mathbb{R}}

\newcommand{\Z}{\mathbb{Z}}

\newcommand{\FF}{\mathcal{F}}
\newcommand{\GG}{\mathcal{G}}
\newcommand{\HH}{\mathcal{H}}

\newcommand{\JJ}{\mathcal{J}}

\newcommand{\LL}{\mathcal{L}}
\newcommand{\MM}{\mathcal{M}}
\newcommand{\NN}{\mathcal{N}}

\newcommand{\TT}{\mathcal{T}}

\newcommand{\bF}{\textnormal{\textbf{F}}}

\newcommand{\bY}{\textnormal{\textbf{Y}}}
\newcommand{\bW}{\textnormal{\textbf{W}}}
\newcommand{\bH}{\textnormal{\textbf{H}}}

\newcommand{\bG}{\textnormal{\textbf{G}}}

\newcommand{\bX}{\textnormal{\textbf{X}}}
\newcommand{\bB}{\textnormal{\textbf{B}}}
\newcommand{\bx}{\mathbf{x}}

\newcommand{\bt}{\mathbf{t}}

\newcommand{\boo}{\mathbf{0}}
\newcommand{\bd}{\textnormal{dist}}
\newcommand{\bbd}{\mathbf{d}}
\newcommand{\bsigma}{\boldsymbol{\sigma}}
\newcommand{\bbeta}{\boldsymbol{\beta}}

\newcommand{\btau}{\boldsymbol{\tau}}

\newcommand{\aut}{\textnormal{Aut}}

\newcommand{\tZ}{\widetilde{Z}}

\newcommand{\oZ}{\overline{Z}}
\newcommand{\oGamma}{\overline{\Gamma}}
\newcommand{\oXi}{\overline{\Xi}}

\newcommand{\osigma}{{\overline{\sigma}}}
\newcommand{\cov}[2]{\langle #1 \, ; \, #2\rangle}
\newcommand{\iprod}[1]{\langle #1 \rangle}

\newcommand{\don}{\mathds{1}}

\DeclareMathOperator*{\argmin}{arg\,min} 

\makeatletter
\DeclareFontFamily{OMX}{MnSymbolE}{}
\DeclareSymbolFont{MnLargeSymbols}{OMX}{MnSymbolE}{m}{n}
\SetSymbolFont{MnLargeSymbols}{bold}{OMX}{MnSymbolE}{b}{n}
\DeclareFontShape{OMX}{MnSymbolE}{m}{n}{
	<-6>  MnSymbolE5
	<6-7>  MnSymbolE6
	<7-8>  MnSymbolE7
	<8-9>  MnSymbolE8
	<9-10> MnSymbolE9
	<10-12> MnSymbolE10
	<12->   MnSymbolE12
}{}
\DeclareFontShape{OMX}{MnSymbolE}{b}{n}{
	<-6>  MnSymbolE-Bold5
	<6-7>  MnSymbolE-Bold6
	<7-8>  MnSymbolE-Bold7
	<8-9>  MnSymbolE-Bold8
	<9-10> MnSymbolE-Bold9
	<10-12> MnSymbolE-Bold10
	<12->   MnSymbolE-Bold12
}{}

\let\llangle\@undefined
\let\rrangle\@undefined
\DeclareMathDelimiter{\llangle}{\mathopen}%
{MnLargeSymbols}{'164}{MnLargeSymbols}{'164}
\DeclareMathDelimiter{\rrangle}{\mathclose}%
{MnLargeSymbols}{'171}{MnLargeSymbols}{'171}
\makeatother

\newtheorem{theorem}{Theorem}[section]
\newtheorem{lemma}[theorem]{Lemma}

\newtheorem{proposition}[theorem]{Proposition}
\newtheorem{corollary}[theorem]{Corollary}

\newtheorem{thm}{Theorem}[section]

\theoremstyle{definition}

\theoremstyle{definition}
\newtheorem{definition}[theorem]{Definition}

\newtheorem*{remark-non}{Remark}
\numberwithin{equation}{section}

\renewcommand{\thefootnote}{\fnsymbol{footnote}}

\title[Ising model on trees and factor of IID]{Ising model on trees and factors of IID}



\author{Danny Nam}
\address{\newline Department of Mathematics \newline Princeton University \newline Princeton, NJ 08544 \newline \textup{\tt dhnam@math.princeton.edu}
\newline \textup{\tt asly@math.princeton.edu}
 \newline \textup{\tt lingfuz@math.princeton.edu}}

\author{Allan Sly}

\author{Lingfu Zhang}

\pgfplotsset{compat=1.16}
\begin{document}
	\bibliographystyle{acm}
	
	\renewcommand{\thefootnote}{\arabic{footnote}} \setcounter{footnote}{0}

	\begin{abstract}
	We study the ferromagnetic Ising model on the infinite $d$-regular tree under the free boundary condition. This model is known to be  a factor of IID  in the uniqueness regime, when the inverse temperature $\beta\ge 0$ satisfies $\tanh \beta \le (d-1)^{-1}$. However, in the reconstruction regime ($\tanh \beta > (d-1)^{-\frac{1}{2}}$), it is not a factor of IID. We construct a factor of IID for the Ising model beyond the uniqueness regime via a strong solution to an infinite dimensional stochastic differential equation which partially answers a question of Lyons \cite{l17}. The solution $\{X_t(v) \}$ of the SDE is distributed as
	\[
	X_t(v) = t\tau_v + B_t(v),
	\]
	where $\{\tau_v \}$ is an Ising sample and $\{B_t(v) \}$ are independent Brownian motions indexed by the vertices in the tree.
	Our construction holds whenever $\tanh \beta \le c(d-1)^{-\frac{1}{2}}$, where $c>0$ is an absolute constant.
	\end{abstract}
	
	\maketitle

	\section{Introduction}

Let $\TT=(\TT^d,\rho)$ be the infinite $d$-regular tree rooted at a vertex $\rho$. For two measurable spaces $\Omega_0$ and $\Omega_1$, a map $\phi: \Omega_0^{\TT} \to \Omega_1^\TT$ is called a $\TT$\textbf{-factor} if it is measurable in terms of the product $\sigma$-algebra and satisfies
\begin{equation*}
\phi(\eta (\omega)) = \eta(\phi(\omega))  \quad (\eta \in \aut(\TT), \ \omega \in \Omega_0^\TT),
\end{equation*}
where $\aut(\TT)$ denotes the automorphism group of $\TT$, that is, the collection of graph isomorphisms from $\TT$ to itself. We are interested in the notion of \textit{factor of IID} on $\TT$, which is an ergodic property of measures defined as follows. 

\begin{definition}[Factor of IID on $\TT$]\label{def:fiid}
	For a measure $\nu$ on a measurable space $\Omega_1^\TT$, $\nu$ is called a \textbf{factor of IID} if there exists a measurable space $\Omega_0$, a measure $\mu_0$ on $\Omega_0$, and a $\TT$-factor $\phi:\Omega_0^\TT \to \Omega_1^\TT$ such that $\nu$ is the $\phi$-push-forward of the product measure $\mu_0^{\otimes \TT}$.
\end{definition}

Besides the classical works from ergodic theory \cite{o70a,o70b,ow87,b10}, study of factor of IIDs has drawn extensive interests in probability theory, see e.g. \cite{hpps09,hls11,ln11,t11,abg12,as20,rs20,s20a,s20b}. In particular, a factor of IID on $\TT$ can be interpreted as an infinite analogue of local algorithms on the random $d$-regular graph (whose limiting local structure is $\TT$) \cite{cghv15,gs17,rv17}. For a detailed introduction to this concept and related works in probability theory, we refer to \cite{l17} and the references therein.

We study the ferromagnetic Ising model defined on $\TT$ under the free boundary condition, and determine if it is a factor of IID. This model is also called as  \textit{binary symmetric channel on trees}, or  \textit{broadcasting problem on trees}, and has drawn substantial interest not only from probability theory and statistical mechanics, but also from information theory and theoretical computer science (see, e.g., \cite{bcmr06,dm10,ekps00,l89,mms12,m01,mp03}). When the inverse temperature $\beta$ satisfies $\tanh\beta \le (d-1)^{-1}$, the Ising model on $\TT$ has the unique Gibbs measure and is known to be a factor of IID. On the other hand, when $\tanh\beta > (d-1)^{-\frac{1}{2}}$, it is in the reconstruction regime (see Section \ref{subsec:intro:ising} for its definition) and the second author proved that it is not a factor of IID (\cite[Theorem 3.1]{l17}). However, it is not known whether the Ising model on $\TT$ is a factor of IID when $(d-1)^{-1}<\tanh \beta \le (d-1)^{-\frac{1}{2}}$. Our main result proves that it is a factor of IID beyond the uniqueness regime to a certain extent, partially answering a question of Lyons \cite{l17}.

\begin{thm}\label{thm:main}
	There exist absolute constants $c,d_0>0$ such that for all $d\ge d_0$ and $\beta\ge 0$ with $ \tanh\beta \le c(d-1)^{-\frac{1}{2}}$, the Ising model on $\TT^d$ under the free boundary condition at inverse temperature $\beta$ is a factor of IID.
\end{thm}

A common approach to show that a measure is a factor of IID is the \textit{divide-and-color} method {from the study of Random Cluster Model (RCM)} (see \cite{st17} and references therein for details). This is particularly useful to study the Ising/Potts model, {which can be viewed as RCM through their natural connections to FK percolation, and each cluster is assigned an independent color.} If $\beta$ is small (i.e. in the high temperature regime), the percolation clusters are a.s. finite, we can construct a factor of IID by {choosing the color of each cluster in an automorphism invariant way}.  Unfortunately, in our case, this method only works for $\tanh\beta \le (d-1)^{-1}$, as otherwise there exist infinite clusters. Thus proving Theorem \ref{thm:main} requires a different idea.

One question asked by Steif and Tykesson in \cite{st17} is whether there were examples of RCM, such that there are infinite clusters but the coloring process is still a factor of IID.
In \cite{sz19}, the second and third author verified that the stationary measures of the voter model on $\Z^d$ are factors of IID, which is another case where {there are infinite clusters in the corresponding RCM, and} a naive divide-and-color method does not work. Their approach was to utilize the \textit{coalescing random walk}, which is a dual model whose  $t\to\infty$ limit gives the stationary measure of the voter model, and is a factor of IID at all finite times by the divide-and-color method. They constructed a factor of IID at $t\to\infty$ limit by a clever coupling of the finite-time divide-and-color configurations,  where all vertices change their color only finitely many times a.s.~as $t$ goes to infinity. However, a similar approach fails in our case; see Section \ref{subsec:intro:gd} for details.

Our approach is to construct a factor of IID from the independent one-dimensional standard Brownian motions $\bW = \{(W_t(v))_{t\ge 0}\}_{v\in \TT}$. We construct an infinite-dimensional  system of stochastic differential equations whose strong solution recovers the Ising model in the $t\to\infty$ limit. Constructing the desired strong solution  is done using the approximation by the corresponding finite-dimensional system of SDEs and requires technically sophisticated computations. Although our method works for large $d$ and up to the point where $\tanh\beta \le c(d-1)^{-\frac{1}{2}}$, we conjecture that the Ising model on $\TT$ should be a factor of IID for all $d\ge 3$ and $\tanh\beta \le (d-1)^{-\frac{1}{2}}$.

\subsection{Ising model on trees}\label{subsec:intro:ising}
In this subsection, we briefly review the definition of the Ising model and its basic properties. 
On a finite graph $\GG$ with edge set $E(\GG)$, the \textit{Ising model on $\GG$ under the free boundary condition at inverse temperature $\beta$} is a probability measure $\pi_\GG$ on $\{\pm 1 \}^\GG$ defined as follows. For $\bsigma = (\sigma_v)_{v\in \GG}$,
\begin{equation}\label{eq:def:ising:noext}
\pi_\GG(\bsigma) = \frac{1}{Z_\GG} \exp \left(\beta \sum_{(u,v)\in E} \sigma_u\sigma_v \right),
\end{equation}
where $Z_\GG$ is the normalizing constant given by $Z_\GG = \sum_{\bsigma\in \{\pm 1\}^\GG} \exp \left(\beta \sum_{(u,v)\in E} \sigma_u\sigma_v \right)$. We will be interested in the ferromagnetic case where $\beta$ is non-negative. 

For given $d\in \N$, let $(\TT,\rho) = (\TT^d,\rho)$ denote the inifnite $d$-regular tree rooted at $\rho$, and let $\TT_R$ be the depth-$R$ subtree defined by
\begin{equation}\label{eq:def:TR}
\TT_R =\TT_R^d:= \{v\in \TT: \textnormal{dist}(\rho,v) \le R \}.
\end{equation} 
It is well-known that $\pi_{\TT_R}$ converges in product $\sigma$-algebra.
We denote the limit as $\pi_\TT$, which is called the \emph{infinite-volume Gibbs measure on $\TT$ with free boundary condition}, and is the main object of interest in this paper. 
The measure $\pi_\TT$ can also be constructed by a recursive Markovian fashion as follows. Let $\theta:= \tanh \beta$, and consider the $2\times 2$ transition matrix 
\begin{equation*}
P:=
\begin{pmatrix}
\frac{1+\theta}{2} & \frac{1-\theta}{2} \\
\frac{1-\theta}{2} & \frac{1+\theta}{2}
\end{pmatrix}.
\end{equation*}
This defines a Markov chain on the single-spin space $\{-1,+1\}$: for instance, $+1$ becomes $-1$ with probability $\frac{1-\theta}{2}$, and stays at $+1$ with probability $\frac{1+\theta}{2}$. Then, it is well-known that $\pi_\TT$ is equivalent to the law of $\bsigma$ generated by the following recursive scheme:
\begin{itemize}
	\item At the root $\rho$, set $\sigma_\rho=\pm 1$  each with probability $\frac{1}{2}.$
	
	\item For each edge $(u,v) $ in $\TT$, suppose that $\textnormal{dist}(\rho,u)+1= \textnormal{dist}(\rho,v)$ and $\sigma_u$ is determined. Then, $\sigma_v$ is obtained by a single-step Markov chain from $\sigma_u$ with respect to the transition matrix $P$.
\end{itemize}
This construction explains why the model is also called the \textit{broadcasting problem on trees}: a vertex passes its spin correctly to a child with probability $\theta$, and otherwise (i.e., with probability $1-\theta$) it delivers a randomized spin that is either $+1$ or $-1$ with equal probability. This procedure happens independently at each edge.

We now introduce the two thresholds for $\beta$ in the Ising model on $\TT$.
The first is the \textit{uniqueness threshold} $\beta_c:=\tanh^{-1}((d-1)^{-1})$.
When $\beta \le \beta_c$, it is known that the infinite-volume Gibbs measure is unique regardless of the boundary condition;
when $\beta > \beta_c$, the measure $\pi_{\TT}$ is not the unique infinite-volume Gibbs measure on $\TT$.
For detailed introduction on this subject, we refer to \cite{l89} and \cite[ Chapter 17]{mm09}. For its relation to the FK percolation {in the context of RCM}, see \cite[Section 1.3]{st17}.

Besides the uniqueness threshold, the \textit{reconstruction threshold} $\beta_r = \tanh^{-1}((d-1)^{-\frac{1}{2}})$ gives the phase transition for whether the reconstruction problem is solvable. For $\beta\le \beta_r$, 
given by $\bsigma \sim \pi_\TT$,
it is impossible to guess $\sigma_\rho$ better than probability $\frac{1}{2}$ in the $R\to \infty$ limit from just looking at $(\sigma_v)_{v\in \TT\setminus \TT_R}$.
However, we can beat the random guess if $\beta>\beta_r$, the \textit{reconstruction regime}. For a detailed study on this concept, we refer to \cite{ekps00} and references therein. 

As mentioned above, we are interested in the \textit{intermediate regime} of $\beta_c<\beta \le \beta_r$, and this is the only case where it is unknown whether $\pi_\TT$ is a factor of IID.

\subsection{A Glauber dynamics approach}\label{subsec:intro:gd}

The first approach that one may come up with to prove Theorem \ref{thm:main} is using the Glauber dynamics, which is a Markov chain {that converges to some Gibbs measure.}
Since the Glauber dynamics can be encoded by IID information assigned at each vertex, we may hope to construct a factor of IID from this Markov chain starting with an IID process. This approach is indeed possible in the uniqueness regime $\tanh \beta \le (d-1)^{-1}$.
In this subsection, we briefly explain why it does not look promising beyond the uniqueness regime. 

To begin with, we give a short description of the Glauber dynamics. Each vertex $v\in \TT$ is assigned with an IID rate-one Poisson process, which defines the update times at $v$. When an update occurs at $v\in \TT$ at time $t$, $\sigma_{t}(v)$ is updated with respect to the Ising measure on $\{v\}$ conditioned on its neighborhood profile $\{\sigma_{t^-}(u) \}_{u\sim v}$. Namely, it becomes $+1$ with probability
\begin{equation*}
\frac{1}{2}\left(1+ \tanh \left(\beta \sum_{u\sim v} \sigma_{t^-}(u) \right) \right),
\end{equation*}
and transitions to $-1$ otherwise. 

To construct the Glauber dynamics that converges to $\pi_\TT$, one can start from the initial condition given by $\sigma_{t=0}(v) = \pm 1$ with probability $\frac{1}{2}$, independently for each vertex $v$. Assume  that we have two instances of Glauber dynamics $\bsigma_t=(\sigma_t(v))_{v\in \TT}$ and $\bsigma_t'=(\sigma_t'(v))_{v\in \TT}$, and further suppose that at some time $t$ we had $\sigma_t(v)=\sigma_t'(v)$ at all $v\in \TT$ except at a single vertex $v_0\in \TT$. Even under an optimal coupling between $\bsigma_t$ and $\bsigma_t'$, the spin at vertex $v \sim v_0$ will be updated differently in the two instances  with probability
\begin{equation*}
\frac{1}{2} \left|\tanh \left(\beta\sum_{u\sim v} \sigma_t(u) \right) - \tanh \left( \beta\sum_{u\sim v} \sigma_t'(u) \right) \right|,
\end{equation*} 
which is equal to $\tanh\beta$ in the worst case. Then, $\tanh\beta >(d-1)^{-1}$ implies that the \textit{disagreement percolation} is not necessarily subcritical. Thus, the function that maps the IID update information of the Glauber dynamics to the Ising model is not measurable, since it has long-range dependence between distant vertices. This infers that in contrast to \cite{sz19}, we may have to look for a different approach rather than relying on a stochastic process that converges to the Ising model.

\section{Factor construction by independent Brownian motions}
 
 In this section, we give the construction of the factor of IID for the Ising model, and prove Theorem \ref{thm:main}.
 
 \subsection{Notations and setup}
 From now on we fix an arbitrary $d$, assuming it being large enough.
 
 We work on the rooted $d$-regular tree $(\TT, \rho)$.
 For any $\rho' \in \TT$ and $R\in\N$, we denote $\TT_R(\rho')$ as the subgraph induced by
 \begin{equation}\label{eq:def:TRrho}
\{v\in \TT: \textnormal{dist}(\rho',v) \le R \}.
\end{equation}
 We also write $\TT_R:=\TT_R(\rho)$ for ease of notations.
 
 For a finite graph $\GG$ with edge set $E(\GG)$, given the external field $\bx=(x(v))_{v\in\GG}$, and inverse temperature $\bbeta = \{\beta_{u,v}\}_{(u,v)\in E(\GG)}$, let $\pi_\GG^{\bbeta,\bx}$ be the measure on $\{\pm 1\}^\GG$, such that
\begin{equation}\label{eq:def:pi:general}
\pi_\GG^{\bbeta,\bx}(\bsigma) = \frac{1}{Z_\GG^{\bbeta,\bx}} \exp \left( \sum_{(u,v)\in E(\GG)} \beta_{u,v}\sigma_u\sigma_v + \sum_{v\in\GG}x(v)\sigma_v\right),
\end{equation}
for any $\bsigma = (\sigma_v)_{v\in \GG}$,
where $Z_\GG^{\bbeta,\bx} := \sum_{\bsigma\in \{\pm 1\}^\GG} \exp \left(\sum_{(u,v)\in E(\GG)} \beta_{u,v}\sigma_u\sigma_v + \sum_{v\in\GG}x(v)\sigma_v\right)$ is the normalizing constant.
For simplicity of notations, we shall usually write $\pi_\GG^\bx, Z_\GG^\bx$ for $\pi_\GG^{\bbeta,\bx}, Z_\GG^{\bbeta,\bx}$.
We shall also omit $\bx$ when $\bx=\boo$, and omit $\GG$ when it is clear which graph we are working on.

For any measure $\mu$ defined on the probability space $\{\pm 1\}^{\GG}$, and two measurable functions $f, g: \{\pm 1\}^{\GG} \to \mathbb{R}$, we write
\[
\iprod{f}_\mu := \sum_{\bsigma \in \{\pm 1\}^{\GG}} f(\bsigma) \mu( \bsigma),\quad
\cov{f}{g}_\mu := \iprod{fg}_\mu - \iprod{f}_\mu \iprod{g}_\mu.
 \]
For simplicity of notations we also write $\iprod{\cdot}_R^\bx$ for $\iprod{\cdot}_{\TT_R^\bx}$ and $\cov{\cdot}{\cdot}_R^\bx$ for $\cov{\cdot}{\cdot}_{\TT_R^\bx}$.
 We denote $\pi_\TT$ as the free boundary Gibbs measure on $\TT$ (i.e. the weak limit of $\pi_{\TT_R}$ as $R\to\infty$).

\subsection{Construction via finite systems of SDEs}
In this subsection we take constant $\bbeta$, i.e. $\beta_{u,v}=\beta$ for any edge $(u, v)$, and
we give the construction for the factor of IID of $\pi_\TT$.
Take independent one-dimensional Brownian motions $\bW = \{(W_t(v))_{t\ge 0} \}_{v\in \TT}$, which is an i.i.d. process on $\TT$.
We shall construct $\pi_\TT$ as a function of $\bW$, invariant under automorphisms of $\TT$.

 We start from the finite tree $(\TT_R, \rho)$, the depth-$R$ subtree rooted at $\rho$, for $R>0$. 
 Define the function $\bF^R:\R^{\TT_R}\to \R^{\TT_R}$, where for any $v \in \TT_R$ and $\bx\in \R^{\TT_R}$, we let $F_v^R(\bx):=\iprod{\sigma_v}_R^{\bx}$.
 We also let $\bX^R=(\bX^R_t)_{t\ge 0}=\{(X_t^R(v))_{t\ge 0}\}_{v\in\TT_R}$ be the strong solution of the following finite dimensional stochastic differential equation system on $C([0,\infty); \R)^{\TT_R}$:
 \begin{equation}\label{eq:def:Xt:basi c}
 d\bX^R_t = \bF^R(\bX^R_t) dt + d\bW_t^R, \quad \bX_0 = \boo,
 \end{equation}
 where $\bW^R$ denotes the restriction of $\bW$ to $\TT_R$.
 We note that for any $u, v \in\TT_R$, $\partial_{x_u}F_v^R(\bx) =  \cov{\sigma_u}{\sigma_v}_R^{\bx}$, and its absolute value is always bounded by $1$.
 This implies that $F_v^R$ is 1-Lipschitz in each coordinate. Thus, the strong solution of this system of SDEs exists and is unique.
 
 We claim that $\bX^R$ has the same law as the following process. 
 Let $\btau=(\tau_v)_{v\in \TT}$ be sampled from $\pi_\TT$, and let ${\bB} = \{({B}_t(v) )_{ t\ge 0}\}_{v\in \TT}$ be another collection of independent one-dimensional standard Brownian motions. Consider the stochastic process $\overline{\bX}=(\overline{\bX}_t)_{t\ge 0} =\{ (\overline{X}_t(v))_{t\ge 0}\}_{v\in \TT}$ defined as
  \begin{equation}\label{eq:def:Xt:basic:with tau}
  \overline{\bX}_t:= t\btau + \bB_t.
  \end{equation}
 Let $\overline{\bX}^R$ be the restriction of $\overline{\bX}$ on $\TT_R$.
 We have that $\overline{\bX}^R$ has the same law as $\bX^R$, by the following lemma.
 \begin{lemma}
 The process $\overline{\bX}^R$ is a weak solution to \eqref{eq:def:Xt:basi c}.
 \end{lemma}
 \begin{proof}
 For any $t\ge 0$, let $\FF_t$ be the $\sigma$-algebra generated by $(\overline{\bX}^R_s)_{0\le s\le t}$.
 Then $\overline{\bX}^R$ is a weak solution to the following system of SDEs (for details see e.g., \cite[Section 7.4]{liptser01})
 \[
 d\overline{X}^R_t(v) = \E[\tau_v|\FF_t] dt + dW_t(v),\; \forall v\in \TT_R, \quad \overline{\bX}^R_0 = \boo.
 \]
 It remains to compute $\E[\tau_v|\FF_t]$.
 Denote $\btau^R$ as the restriction of $\btau$ on $\TT_R$.
 For any $\FF_t$ measurable set $A$, and 
 $\bsigma = (\sigma_v)_{v\in \TT_R} \in \{\pm 1\}^{\TT_R}$,
\[
\begin{split}
  \P\big( \btau^R = \bsigma,\ (\overline{\bX}_s^R)_{0\le s \le t} \in A \big) 
  =&
  \pi_{\TT_R}(\bsigma) \P \big( (\bB_s + s\bsigma)_{0\le s \le t} \in A \big)\\
  =&
  \pi_{\TT_R}(\bsigma) \frac{\P \big( (\bB_s + s\bsigma)_{0\le s \le t} \in A \big)}{\P \big( (\bB_s)_{0\le s \le t} \in A \big)}\P \big( (\bB_s)_{0\le s \le t} \in A \big).
\end{split}
\]
Thus by Girsanov theorem, we have
\[
\P\big( \btau^R = \bsigma | \FF_t \big)
=  \pi_{\TT_R}(\bsigma)
  \prod_{v\in\TT_R}\exp\left( \overline{X}_t(v)\sigma_v - \frac{t}{2} \right) f((\overline{\bX}^R_s)_{0\le s \le t})
\]
where $f$ is the Radon-Nikodym derivative of the law of $(\bB^R_s)_{0\le s \le t}$ over the law of $(\overline{\bX}^R_s)_{0\le s \le t}$.
Thus we have that conditional on $\FF_t$, the law of $\btau^R$ is given by the Ising model with external field $\overline{\bX}_t^R$;
then for any $v\in\TT_R$ we have $\E[\tau_v|\FF_t] = \iprod{\sigma_v}_R^{\overline{\bX}_t^R} = F_v^R(\overline{\bX}_t^R)$, and the conclusion follows.
 \end{proof}

Given this lemma, our general strategy is to show almost sure convergence of $\bX^R$ as $R\to \infty$;
and the limit is the same if one starts from a different root $\rho'$.
The limit would have the same law as $\overline{\bX}$, from which we could recover the Ising model by taking $\lim_{t\to\infty} \textnormal{sign}(\overline{\bX}_t)$.

To get the desired convergence, our key step bounds the difference between $\bX^R$ and $\bX^{R-1}$, as follows.
\begin{proposition}\label{prop:L2cauchy}
There exist an absolute constant  $c>0$,  a constant $C_d>0$ depending only on $d$, and $\alpha=\alpha(d,\beta)\in(0,1)$, such that when $\beta\ge 0$ and $ \tanh\beta \le c(d-1)^{-\frac{1}{2}} $, we have
\begin{equation}\label{eq:L2cauchy:thm}
  \E \left[ \left(X_t^R(u) - X_t^{R-1}(u) \right)^2 \right] \le C_d e^{C_d t} \alpha^{R-\textnormal{dist}(\rho,u)-3} (d-1)^{\textnormal{dist}(\rho, u)},
\end{equation}
	for all $t,R>0$ and $u\in \TT_{R-1}$.
\end{proposition}
Now we study the equations starting from a different root.
Take any $\rho'\in\TT$.
We define the function $\bF^{R,\rho'}:\R^{\TT_R(\rho')}\to \R^{\TT_R(\rho')}$, where for any $v \in \TT_R(\rho')$ and $\bx\in \R^{\TT_R(\rho')}$, we let $F_v^{R,\rho'}(\bx):=\iprod{\sigma_v}_{\pi_{\TT_R(\rho')}^{\bx}}$.
 We also let $\bX^{R,\rho'}=(\bX^{R,\rho'}_t)_{t\ge 0}=\{(X_t^{R,\rho'}(v))_{t\ge 0}\}_{v\in\TT_R}$ be the strong solution of the following finite dimensional stochastic differential equation system on $C([0,\infty); \R)^{\TT_R(\rho')}$:
 \begin{equation}\label{eq:def:X:differnt root}
 d\bX^{R,\rho'}_t = \bF^{R,\rho'}(\bX^{R,\rho'}_t) dt + d\bW_t^{R,\rho'}, \quad \bX_0 = \boo,
 \end{equation}
 where $\bW^{R,\rho'}$ denotes the restriction of $\bW$ to $\TT_R(\rho')$.

\begin{proposition}\label{prop:L2 cauchy2}
Let $c, C_d, \alpha, \beta$ be as in Proposition \ref{prop:L2cauchy}, and $\rho'$ be a neighbor of the root $\rho$.
For any $t,R>0$ and $u\in \TT_R(\rho) \cap \TT_{R}(\rho')$, we have
	\[
	\mathbb{E} \left[ \left(X_t^{R}(u)-X_t^{R,\rho'}(u) \right)^2 \right]\le  C_d e^{C_dt}  \alpha^{R-\bd(\rho,u)-4} (d-1)^{\bd(\rho,u)}.
	\]
\end{proposition}

We now deduce Theorem \ref{thm:main} from Proposition \ref{prop:L2cauchy} and \ref{prop:L2 cauchy2}.

\begin{proof}[Proof of Theorem \ref{thm:main}]
Let $c$ be the same as in Proposition \ref{prop:L2cauchy} and \ref{prop:L2 cauchy2}.
By Proposition \ref{prop:L2cauchy} we have that for any $t\ge 0$ and $u\in\TT$, 
\[
\E\left[\sum_{R=1}^\infty \left| X_t^R(u) - X_t^{R-1}(u)\right| \right] < \infty.
\]
This means that as $R\to \infty$, $X_t^R(u)$ converges almost surely.
For each $t\in\Q_{\ge 0}$ we define $X_t(u)$ as the limit.
Then the process $(\bX_t)_{t\in\Q_{\ge 0}} = \{(X_t(v))_{t\in\Q_{\ge 0}}\}_{v\in\TT}$ is defined for a.s.~$\bW$, and its finite dimensional distribution is the same as that of $(\overline{\bX}_t)_{t\in\Q_{\ge 0}}$.

For $n\in\N$ we define $\bG_n=\{G_n(u)\}_{u\in\TT}$ as
$G_n(u) := \textnormal{sign}(X_{2^n}(u))$.
Then as $n\to\infty$, $\bG_n$ converges in law to $\pi_\TT$.
We now show that $\bG_n$ converges almost surely.
Indeed, we have
\[
\begin{split}
\P\left( G_n(u) \neq G_{n+1}(u) \right)
= &
\P\left( \textnormal{sign}(\overline{X}_{2^n}(u)) \neq \textnormal{sign}(\overline{X}_{2^{n+1}}(u)) \right)\\
\le &
\P\left(|B_{2^n}(u)| > 2^{n-1}\right)
+
\P\left(|B_{2^{n+1}}(u)- B_{2^n}(u)| > 2^{n-1}\right).
\end{split}
\]
Thus, we can see that  $\P(G_n(u) \neq G_{n+1}(u))$ is summable in $n$, which by Borel-Cantelli lemma implies that $G_n(u)$ converges a.s., as $n\to\infty$.
We denote the limit by $\bG \in \R^{\TT}$, which is a measurable function of $\bW$, defined for a.s. $\bW$, and the law of $\bG$ is $\pi_\TT$.

Finally we show that this function $\bW \mapsto \bG$ is invariant under $\aut(\TT)$.
From our construction of $\bG$ it suffices to prove the following:
take any $\rho' \in \TT$, then as $R\to\infty$, $X_t^{R,\rho'}(u)$ a.s. converges to $X_t(u)$, for any $t\in \Q_{\ge 0}$ and $u\in \TT$.
Indeed, $X_t^{R,\rho'}(u)$ a.s. converges, since the function $\bW \mapsto X_t^{R,\rho'}(u)$ is the push-forward of  $\bW \mapsto X_t^{R}(u)$ under an action in $\aut(\TT)$.
Denote this limit as $X_t^{\rho'}(u)$. Let $\rho_0, \rho_1, \cdots, \rho_k$ be the path from $\rho_0 = \rho$ and $\rho_k = \rho'$. Applying Proposition \ref{prop:L2 cauchy2} to the consecutive pairs $(\rho_i,\rho_{i+1})$  and sending $R\to\infty$, we have $\E\left[\left( X_t(u) - X_t^{\rho'}(u) \right)^2\right] = 0$;
then $X_t(u) = X_t^{\rho'}(u)$ almost surely, and our conclusion follows.
\end{proof}

 \subsection{An interpolation approach}\label{subsec:formulation:induction on R}
 In this subsection, we develop a framework to analyze the difference between $\bX^{R-1}$ and $\bX^{R}$, the unique strong solutions of \eqref{eq:def:Xt:basi c}. As a result, we reduce Propositions \ref{prop:L2cauchy} and \ref{prop:L2 cauchy2} into more tractable forms. From now on, denote $E_R:=E(\TT_R)$, and let $\partial E_R$ denote the boundary edges of $\TT_R$, i.e.
 \[
 \partial E_R : = E_R\setminus E_{R-1}= \{(u,v)\in E(\TT_R): v\in \partial \TT_R  \},
 \]
 where $\partial \TT_R := \{v\in \TT_R : \textnormal{dist}(\rho,v)=R \}$.

The Ising model on $\TT_{R-1}$ can be viewed as the Ising model on $\TT_{R}$, by setting $\beta_{u,v}=0$ for $(u,v)\in\partial E_R$. Thus, our approach to compare $\bX^{R-1}$ and $\bX^{R}$ is by interpolating the inverse temperature:  we investigate the Ising model on $\TT_R$ with inverse temperature $\bbeta^\gamma:=\{\beta^\gamma_{u,v}\}_{(u,v)\in E_R}$, where $\beta^\gamma_{u,v}=\beta$ for $(u,v)\in E_{R-1}$ and $\beta^\gamma_{u,v}=\gamma$ for $(u,v)\in \partial E_R$, for some $\gamma \in[0,\beta]$.
 Formally, for $\bx = (x(v))_{v\in \TT_R} \in \R^{\TT_R}$, we define the probability measure ${\pi}^{\gamma,\bx}:={\pi}^{\bbeta^\gamma,\bx}_{\TT_R}$.
 Define the function $\bF^{\gamma}:\R^{\TT_R}\to \R^{\TT_R}$, where for any $v \in \TT_R$ and $\bx\in \R^{\TT_R}$, we let $F_v^\gamma(\bx):=\iprod{\sigma_v}_{\pi^{\gamma,\bx}}$.
 We also let $\bY^\gamma=(\bY^\gamma_t)_{t\ge 0}=\{(Y_t^\gamma(v))_{t\ge 0}\}_{v\in\TT_R}$ be the strong solution of the following finite dimensional stochastic differential equation system on $C([0,\infty); \R)^{\TT_R}$:
 \begin{equation}\label{eq:def:tildeXt:SDE}
 d\bY^\gamma_t = \bF^\gamma(\bY^\gamma_t) dt + d\bW_t^R, \quad \bY^\gamma_0 = \boo,
 \end{equation}
 where as before $\bW^R$ denotes the restriction of $\bW$ to $\TT_R$, and are the driving Brownian motions.
In particular, we have $\bY^0 = \bX^{R-1}$ on $\TT_{R-1}$, and $\bY^\beta = \bX^R $. 
As for $\bX^R$, the law of $\bY^\gamma=(\bY^\gamma_t)_{t\ge 0}$ is the same as that of $(t\btau + \bB_t)_{t\ge 0}$, where $\btau=(\tau_v)_{v\in \TT_R}$ is sampled from $\pi^{\gamma,\boo}$, and ${\bB} = \{({B}_v(t) )_{ t\ge 0}\}_{v\in \TT_R}$ is another collection of independent one-dimensional standard Brownian motions.

Note that $\bY^\gamma$ for all $\gamma\in[0,\beta]$ are generated by the same driving Brownian motions $\bW^R_t$, thus if we define $\bH^\gamma= (\bH^\gamma_t)_{t\ge 0} = \{(H^\gamma_t(v))_{t\ge 0} \}_{v\in \TT_R}$ to be
\begin{equation}\label{eq:def:Hgamma}
\bH^\gamma_t := \partial_\gamma \bY^\gamma_t,
\end{equation}
then from \eqref{eq:def:tildeXt:SDE}, we can deduce that 
 	\begin{equation}\label{eq:def:Yt}
 \frac{d \bH_t^\gamma}{dt}	= \partial_\gamma  \left\{ \bF^\gamma(\bY^\gamma_t)  \right\} =
 \nabla \bF^\gamma(\bY^\gamma_t) \bH_t^\gamma + \partial_\gamma \bF^\gamma \left(\bY_t^\gamma \right) ,
 	\end{equation}
and $\bH_0^\gamma =\boo$. We have the following estimate on $\bH^\gamma$.
	\begin{proposition}\label{prop:Ybd L2}
		 Under the above setting, there exist absolute constants $c>0$, and $C_d>0$ depending only on $d$, and $\alpha=\alpha(d,\beta)\in(0,1)$, such that if $d>c^{-1}$, $\tanh\beta \le c(d-1)^{-\frac{1}{2}}$, and $0\le \gamma \le \beta$, we have
		 \begin{equation*} \E\left[H_t^\gamma(u)^2 \right] \le C_de^{C_dt} \alpha^{R-\textnormal{dist}(\rho,u)-3}(d-1)^{\bd(\rho,u)}
		 \end{equation*}
		 for all $t>0$, $u\in \TT_R$. 
	\end{proposition} 
	
	It is straight-forward to deduce Proposition \ref{prop:L2cauchy} from Proposition \ref{prop:Ybd L2}.
	\begin{proof}[Proof of Proposition \ref{prop:L2cauchy}]
From the construction we have
		\begin{equation}\label{eq:XRdiff:integralfrom}
		X_t^R(u) - X_t^{R-1}(u) = \int_0^\beta H_t^\gamma(u) d\gamma,
		\end{equation}
		and thus, the Cauchy-Schwarz inequality gives
		\[
		 \E \left[ \left(X_t^R(u)-X_t^{R-1}(u) \right)^2 \right] \le \beta^2  \sup_{\gamma \in[0,\beta]} \E\left[H_t^\gamma(u)^2 \right] \le C_d e^{C_dt} \alpha^{R-\textnormal{dist}(\rho,u)-3}(d-1)^{\bd(\rho,u)},
		\]
		concluding the proof. 
	\end{proof}

	Now we move on to the case of Proposition \ref{prop:L2 cauchy2}, where we use a similar argument. Let $\rho'$ be a neighbor of $\rho$, and let $\TT^\sharp_R:= \TT_R(\rho) \cup \TT_R(\rho')$. Here, we use the inverse temperature $\bbeta^{\gamma,\rho'} = \{\beta^{\gamma, \rho'}_{u,v} \}_{(u,v)\in E(\TT_R^\sharp)} $ defined as $\beta^{\gamma, \rho'}_{u,v} = \beta$ for $(u,v)\in E_R$, and $\beta^{\gamma,\rho'}_{u,v}=\gamma$ for $(u,v)\in E(\TT_R^\sharp)\setminus E_R$. Moreover, define the probability measure $\pi^{\gamma,\rho',\bx} := \pi_{\TT^\sharp_R}^{\bbeta^{\gamma,\rho'}, \bx}$.  The function $\bF^{\gamma,\rho'} : \mathbb{R}^{\TT^\sharp_R} \to \mathbb{R}^{\TT^\sharp_R}$ is defined to be $F_v^{\gamma,\rho'}(\bx):= \iprod{\sigma_v}_{\pi^{\gamma,\rho',\bx}}$. Then, as before, we consider the strong solution $\bY^{\gamma,\rho'} = (\bY^{\gamma,\rho'}_t)_{t\ge 0} = \{(Y_t^{\gamma,\rho'}(v))_{t\ge 0} \}_{v\in \TT^\sharp_R}$ of the following system of stochastic differential equations:
	\begin{equation*}
	    d\bY^{\gamma,\rho'}_t = \bF^{\gamma,\rho'}(\bY^{\gamma,\rho'}_t) dt + d\bW_t^\sharp, \quad \bY_0^{\gamma, \rho'} = \boo,
	\end{equation*}
	where $\bW^\sharp$ denotes the restriction of $\bW$ to $\TT^\sharp_R$. This setup gives $Y^{\gamma,\rho'} = X^R$ for $\gamma = 0$, and if we switch the roles of $\rho $ and $\rho'$ from the beginning then it will correspond to $X^{R,\rho'}$.
	Then, as before, its $\gamma$-derivative $\bH^{\gamma,\rho'}_t:= \partial_\gamma \bY_t^{\gamma,\rho'}$ satisfies
	\begin{equation*}
	    \frac{d \bH_t^{\gamma,\rho'}}{dt} = \partial_\gamma \{ \bF^{\gamma,\rho'}(\bY_t^{\gamma,\rho'}) \} = \nabla \bF^{\gamma,\rho'}(\bY^{\gamma,\rho'}_t) \bH_t^{\gamma,\rho'} + \partial_\gamma \bF^{\gamma,\rho'}(\bY^{\gamma,\rho'})_t,
	\end{equation*}
	and $\bH_0^{\gamma,\rho'} = \boo$. We can control $\bH_t^{\gamma,\rho'}$ as follows.

	\begin{corollary}\label{cor:Ybd L2:different center}
		Under the setting of Proposition \ref{prop:Ybd L2} and the above notations, we have
		\begin{equation*}
	  \E \left[ H_t^{\gamma,\rho'}(u)^2 \right] \le C_d e^{C_dt} \alpha^{R-\textnormal{dist}(\rho,u)-3} (d-1)^{\textnormal{dist}(\rho,u)},
		\end{equation*}
		for all $t>0$ and $u \in \TT_R^\sharp$. 
	\end{corollary}
	
	We address the proof of Proposition \ref{prop:L2 cauchy2} which comes as a direct consequence.
	\begin{proof}[Proof of Proposition \ref{prop:L2 cauchy2}]
	To bound $\E \left[ \left(X_t^{R}(u)-X_t^{R,\rho'}(u) \right)^2 \right]$, we can just bound the expectations $\E \left[ \left(X_t^{R} (u)- Y_t^{\beta,\rho'}(u) \right)^2 \right]$ and $\E \left[\left( X_t^{R,\rho'}(u) - Y_t^{\beta,\rho'}(u) \right)^2 \right]$ respectively.
	
	For the first one, we can write the difference $X_t^R(u) - Y_t^{\beta,\rho'}(u)=Y_t^{0,\rho'}(u) - Y_t^{\beta, \rho'}(u)$ in an integral form as \eqref{eq:XRdiff:integralfrom}, and hence Corollary \ref{cor:Ybd L2:different center} implies
		\begin{equation*}
	\E \left[ \left(X_t^{R} (u)- Y_t^{\beta,\rho'}(u) \right)^2 \right] \le C_d e^{C_dt}\alpha^{R-\textnormal{dist}(\rho,u)-3} (d-1)^{\textnormal{dist}(\rho,u)}.
		\end{equation*}
		By the same arguments and switching the roles of $\rho$ and $\rho'$, we can get the same bound for $\E \left[\left( X_t^{R,\rho'}(u) - Y_t^{\beta,\rho'}(u) \right)^2 \right]$. We conclude the proof by combining the two bounds together, and noting that $|\bd(\rho,u)-\bd(\rho',u)|\le 1$.
	\end{proof}

\subsection{Reduction to a chain of covariances}

Recall the definitions of $\bY^\gamma = \{Y^\gamma_t(u) \}$ given by \eqref{eq:def:tildeXt:SDE} and  $\bH^\gamma = \{ H^\gamma_t(u)\}$ given by \eqref{eq:def:Hgamma}.   From now on, let $\gamma \in [0,\beta]$ be a fixed number and write $\bY=\bY^\gamma$, $\bH = \bH^\gamma$, i.e., drop the superscript $\gamma$ from their expressions.
We shall also just write $\pi^{\bx}$ for $\pi^{\gamma,\bx}$ for any $\bx \in \R^{\TT_R}$.
In this subsection, we give an explicit formula of $\bH$ written by a chain of covariances, and state an estimate that controls its second moment, then prove Proposition \ref{prop:Ybd L2}. For the simplicity of exposition, we work with $\bH$, and it will be clear from the discussion that the methods we describe below can be used to investigate $\bH^{\gamma,\rho'}$ in the same way.

By straightforward computations, we can write \eqref{eq:def:Yt} as
$$
\frac{d\bH_t}{dt} 
=
\MM_t\bH_t + \NN_t
,$$
where $\MM_t$ (resp.~$\NN_t$) is a $\TT_R \times \TT_R$ matrix (resp.~$\TT_R$-vector) given by
\[
\MM_t(u,v)= \cov{\sigma_u}{\sigma_v}_{\pi^{\bY_t}}, \quad \NN_t(v)=\sum_{(u,u')\in \partial E_R}\cov{\sigma_u\sigma_{u'}}{\sigma_v}_{\pi^{\bY_t}}.
\]
Since each entry of $\MM_t$ and $\NN_t$ is bounded by $1$ and $|\partial E_R|$, respectively, we can write
\[
\bH_t = \sum_{k=1}^\infty
\int_{0<t_1<\cdots<t_k<t}
\MM_{t_k}\cdots \MM_{t_2}\NN_{t_1}d\bt,
\]
where $d\bt$ denotes $dt_1\ldots dt_k$.

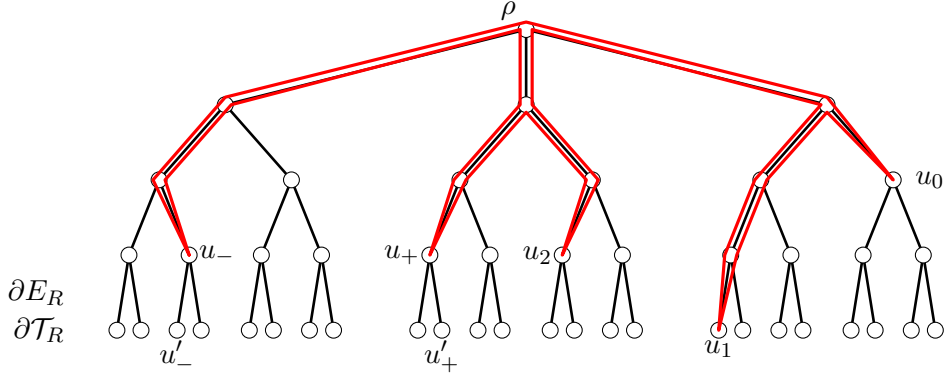
\begin{figure}
    \centering
\begin{tikzpicture}[line cap=round,line join=round,>=triangle 45,x=8cm,y=10cm]
\clip(-0.15,0.53) rectangle (1.51,1.05);

\draw [line width=1pt] (0.75,1.) -- (0.25,0.9);
\draw [line width=1pt] (0.75,1.) -- (0.75,0.9);
\draw [line width=1pt] (0.75,1.) -- (1.25,0.9);

\draw [line width=1pt] (0.14,0.8) -- (0.25,0.9);
\draw [line width=1pt] (0.36,0.8) -- (0.25,0.9);
\draw [line width=1pt] (0.64,0.8) -- (0.75,0.9);
\draw [line width=1pt] (0.86,0.8) -- (0.75,0.9);
\draw [line width=1pt] (1.14,0.8) -- (1.25,0.9);
\draw [line width=1pt] (1.36,0.8) -- (1.25,0.9);

\draw [line width=1pt] (0.14,0.8) -- (0.09,0.7);
\draw [line width=1pt] (0.36,0.8) -- (0.31,0.7);
\draw [line width=1pt] (0.64,0.8) -- (0.59,0.7);
\draw [line width=1pt] (0.86,0.8) -- (0.81,0.7);
\draw [line width=1pt] (1.14,0.8) -- (1.09,0.7);
\draw [line width=1pt] (1.36,0.8) -- (1.31,0.7);
\draw [line width=1pt] (0.14,0.8) -- (0.19,0.7);
\draw [line width=1pt] (0.36,0.8) -- (0.41,0.7);
\draw [line width=1pt] (0.64,0.8) -- (0.69,0.7);
\draw [line width=1pt] (0.86,0.8) -- (0.91,0.7);
\draw [line width=1pt] (1.14,0.8) -- (1.19,0.7);
\draw [line width=1pt] (1.36,0.8) -- (1.41,0.7);

\draw [line width=1pt] (0.07,0.6) -- (0.09,0.7);
\draw [line width=1pt] (0.29,0.6) -- (0.31,0.7);
\draw [line width=1pt] (1.07,0.6) -- (1.09,0.7);
\draw [line width=1pt] (1.29,0.6) -- (1.31,0.7);
\draw [line width=1pt] (0.57,0.6) -- (0.59,0.7);
\draw [line width=1pt] (0.79,0.6) -- (0.81,0.7);
\draw [line width=1pt] (0.17,0.6) -- (0.19,0.7);
\draw [line width=1pt] (0.39,0.6) -- (0.41,0.7);
\draw [line width=1pt] (1.17,0.6) -- (1.19,0.7);
\draw [line width=1pt] (1.39,0.6) -- (1.41,0.7);
\draw [line width=1pt] (0.67,0.6) -- (0.69,0.7);
\draw [line width=1pt] (0.89,0.6) -- (0.91,0.7);
\draw [line width=1pt] (0.11,0.6) -- (0.09,0.7);
\draw [line width=1pt] (0.33,0.6) -- (0.31,0.7);
\draw [line width=1pt] (1.11,0.6) -- (1.09,0.7);
\draw [line width=1pt] (1.33,0.6) -- (1.31,0.7);
\draw [line width=1pt] (0.61,0.6) -- (0.59,0.7);
\draw [line width=1pt] (0.83,0.6) -- (0.81,0.7);
\draw [line width=1pt] (0.21,0.6) -- (0.19,0.7);
\draw [line width=1pt] (0.43,0.6) -- (0.41,0.7);
\draw [line width=1pt] (1.21,0.6) -- (1.19,0.7);
\draw [line width=1pt] (1.43,0.6) -- (1.41,0.7);
\draw [line width=1pt] (0.71,0.6) -- (0.69,0.7);
\draw [line width=1pt] (0.93,0.6) -- (0.91,0.7);

\draw [fill=white] (0.75,1.) circle (3.0pt);
\draw [fill=white] (0.25,0.9) circle (3.0pt);
\draw [fill=white] (0.75,0.9) circle (3.0pt);
\draw [fill=white] (1.25,0.9) circle (3.0pt);
\draw [fill=white] (0.14,0.8) circle (3.0pt);
\draw [fill=white] (0.36,0.8) circle (3.0pt);
\draw [fill=white] (1.14,0.8) circle (3.0pt);
\draw [fill=white] (1.36,0.8) circle (3.0pt);
\draw [fill=white] (0.64,0.8) circle (3.0pt);
\draw [fill=white] (0.86,0.8) circle (3.0pt);
\draw [fill=white] (0.09,0.7) circle (3.0pt);
\draw [fill=white] (0.19,0.7) circle (3.0pt);
\draw [fill=white] (0.31,0.7) circle (3.0pt);
\draw [fill=white] (0.41,0.7) circle (3.0pt);
\draw [fill=white] (1.09,0.7) circle (3.0pt);
\draw [fill=white] (1.19,0.7) circle (3.0pt);
\draw [fill=white] (1.31,0.7) circle (3.0pt);
\draw [fill=white] (1.41,0.7) circle (3.0pt);
\draw [fill=white] (0.59,0.7) circle (3.0pt);
\draw [fill=white] (0.69,0.7) circle (3.0pt);
\draw [fill=white] (0.81,0.7) circle (3.0pt);
\draw [fill=white] (0.91,0.7) circle (3.0pt);
\draw [fill=white] (0.07,0.6) circle (3.0pt);
\draw [fill=white] (0.11,0.6) circle (3.0pt);
\draw [fill=white] (0.17,0.6) circle (3.0pt);
\draw [fill=white] (0.21,0.6) circle (3.0pt);
\draw [fill=white] (0.29,0.6) circle (3.0pt);
\draw [fill=white] (0.33,0.6) circle (3.0pt);
\draw [fill=white] (0.39,0.6) circle (3.0pt);
\draw [fill=white] (0.43,0.6) circle (3.0pt);
\draw [fill=white] (0.57,0.6) circle (3.0pt);
\draw [fill=white] (0.61,0.6) circle (3.0pt);
\draw [fill=white] (0.67,0.6) circle (3.0pt);
\draw [fill=white] (0.71,0.6) circle (3.0pt);
\draw [fill=white] (0.79,0.6) circle (3.0pt);
\draw [fill=white] (0.83,0.6) circle (3.0pt);
\draw [fill=white] (0.89,0.6) circle (3.0pt);
\draw [fill=white] (0.93,0.6) circle (3.0pt);
\draw [fill=white] (1.07,0.6) circle (3.0pt);
\draw [fill=white] (1.11,0.6) circle (3.0pt);
\draw [fill=white] (1.17,0.6) circle (3.0pt);
\draw [fill=white] (1.21,0.6) circle (3.0pt);
\draw [fill=white] (1.29,0.6) circle (3.0pt);
\draw [fill=white] (1.33,0.6) circle (3.0pt);
\draw [fill=white] (1.39,0.6) circle (3.0pt);
\draw [fill=white] (1.43,0.6) circle (3.0pt);

\draw (0.75,1) node[anchor=south east]{$\rho$};

\draw (0,0.6) node[anchor=east]{$\partial \TT_R$};
\draw (0,0.65) node[anchor=east]{$\partial E_R$};

\draw (0.17,0.6) node[anchor=north]{$u_-'$};
\draw (0.19,0.7) node[anchor=west]{$u_-$};
\draw (0.59,0.7) node[anchor=east]{$u_+$};
\draw (0.61,0.6) node[anchor=north]{$u_+'$};
\draw (1.38,0.8) node[anchor=west]{$u_0$};
\draw (1.07,0.6) node[anchor=north]{$u_1$};
\draw (0.81,0.7) node[anchor=east]{$u_2$};

\draw[very thick] [red] plot coordinates {(0.19,0.7) (0.15,0.8) (0.26,0.9) (0.74,1) (0.74,0.9) (0.63,0.8) (0.59,0.7) (0.65,0.8) (0.75,0.89) (0.85,0.8) (0.81,0.7) (0.87,0.8) (0.76,0.9) (0.76,1) (1.24,0.9) (1.13,0.8) (1.08,0.7) (1.07,0.6) (1.10,0.7) (1.15,0.8) (1.25,0.89) (1.36,0.8) (1.25,0.91) (0.75,1.01) (0.25,0.91) (0.13,0.8) (0.19,0.7)};

\end{tikzpicture}
\caption{
An illustration of a path visiting $u_-, u_0, \ldots, u_k, u_+$ in a $d$-regular tree
}  
\label{fig:path}
\end{figure}

We use the Cauchy-Schwarz inequality to obtain that
\[
\E [H_t(u)^2] \le \left(\sum_{k=1}^\infty \frac{1}{2^k} \right) \cdot \sum_{k=1}^\infty 2^k \E \left[\left( \int_{0<t_1< \cdots < t_k<t } [\MM_{t_k} \cdots \MM_{t_2} \NN_{t_1} ](u) d\bt \right)^2 \right].
\]
Moreover, applying the Cauchy-Schwarz inequality given by  $(\int_A 1 dt)(\int_A f(t)^2 dt) \ge (\int_A f(t)dt)^2 $ to the integral in the RHS, we have for any $u\in\TT_R$ that
\begin{equation}  \label{eq:bound-H2-sum-int}
\E[H_t(u)^2 ] \le \sum_{k=1}^\infty
2^k \frac{t^k}{k!}\int_{0<t_1<\cdots<t_k<t}\E\left[\left(\sum_{v_1,\cdots,v_k=u}\NN_{t_1}(v_1)\prod_{i=2}^k \MM_{t_i}(v_i,v_{i-1}) \right)^2\right] d\bt.    
\end{equation}
To study this formula and establish Proposition \ref{prop:Ybd L2}, the following bound is crucial, and its proof is presented in the next section. For a sequence of vertices $v_0, \ldots, v_k$, we write for simplicity that
\begin{equation*}
    \textnormal{dist}(v_{0:k}) := \sum_{i=1}^k \textnormal{dist}(v_i,v_{i-1}).
\end{equation*}
Also recall that we denote $\theta=\tanh\beta$.
\begin{proposition}   \label{prop:chain-expect}
There exists an absolute constant $C>0$ such that the following holds when $d$ is large enough.
For any $k\ge 0$, $u_0, \cdots, u_k \in \TT_R$, and $(u_-, u_-'), (u_+, u_+') \in \partial E_R$, $u_-', u_+' \in \partial \TT_R$, 
and $0\le t_-, t_+, t_1,\cdots, t_k <t$,
we have
\begin{equation}  \label{eq:chain-cov}
\begin{split}
\E\left[
\cov{\sigma_{u_-}\sigma_{u_-'}}{\sigma_{u_0}}_{\pi^{\bY_{t_-}}}
\cov{\sigma_{u_+}\sigma_{u_+'}}{\sigma_{u_k}}_{\pi^{\bY_{t_+}}}
\prod_{i=1}^k\cov{\sigma_{u_i}}{\sigma_{u_{i-1}}}_{\pi^{\bY_{t_i}}
}
\right]\\
\le  C^k (C\theta)^{\bd(u_-,u_+)+\bd(u_-,u_0)+\bd(u_+,u_k)+\bd(u_{0:k})-2}e^{25t+d}. \end{split}
\end{equation}
\end{proposition}
For the exponent of $\bd(u_-,u_+)+\bd(u_-,u_0)+\bd(u_+,u_k)+\bd(u_{0:k})$, it can be understood as the length of the shortest path, which starts from $u_-$ and visits $u_0,\ldots, u_k$ sequentially, then goes to $u_+$ and back to $u_-$ (see Figure \ref{fig:path}). 
We also note that without the first term $\bd(u_-,u_+)$, such bound would be easy to prove (e.g. it directly follows from Lemma \ref{L:cov-formula} below). However, the term  $\bd(u_-,u_+)$ is crucial to get Proposition \ref{prop:Ybd L2}.

For $C$ from Proposition \ref{eq:chain-cov}, let $\eta=(C\theta)^{1/2}(d-1)^{-1/4}$. Then $C\theta < \eta < 1/\sqrt{d-1}$. Also let $S=(1-(d-1)C\theta\eta)^{-1}(1-C\theta/\eta)^{-1}$, and assume that $C\theta < 1/\sqrt{d-1}$.
Then we have
\begin{equation}  \label{eq:bd-E-prod-NM}
\begin{split}
&
\E\left[\left(\sum_{v_1,\cdots,v_k=u}\NN_{t_1}(v_1)\prod_{i=2}^k \MM_{t_i}(v_i,v_{i-1}) \right)^2\right]\\
=&
\E\left[\sum_{\substack{v_1,\cdots,v_k=u\\ w_1,\cdots,w_k=u}}\NN_{t_1}(v_1)\NN_{t_1}(w_1)\prod_{i=2}^k \MM_{t_i}(v_i,v_{i-1})\MM_{t_i}(w_i,w_{i-1}) \right]\\
\le &
e^{25t+d}C^{2k}\sum_{\substack{(u_-, u_-') \in \partial E_R, \\ (u_+, u_+') \in \partial E_R}}\sum_{\substack{v_1,\cdots,v_k=u\\ w_1,\cdots,w_k=u}}
(C\theta)^{\bd(u_-,u_+)+\bd(u_-,v_1)+\bd(u_+,w_1)+\bd(v_{1:k})+\bd(w_{1:k})-2}
\\
<&
e^{25t+d}C^{2k}\theta^{-2}\sum_{\substack{(u_-, u_-') \in \partial E_R, \\ (u_+, u_+') \in \partial E_R}}
S^{2k+2}\eta^{\bd(u_-,u_+)+\bd(u_-,u)+\bd(u_+,u)}
\\
\le&
e^{25t+d}C^{2k}\theta^{-2}S^{2k+2}\eta^{2(R-1-\bd(\rho,u))}\sum_{\substack{(u_-, u_-') \in \partial E_R, \\ (u_+, u_+') \in \partial E_R}}
\eta^{\bd(u_-,u_+)}
\\
\le&
e^{25t+d}C^{2k}\theta^{-2}S^{2k+2}\eta^{2(R-1-\bd(\rho,u))} d^2(d-1)^{R-1} \sum_{i=0}^{R-1} \eta^{2i}(d-1)^i
\\
<&
e^{25t+d}C^{2k}\theta^{-2}S^{2k+2}\eta^{2(R-1-\bd(\rho,u))} d^2(d-1)^{R-1} (1-\eta^2(d-1))^{-1}.
\end{split}
\end{equation}
Here the first inequality is by Proposition \ref{prop:chain-expect}, the third equality is by $\bd(u_-,u), \bd_{u_+,u}\ge R-1-\bd(\rho,u)$, and the fourth and fifth inequalities are by direct computations. The second inequality is by the following lemma.
\begin{lemma}
For any $k\ge 2$ and $v_1, v_k \in \TT$, we have
\[
\sum_{v_2,\cdots, v_{k-1}\in\TT} (C\theta)^{\bd(v_{1:k})} < S^k\eta^{\bd(v_1,v_k)}.
\]
\end{lemma}
\begin{proof}
First, by symmetry the LHS depends only on $k$ and $N:=\bd(v_1,v_k)$.
Denote the LHS by $A_{k,N}$, and we prove $A_{k,N}<S^k\eta^N$ by induction in $k$.
For $k=2$, we have $A_{2,N}=(C\theta)^N < \eta^N$.
Now suppose that $A_{k,N}<S^k\eta^N$ for some $k\ge 2$ and any $N$.
For $v_1, \cdots, v_{k+1}\in\TT$, let $v'$ be the (only) vertex with the smallest $\bd(v_1,v')+\bd(v_2,v')+\bd(v_{k+1},v')$. 
Let $m=\bd(v_1,v')$ and $m'=\bd(v',v_2)$, then $\bd(v_1,v_2)=m+m'$ and $\bd(v_2,v_{k+1})=\bd(v_1,v_{k+1})-m+m'$. 
Thus by the induction hypothesis we have
\[
A_{k+1,N} \le \sum_{m=0}^N \sum_{m'=0}^{\infty} (d-1)^{m'}(C\theta)^{m+m'}A_{k,N-m+m'}
< \sum_{m=0}^N \sum_{m'=0}^{\infty} (d-1)^{m'}(C\theta)^{m+m'}S^k\eta^{N-m+m'}.
\]
Since $C\theta < \eta < 1/\sqrt{d-1}$, by first summing over $m'$ then over $m$, we can bound this by
\[
S^k\eta^N (1-(d-1)C\theta\eta)^{-1}(1-C\theta/\eta)^{-1} = S^{k+1}\eta^N
\]
where the equality holds by the definition of $S$.
\end{proof}

\begin{proof}[Proof of Proposition \ref{prop:Ybd L2}]
By \eqref{eq:bound-H2-sum-int} and \eqref{eq:bd-E-prod-NM}, we conclude that when $d$ is large enough,
\[
\begin{split}
&\E[H_t^\gamma(u)^2 ] \\ &\le e^{25t+d}\eta^4\theta^{-2}S^2
d^2(d-1)^{\bd(\rho,u)+2}
(\eta^2(d-1))^{R-\bd(\rho,u)-3}(1-\eta^2(d-1))^{-1}
\sum_{k=1}^\infty
2^k \left(\frac{t^k}{k!}\right)^2
C^{2k}S^{2k}\\
&<
e^{25t+d}\eta^4\theta^{-2}S^2
d^2(d-1)^{\bd(\rho,u)+2}
(\eta^2(d-1))^{R-\bd(\rho,u)-3}(1-\eta^2(d-1))^{-1}e^{4tCS}
\end{split}
\] 
Now we take $c$ small enough to ensure that $d$ is large and $c<C^{-2}$. Recall that $\theta=\tanh\beta \le c(d-1)^{-\frac{1}{2}}$ and $\eta=(C\theta)^{1/2}(d-1)^{-1/4}$, so we have $C^2\theta<(d-1)^{-1/2}$ and $S<(1-C^{-1})^{-1}(1-C^{-1/2})^{-1}$, thus $S$ has a universal upper bound; and we also have $(1-\eta^2(d-1))^{-1} < (1-C^{-1})^{-1}$, and $\eta^2\theta^{-1}=C(d-1)^{-1/2}$.
By taking
$\alpha=\eta^2(d-1)$, and \[C_d \ge (4CS+25)\vee e^{d}\eta^4\theta^{-2}S^2
d^2(d-1)^2(1-\eta^2(d-1))^{-1},\]
depending only on $d$,
the conclusion of Proposition \ref{prop:Ybd L2} follows. 
\end{proof}

\section{Inductive coupling for a chain of covariances}\label{sec:inductivecoup}

This section is devoted to the proof of Proposition \ref{prop:chain-expect}. We begin by introducing the notion of induced external field and setting up the necessary notations in Section \ref{subsec:ind-coup:prelim}. Then, in Section \ref{subsec:ind-coup:covchain}, we reformulate the chain of covariances in Proposition \ref{prop:chain-expect} into a more tractable form, consisting of products of partition functions. Furthermore, in Section \ref{subsec:ind-coup:ind-coup} we discuss the method of inductive coupling to investigate these partition functions, and derive necessary estimates needed for the coupling argument in Section \ref{subsec:ind-coup:estims}, the final subsection.

\subsection{Preliminaries}\label{subsec:ind-coup:prelim}
In this section we work under the setting of Proposition \ref{prop:chain-expect}.
Specifically, we work on $\TT_R$ for fixed $R$, and let $\bbeta \in \R^{E_R}$ such that $\beta_{u, v}=\beta$ if $(u, v) \not\in \partial E_R$, and $\beta_{u, v}=\gamma$ if $(u, v) \in \partial E_R$.

We first introduce the notion of \textit{induced external field}.
Take any $\bx \in \R^{\TT_R}$ and an edge
 $(u, v) \in E_R$.
Let the subgraph $\TT_{u\setminus v}\subset\TT_R$ be defined by removing the edge $(u, v)$ from $\TT_R$ and taking the connected component containing $u$, and consider the subgraph $\TT_{u\to v}$ obtained by adding the vertex $v$ back to $\TT_{u \setminus v}$ via the edge $(u,v)$.  We define the \textit{Belief Propagation message} $m_{u\to v}^{\bx}$  from $u$ to $v$ as the probability measure on $\{\pm 1 \}$ given by
\[
m_{u\to v}^{\bx}(\sigma) := \frac{1}{Z_{u \to v}^{\bx}} \sum_{\substack{(\sigma_{v'})_{v'\in \TT_{u \to v}} \in \{\pm 1\}^{ \TT_{u \to v} }, \\ \textnormal{with } \sigma_v = \sigma}} \exp \left(\sum_{(u',v')\in E(\TT_{u\to v}) }  \beta_{u',v'}\sigma_{u'}\sigma_{v'}  +\sum_{v' \in \TT_{u\setminus v}} x(v')  \right),
\]
where $Z_{u\to v}^{\bx}$ is the normalizing constant that makes $m_{u\to v}^{\bx}$ a probability measure, i.e. $m_{u\to v}^{\bx}(+1)+m_{u\to v}^{\bx}(-1)=1$.
Note that we regard the external field at $v$ as 0, to measure the effect of $\TT_{u\setminus v}$  on $v$ via the edge $(u,v)$.
See \cite[Chapter 14]{mm09} for a detailed background on the notion of Belief Propagation.

Then, the \textit{induced external field on} $v$ from $u$, denoted by $\zeta_{u\to v}^{\bx}$, is defined as
\begin{equation*}
    \zeta_{u \to v}^{\bx} := \frac{1}{2} \log \left( \frac{m_{u\to v}^{\bx}(+1) }{m_{u\to v}^{\bx}(-1) } \right).
\end{equation*}
In particular, it satisfies $m_{u\to v}^{\bx}(+1) = \frac{1}{2}\{1+\tanh (\zeta_{u\to v}^{\bx})\}$. 
It is straightforward to compute that
$\tanh(\zeta_{u\to v}^{\bx}) = \tanh(\beta_{u,v})\tanh(x(u)+\sum_{w\sim u, w\neq v} \zeta_{w\to u}^{\bx})$, so
\begin{equation}  \label{eq:boundutov}
|\zeta_{u \to v}^{\bx}| \le \beta_{u,v} \le \beta. 
\end{equation}
We also have that
\[
\iprod{\sigma_v}_{\pi_{\TT_{u\to v}}^{\bx}}
=
\frac{m_{u \to v}^{\bx}(+1)e^{x(v)} -m_{u \to v}^{\bx}(-1)e^{-x(v)} }{m_{u \to v}^{\bx}(+1)e^{x(v)} + m_{u \to v}^{\bx}(-1)e^{-x(v)}}=
\tanh(\zeta^\bx_{u\to v}+x(v))  .
\]

We can understand $\zeta_{u\to v}^{\bx} $ alternatively as follows. Let $\TT_{v \setminus u}$ be defined as before (but take the connected component of $v$). If we let $\bx'=(x'(v'))_{v'\in\TT_{v\setminus u}}\in\R^{\TT_{v\setminus u}}$ be the vector where $x'(v')=x(v')+\don[v'=v]\zeta^\bx_{u\to v}$, then $\zeta_{u \to v}^{\bx}$ is the number such that 
the measure $\pi_{\TT_{v\setminus u}}^{\bx'}$ on $\{\pm 1\}^{\TT_{v\setminus u}}$ gives the marginal distribution of $\pi^{\bx}$ on $\TT_{v\setminus u}$.

For any connected subgraph $\GG$ of $\TT_R$ and for any $v\in\GG$, we let 
\begin{equation}\label{eq:def:zetaG}
    \zeta^\bx_\GG(v) = \sum_{u\sim v, u\not\in \GG} \zeta^\bx_{u\to v}.
\end{equation}
Then, if we let $\bx'=(x'(v))_{v\in\GG}\in\R^{\GG}$ be the vector where $x'(v)=x(v)+\zeta^\bx_{\GG}(v)$, then the measure $\pi_{\GG}^{\bx'}$ on $\GG$ gives the marginal distribution of $\pi^{\bx}$ on $\GG$.
Let $\tZ_\GG(\bx)$ be the partition function of $\pi^\bx$ restricted to $\GG$; i.e. we let
 \begin{equation}\label{eq:def:restrictedpartitionfunction:plain}
 \tZ_\GG(\bx) := \sum_{(\sigma_v) \in \{\pm 1\}^{\GG}} \exp\left( \sum_{(u,v)\in E(\GG)} \beta_{u,v}\sigma_u\sigma_v + \sum_{v\in \GG} (x(v)+\zeta_\GG^\bx(v))\sigma_v \right). 
 \end{equation}
We shall also need the following notation of the partition function with some fixed spins.
Take any $\HH\subset \GG$ and $h\in \{\pm 1\}^\HH$. We denote
 \begin{equation}\label{eq:def:ZGh}
 \tZ_{\GG}^{h}(\bx) := \sum_{\substack{(\sigma_v) \in \{\pm 1\}^{\GG}\\ \sigma_v=h(v),\forall v\in\HH}} \exp\left( \sum_{(u,v)\in E(\GG)} \beta_{u,v}\sigma_u\sigma_v + \sum_{v\in \GG} (x(v)+\zeta_\GG^\bx(v))\sigma_v \right).
 \end{equation}
 Here, $\HH \subset \GG$ is a set of vertices in $\GG$ whose spins are fixed by the assignment  $h$.

For each $u,v\in\TT_R$, denote $[u,v]$ as the subgraph given by the shortest path from $u$ to $v$.
For simplicity of notations, we also write $\tZ_{u,v}(\bx):=\tZ_{[u,v]}(\bx)$, and we denote the normalized version as $\oZ_{u,v}(\bx):= \frac{\tZ_{u,v}(\bx)}{\tZ_{u,v}(\boo)}\ge 1$.
We also let
\begin{equation} \label{eq:defA}
A_{u,v} = \prod_{(u',v')\in E([u,v])} \tanh(\beta_{u',v'}).    
\end{equation}
Then by straightforward computation we have that $A_{u,v}$ equals the covariance without external field $\cov{\sigma_u}{\sigma_v}_{\pi}$.
We can use these quantities to write the covariances with external field $\bx$.
\begin{lemma}  \label{L:cov-formula}
For any $u, v \in \TT_R$, we have
\begin{equation}\label{eq:cov:twopts}
\cov{\sigma_u}{\sigma_v}_{\pi^{\bx}} = \frac{A_{u,v}}{\oZ_{u,v}(\bx)^2}.
\end{equation}
For any $(u,u')\in\partial E_R$ with $u'\in \partial\TT_R$, if $v\neq u'$ we have
\begin{equation}\label{eq:cov:threepts:nonendptcase}
\cov{\sigma_u\sigma_{u'}}{\sigma_v}_{\pi^{\bx}} = \frac{\sinh(2x(u'))A_{u,v}}{2\oZ_{u',v}(\bx)^2\cosh^2(\gamma)},
\end{equation}
and if $v=u'$ we have
\begin{equation}\label{eq:cov:threepts:endptcase}
\cov{\sigma_u\sigma_{u'}}{\sigma_v}_{\pi^{\bx}} = \frac{\sinh(2x(u)+2\zeta^{\bx}_{[u,u']}(u))}{2\oZ_{u',u}(\bx)^2\cosh^2(\gamma)}.
\end{equation}
\end{lemma}
We leave the proof of this lemma to Appendix \ref{sec:appa}.

\subsection{Evaluating the chain of covariances}\label{subsec:ind-coup:covchain}
Take $u_0, \cdots, u_k \in \TT_R$ and $(u_-, u_-'), (u_+, u_+') \in \partial E_R$, as in Proposition \ref{prop:chain-expect}.
Also recall the definition of $\bY$ given by \eqref{eq:def:tildeXt:SDE}, and that in law we have \begin{equation}  \label{eq:bYlaw}
\bY=(\bY_t)_{t\ge 0} \overset{d}{=} (t\btau + \bB_t)_{t\ge 0},
\end{equation}
where $\btau=(\tau_v)_{v\in \TT_R}$ is sampled from $\pi=\pi^{\gamma,\boo}$, the free Ising model on $\TT_R$ with inverse temperature $\gamma$ on $\partial E_R$, and inverse temperature $\beta$ on other edges;
and ${\bB} = \{({B}_v(t) )_{ t\ge 0}\}_{v\in \TT_R}$ is another collection of independent one-dimensional standard Brownian motions.
We couple $\bY$ with $\btau, \bB$, so that this equality holds almost surely.

Now we evaluate the LHS of \eqref{eq:chain-cov} using Lemma \ref{L:cov-formula}.
If $u_0\neq u_-'$ and $u_k\neq u_+'$ it equals
\begin{equation} \label{eq:cov-chain}
\E\left[
\frac{\sinh\left(2Y_{t_-}(u_{-}')\right)\sinh\left(2Y_{t_{+}}(u_{+}')\right)}
{4
\oZ_{u_-',u_{0}}(\bY_{t_-})^2
\oZ_{u_+',u_{k}}(\bY_{t_+})^2
\prod_{i=1}^{k}\oZ_{u_i,u_{i-1}}(\bY_{t_i})^2
\cosh^4(\gamma)}
\right]
A_{u_-,u_0}A_{u_k,u_+}\prod_{i=1}^k A_{u_i,u_{i-1}}.
\end{equation}
If $u_0= u_-'$ and $u_k\neq u_+'$, it equals
\begin{equation} \label{eq:cov-chain2}
\E\left[
\frac{\sinh\left(2Y_{t_-}(u_-) + 2\zeta^{\bY_{t_-}}_{[u_-,u_-']}(u_-)\right)\sinh\left(2Y_{t_{+}}(u_{+}')\right)}
{4
\oZ_{u_-',u_-}(\bY_{t_-})^2
\oZ_{u_+',u_{k}}(\bY_{t_+})^2
\prod_{i=1}^{k}\oZ_{u_i,u_{i-1}}(\bY_{t_i})^2
\cosh^4(\gamma)}
\right]
A_{u_k,u_+}\prod_{i=1}^k A_{u_i,u_{i-1}}.
\end{equation}
If $u_0\neq u_-'$ and $u_k= u_+'$, it equals
\begin{equation} \label{eq:cov-chain3}
\E\left[
\frac{
\sinh\left(2Y_{t_-}(u_-')\right)
\sinh\left(2Y_{t_+}(u_+) + 2\zeta^{\bY_{t_+}}_{[u_+,u_+']}(u_+)\right)}
{4
\oZ_{u_-',u_0}(\bY_{t_-})^2
\oZ_{u_+',u_+}(\bY_{t_+})^2
\prod_{i=1}^{k}\oZ_{u_i,u_{i-1}}(\bY_{t_i})^2
\cosh^4(\gamma)}
\right]
A_{u_-,u_0}\prod_{i=1}^k A_{u_i,u_{i-1}}.
\end{equation}
If $u_0= u_-'$ and $u_k= u_+'$, it equals
\begin{equation} \label{eq:cov-chain4}
\E\left[
\frac{
\sinh\left(2Y_{t_-}(u_-) + 2\zeta^{\bY_{t_-}}_{[u_-,u_-']}(u_-)\right)
\sinh\left(2Y_{t_+}(u_+) + 2\zeta^{\bY_{t_+}}_{[u_+,u_+']}(u_+)\right)}
{4
\oZ_{u_-',u_-}(\bY_{t_-})^2
\oZ_{u_+',u_+}(\bY_{t_+})^2
\prod_{i=1}^{k}\oZ_{u_i,u_{i-1}}(\bY_{t_i})^2
\cosh^4(\gamma)}
\right]
\prod_{i=1}^k A_{u_i,u_{i-1}}.
\end{equation}
If $u_-=u_+$, we would get Proposition \ref{prop:chain-expect} immediately from these equations.
Indeed, for each $u, v\in \TT_R$ we have $A_{u,v}\le (\tanh\beta)^{\bd(u,v)}=\theta^{\bd(u,v)}$, by \eqref{eq:defA} (recall that $\theta=\tanh\beta$).
For the expectation factor in each case, note that the denominator is at least $4$.
Thus using \eqref{eq:boundutov} and \eqref{eq:bYlaw}, the expectation  is bounded by $\E[e^{4t+2B_t(u_-')+2B_t(u_+')}]$, $\E[e^{4t+2B_t(u_-)+2d\beta+2B_t(u_+')}]$, $\E[e^{4t+2B_t(u_-')+2B_t(u_+)+2d\beta}]$, or $\E[e^{4t+2B_t(u_-)+2B_t(u_+)+4d\beta}]$, in each case respectively; and each can be further bounded by $e^{8t+d}$, since $d$ is large enough thus $\beta$ is small enough.

Below we assume that $u_-\neq u_+$, thus $u_-'\neq u_+'$.
When $u_0\neq u_-'$ we let $u_{-1}=u_-'$, and otherwise we let $u_{-1}=u_-$;
similarly, when $u_k\neq u_+'$ we let $u_{k+1}=u_+'$, and otherwise we let $u_{k+1}=u_+$.
We also denote $t_0:=t_-$, $t_{k+1}:=t_+$.
A motivation for such definitions is that, for the denominator inside the expectation in each of \eqref{eq:cov-chain}, \eqref{eq:cov-chain2}, \eqref{eq:cov-chain3}, \eqref{eq:cov-chain4}, it can now be written as $4\prod_{i=0}^{k+1}\oZ_{u_i,u_{i-1}}(\bY_{t_i})^2\cosh^4(\gamma)$.

Consider $[u_{-1},u_{k+1}]$, the shortest path from $u_{-1}$ to $u_{k+1}$, and we enumerate the vertices on the path as $u_{-1}=v_1, v_2, \ldots, v_n = u_{k+1}$ for some $n\in\Z_+$. Denote $v_0=u_-'$ and $v_{n+1}=u_+'$. Then the path $[v_1,v_n]$ is contained in the path $[v_0,v_{n+1}]$; $v_0=v_1$ if $u_0\neq u_-'$, and $v_n=v_{n+1}$ if $u_{k+1}\neq u_+'$.
See Figure \ref{fig:chain} for an illustration.

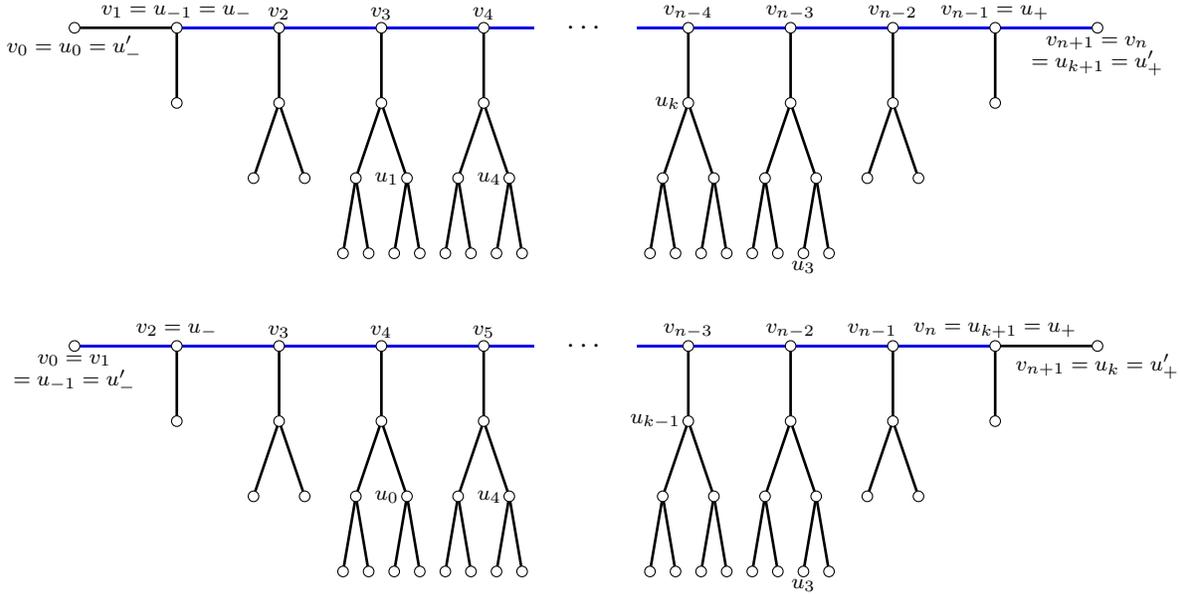
\begin{figure}
    \centering
    \begin{subfloat}
    \centering
\begin{tikzpicture}[line cap=round,line join=round,>=triangle 45,x=6.8cm,y=10cm]
\clip(-0.4,0.73) rectangle (2.15,1.15);

\draw [line width=1pt] (-0.2,1.1) -- (1.8,1.1);

\draw [line width=1pt, color=blue] (0,1.1) -- (1.8,1.1);

\foreach \i in {1,...,7}
{
\draw [line width=1pt] (\i*0.2,1.) -- (\i*0.2-0.05,0.9);
\draw [line width=1pt] (\i*0.2,1.) -- (\i*0.2+0.05,0.9);
}

\foreach \i in {0,...,8}
{
\draw [line width=1pt] (\i*0.2,1.1) -- (\i*0.2,1.);
}

\foreach \i in {3,...,12}
{
\draw [line width=1pt] (\i*0.1+0.075,0.8) -- (\i*0.1+0.05,0.9);
\draw [line width=1pt] (\i*0.1+0.025,0.8) -- (\i*0.1+0.05,0.9);
}

\foreach \i in {0,...,8}
{
\draw [fill=white] (\i*0.2,1) circle (2.0pt);
}

\foreach \i in {-1,...,9}
{
\draw [fill=white] (\i*0.2,1.1) circle (2.0pt);
}

\foreach \i in {1,...,7}
{
\draw [fill=white] (\i*0.2-0.05,0.9) circle (2.0pt);
\draw [fill=white] (\i*0.2+0.05,0.9) circle (2.0pt);
}

\foreach \i in {3,...,12}
{
\draw [fill=white] (\i*0.1+0.075,0.8) circle (2.0pt);
\draw [fill=white] (\i*0.1+0.025,0.8) circle (2.0pt);
}
\fill[line width=0.pt,color=white,fill=white]
(.7,-10) -- (.9,-10) -- (.9,10)-- (.7,10) -- cycle;

\begin{scriptsize}
\draw (-0.2,1.1) node[anchor=north]{$v_0=u_{0}=u_-'$};
\draw (0.0,1.1) node[anchor=south]{$v_1=u_{-1}=u_-$};
\draw (0.2,1.1) node[anchor=south]{$v_2$};
\draw (0.4,1.1) node[anchor=south]{$v_3$};
\draw (0.6,1.1) node[anchor=south]{$v_4$};
\draw (1.0,1.1) node[anchor=south]{$v_{n-4}$};
\draw (1.2,1.1) node[anchor=south]{$v_{n-3}$};
\draw (1.4,1.1) node[anchor=south]{$v_{n-2}$};
\draw (1.6,1.1) node[anchor=south]{$v_{n-1}=u_+$};
\draw (1.8,1.1) node[anchor=north]{$v_{n+1}=v_n$};
\draw (1.8,1.08) node[anchor=north]{$=u_{k+1}=u_+'$};
\draw (0.45,0.9) node[anchor=east]{$u_1$};
\draw (1.0,1.0) node[anchor=east]{$u_k$};
\draw (1.225,0.8) node[anchor=north]{$u_3$};
\draw (0.65,0.9) node[anchor=east]{$u_4$};
\end{scriptsize}
\draw (0.8,1.1) node[anchor=center]{$\cdots$};

\end{tikzpicture}
\end{subfloat}
    \begin{subfloat}
    \centering
\begin{tikzpicture}[line cap=round,line join=round,>=triangle 45,x=6.8cm,y=10cm]
\clip(-0.4,0.73) rectangle (2.15,1.15);

\draw [line width=1pt] (-0.2,1.1) -- (1.8,1.1);

\draw [line width=1pt, color=blue] (-0.2,1.1) -- (1.6,1.1);

\foreach \i in {1,...,7}
{
\draw [line width=1pt] (\i*0.2,1.) -- (\i*0.2-0.05,0.9);
\draw [line width=1pt] (\i*0.2,1.) -- (\i*0.2+0.05,0.9);
}

\foreach \i in {0,...,8}
{
\draw [line width=1pt] (\i*0.2,1.1) -- (\i*0.2,1.);
}

\foreach \i in {3,...,12}
{
\draw [line width=1pt] (\i*0.1+0.075,0.8) -- (\i*0.1+0.05,0.9);
\draw [line width=1pt] (\i*0.1+0.025,0.8) -- (\i*0.1+0.05,0.9);
}

\foreach \i in {0,...,8}
{
\draw [fill=white] (\i*0.2,1) circle (2.0pt);
}

\foreach \i in {-1,...,9}
{
\draw [fill=white] (\i*0.2,1.1) circle (2.0pt);
}

\foreach \i in {1,...,7}
{
\draw [fill=white] (\i*0.2-0.05,0.9) circle (2.0pt);
\draw [fill=white] (\i*0.2+0.05,0.9) circle (2.0pt);
}

\foreach \i in {3,...,12}
{
\draw [fill=white] (\i*0.1+0.075,0.8) circle (2.0pt);
\draw [fill=white] (\i*0.1+0.025,0.8) circle (2.0pt);
}

\fill[line width=0.pt,color=white,fill=white]
(.7,-10) -- (.9,-10) -- (.9,10)-- (.7,10) -- cycle;

\begin{scriptsize}
\draw (-0.2,1.1) node[anchor=north]{$v_0=v_1$};
\draw (-0.2,1.08) node[anchor=north]{$=u_{-1}=u_-'$};
\draw (0.0,1.1) node[anchor=south]{$v_2=u_-$};
\draw (0.2,1.1) node[anchor=south]{$v_3$};
\draw (0.4,1.1) node[anchor=south]{$v_4$};
\draw (0.6,1.1) node[anchor=south]{$v_5$};
\draw (1.0,1.1) node[anchor=south]{$v_{n-3}$};
\draw (1.2,1.1) node[anchor=south]{$v_{n-2}$};
\draw (1.36,1.1) node[anchor=south]{$v_{n-1}$};
\draw (1.6,1.1) node[anchor=south]{$v_{n}=u_{k+1}=u_+$};
\draw (1.8,1.1) node[anchor=north]{$v_{n+1}=u_k=u_+'$};
\draw (0.45,0.9) node[anchor=east]{$u_0$};
\draw (1.0,1.0) node[anchor=east]{$u_{k-1}$};
\draw (1.225,0.8) node[anchor=north]{$u_3$};
\draw (0.65,0.9) node[anchor=east]{$u_4$};
\end{scriptsize}
\draw (0.8,1.1) node[anchor=center]{$\cdots$};

\end{tikzpicture}
\end{subfloat}
\caption{
Illustrations of the path $[u_{-1}, u_{k+1}]=[v_1,v_n]$, indicated by the blue edges. 
The top one is the case where $u_0= u_-'$ and $u_{k+1}\neq u_+'$,
while the bottom one is the case where $u_0\neq u_-'$ and $u_{k+1}= u_+'$.
}  
\label{fig:chain}
\end{figure}
Our next goal is to expand the factors
\begin{equation}  \label{eq:two-sinh}
\sinh\left(2Y_{t_-}(u_-) + 2\zeta^{\bY_{t_-}}_{[u_-,u_-']}(u_-)\right), \quad \sinh\left(2Y_{t_+}(u_+) + 2\zeta^{\bY_{t_+}}_{[u_+,u_+']}(u_+)\right),
\end{equation}
which appeared in \eqref{eq:cov-chain2}, \eqref{eq:cov-chain3}, \eqref{eq:cov-chain4}.
For this we set up some notations for induced external fields along this path $[v_0,v_{n+1}]$.
For any $\bx \in \R^{\TT_R}$ and $1\le \ell \le n$, let
\begin{equation*}
    \begin{split}
        &\zeta^\bx_\ell := x(v_\ell)+\sum_{v\sim v_\ell, v\neq v_{\ell-1},v_{\ell+1}}\zeta^\bx_{v\to v_\ell}=x(v_\ell)+\zeta_{[v_0,v_{n+1}]}^\bx(v_\ell);\\
        &\zeta^\bx_{\ell-1\to \ell} := \zeta^\bx_{v_{\ell-1}\to v_\ell}, \quad \textnormal{and} \quad  \zeta^\bx_{\ell+1\to \ell} := \zeta^\bx_{v_{\ell+1}\to v_\ell}.
    \end{split}
\end{equation*}
Here we assume that $\zeta^\bx_{u_-'\to u_-'} = \zeta^\bx_{u_+'\to u_+'} = 0$.
In words, $\zeta^\bx_\ell$ is the total external and induced external field on $v_\ell$, except for those from $v_{\ell-1}$ and $v_{\ell+1}$, which are $\zeta^\bx_{\ell-1\to \ell}$ and $\zeta^\bx_{\ell+1\to \ell}$ respectively.
From these definitions we have that
\begin{equation}  \label{eq:Y2-}
2Y_{t_-}(u_-) + 2\zeta^{\bY_{t_-}}_{[u_-,u_-']}(u_-)
=
2\zeta^{\bY_{t_-}}_1 + 2\zeta^{\bY_{t_-}}_{2\to 1},
\end{equation}
when $u_0=u_-'$ (then $v_1=u_{-1}=u_-$, see top-left of Figure \ref{fig:chain}); and
\begin{equation}  \label{eq:Y2+}
2Y_{t_+}(u_+) + 2\zeta^{\bY_{t_+}}_{[u_+,u_+']}(u_+)
=
2\zeta^{\bY_{t_+}}_n + 2\zeta^{\bY_{t_+}}_{n-1\to n},
\end{equation}
when $u_k=u_+'$ (then $v_n=u_{k+1}=u_+$, see bottom-right of Figure \ref{fig:chain}).

We now study \eqref{eq:two-sinh}.
For each $1\le \ell < n$, by direct computation we can write
\[
\sinh(2\zeta^\bx_{\ell+1\to \ell}) = \frac{2\tanh(\beta_{v_\ell,v_{\ell+1}}) \cosh^2(\beta_{v_\ell,v_{\ell+1}}) \sinh(2\zeta^\bx_{\ell+1} + 2\zeta^\bx_{\ell+2\to \ell+1})}{\cosh(2\beta_{v_\ell,v_{\ell+1}})+\cosh(2\zeta^\bx_{\ell+1} + 2\zeta^\bx_{\ell+2\to \ell+1})},
\]
which implies that
\begin{equation}  \label{eq:zeta-l-ind-l+1}
\begin{split}
\sinh( 2\zeta^\bx_\ell+& 2\zeta^\bx_{\ell+1\to \ell})
=
\sinh(2\zeta^\bx_\ell)\cosh(2\zeta^\bx_{\ell+1\to \ell}) + \cosh(2\zeta^\bx_\ell)\sinh(2\zeta^\bx_{\ell+1\to \ell})
\\
=&
\sinh(2\zeta^\bx_\ell)\cosh(2\zeta^\bx_{\ell+1\to \ell})\\
&+
\frac{2\tanh(\beta_{v_\ell,v_{\ell+1}}) \cosh(2\zeta^\bx_\ell)\cosh^2(\beta_{v_\ell,v_{\ell+1}}) \sinh(2\zeta^\bx_{\ell+1} + 2\zeta^\bx_{\ell+2\to \ell+1})}{\cosh(2\beta_{v_\ell,v_{\ell+1}})+\cosh(2\zeta^\bx_{\ell+1} + 2\zeta^\bx_{\ell+2\to \ell+1})}.
\end{split}    
\end{equation}
Now for $1\le \ell < n$, we denote 
\begin{equation}  \label{eq:def-U}
U_\ell(\bx) := \frac{2\cosh(2\zeta^\bx_\ell)\cosh^2(\beta_{v_\ell,v_{\ell+1}})}{\cosh(2\beta_{v_\ell,v_{\ell+1}})+\cosh(2\zeta^\bx_{\ell+1} + 2\zeta^\bx_{\ell+2\to \ell+1})}.
\end{equation}
Thus we can write \eqref{eq:zeta-l-ind-l+1} as
\begin{multline}  \label{eq:zeta-one-step-plus}
\sinh(2\zeta^\bx_\ell+2\zeta^\bx_{\ell+1\to \ell}) \\
=
\sinh(2\zeta^\bx_\ell)\cosh(2\zeta^\bx_{\ell+1\to \ell}) + U_\ell(\bx)\tanh(\beta_{v_\ell,v_{\ell+1}}) \sinh(2\zeta^\bx_{\ell+1} + 2\zeta^\bx_{\ell+2\to \ell+1}).
\end{multline}
When $u_0=u_-'$, we have
\begin{equation}  \label{eq:expan-1}
\begin{split}
\sinh & \left(2Y_{t_-}(u_-) + 2\zeta^{\bY_{t_-}}_{[u_-,u_-']}(u_-)\right)
=
\sinh\left(2\zeta^{\bY_{t_-}}_1 + 2\zeta^{\bY_{t_-}}_{2\to 1}\right)
\\
=
&\sum_{\ell=1}^{n -1}
\sinh\left(2\zeta^{\bY_{t_-}}_\ell\right)\cosh\left(2\zeta^{\bY_{t_-}}_{\ell+1\to \ell}\right) A_{v_1,v_\ell}
\prod_{1\le j < \ell} U_j(\bY_{t_-})
 \\ &   +
\sinh\left(2\zeta^{\bY_{t_-}}_{n}+2\zeta^{\bY_{t_-}}_{n+1\to n}\right) A_{v_1,v_n}
\prod_{1\le j < n} U_j(\bY_{t_-}).
\end{split}
\end{equation}
where the first equality is by \eqref{eq:Y2-}, and the second equality is by repeated applying \eqref{eq:zeta-one-step-plus}.
Similarly, when $u_k=u_+'$, by \eqref{eq:Y2+} and repeated expansion, we have
\begin{equation}  \label{eq:expan-2}
\begin{split}
\sinh&\left(2Y_{t_+}(u_+) + 2\zeta^{\bY_{t_+}}_{[u_+,u_+']}(u_+)\right)
=
\sinh\left(2\zeta^{\bY_{t_+}}_n + 2\zeta^{\bY_{t_+}}_{n-1\to n}\right)
\\
\quad =&
\sum_{\ell=2}^{n}
\sinh\left(2\zeta^{\bY_{t_+}}_\ell\right)\cosh\left(2\zeta^{\bY_{t_+}}_{\ell-1\to \ell}\right) A_{v_\ell,v_n}
\prod_{\ell < j \le n} V_j(\bY_{t_+})
\\
\quad &+
\sinh\left(2\zeta^{\bY_{t_+}}_{1}+2\zeta^{\bY_{t_+}}_{0\to 1}\right) A_{v_1,v_n}
\prod_{1< j \le n} V_j(\bY_{t_+}),
\end{split}
\end{equation}
where
\begin{equation} \label{eq:def-V}
V_\ell(\bx) := \frac{2\cosh(2\zeta^{\bx}_\ell)\cosh^2(\beta_{v_\ell,v_{\ell-1}})}{\cosh(2\beta_{v_\ell,v_{\ell-1}})+\cosh(2\zeta^\bx_{\ell-1} + 2\zeta^\bx_{\ell-2\to \ell-1})}    
\end{equation}
for each $1< \ell \le n$ and $\bx \in \R^{\TT_R}$.
Now to prove Proposition \ref{prop:chain-expect} it now suffices to prove the following technical estimate.
Recall that $\theta=\tanh\beta$.
\begin{proposition}  \label{p:chain-reformulate}
For $1\le l \le r \le n$, 
let $C_l:\R^{\TT_R}\to\R$ satisfy either $C_l\equiv 1$ or $C_l(\bx) = \cosh(2\zeta^{\bx}_{l+1\to l}) $ for any $\bx \in \R^{\TT_R}$;
and $C_r:\R^{\TT_R}\to\R$ satisfy either $C_r\equiv 1$ or $C_r(\bx) = \cosh(2\zeta^{\bx}_{r-1\to r}) $ for any $\bx \in \R^{\TT_R}$.
Then we have
\begin{align*}
\E\left[
\frac{
\sinh\left(2\zeta^{\bY_{t_-}}_{l}\right)
\sinh\left(2\zeta^{\bY_{t_+}}_{r}\right)
C_l(\bY_{t_-}) C_r(\bY_{t_+})
}
{
\prod_{i=0}^{k+1}\oZ_{u_i,u_{i-1}}(\bY_{t_i})^2}
\prod_{1\le j < l} U_j(\bY_{t_-})
\prod_{r < j \le n} V_j(\bY_{t_+}) 
\right]    
\\
\le
e^{17t+d} \theta^{r-l} C^{n+k+\sum_{j=0}^{k+1}\bd(u_j, u_{j-1})}
\end{align*}
where $C$ is an absolute constant.
\end{proposition}

Assuming this, we now prove Proposition \ref{prop:chain-expect}, by bounding each of \eqref{eq:cov-chain}, \eqref{eq:cov-chain2}, \eqref{eq:cov-chain3}, and \eqref{eq:cov-chain4}, using the expansions \eqref{eq:expan-1} and \eqref{eq:expan-2} and Proposition \ref{p:chain-reformulate}.
\begin{proof}[Proof of Proposition \ref{prop:chain-expect}]
In the case where $u_0\neq u_-'$ and $u_k\neq u_+'$, we have $\zeta^{\bY_{t_-}}_{1}=Y_{t_-}(u_-')$ and $\zeta^{\bY_{t_+}}_{n}=Y_{t_+}(u_+')$ by their definitions.
We apply Proposition \ref{p:chain-reformulate} with $l=1$ and $r=n$,
and bound \eqref{eq:cov-chain} by
\[
C^{n+k}(C\theta)^{\bd(u_-',u_+')+\bd(u_-,u_0)+\bd(u_k,u_+)+\sum_{i=1}^k\bd(u_i,u_{i-1})}
e^{17t+d} .
\]
For the case where $u_0\neq u_-'$ and $u_k= u_+'$,
using \eqref{eq:expan-2}, we apply Proposition \ref{p:chain-reformulate} for $l=1$ and each $1\le r\le n$ (note that $\zeta_{0\to 1}^{\bY_{t+}} = 0$ in this case), and bound \eqref{eq:cov-chain2} by
\[
n
C^{n+k}(C\theta)^{\bd(u_-',u_+)+\bd(u_-,u_0)+\sum_{i=1}^k\bd(u_i,u_{i-1})}
e^{17t+d}.
\]
Similarly, for the case where $u_0= u_-'$ and $u_k\neq u_+'$, using \eqref{eq:expan-1} and Proposition \ref{p:chain-reformulate} we can bound \eqref{eq:cov-chain3} by
\[
n
C^{n+k}(C\theta)^{\bd(u_-,u_+')+\bd(u_+,u_k)+\sum_{i=1}^k\bd(u_i,u_{i-1})}
e^{17t+d}.
\]
Finally, for the case where $u_0= u_-'$ and $u_k= u_+'$, by \eqref{eq:expan-1} and \eqref{eq:expan-2}, and using Proposition \ref{p:chain-reformulate} for each $1\le l < r\le n$, we can bound
\eqref{eq:cov-chain4} by
\begin{equation} \label{eq:cov-chain4-pf}
\begin{split} 
&
\frac{n(n-1)}{2}
C^{n+k}(C\theta)^{\bd(u_-,u_+)+\sum_{i=1}^k\bd(u_i,u_{i-1})}
e^{17t+d}
+\sum_{1\le r\le l \le n} \theta^{n-r+l-1+\sum_{i=1}^k\bd(u_i,u_{i-1})}
\\
&\times
\E\left[
e^{2|\zeta^{\bY_{t_-}}_l| + 2|\zeta^{\bY_{t_-}}_{l+1\to l}|
+2|\zeta^{\bY_{t_+}}_r| + 2|\zeta^{\bY_{t_+}}_{r-1\to r}|}
\prod_{1\le j < l} U_j(\bY_{t_-})
\prod_{r < j \le n} V_j(\bY_{t_+})
\right].
\end{split}    
\end{equation}
We next bound the expectation in the second line.
For any $\bx\in\R^{\TT_R}$, we have 
$\frac{\cosh(2\zeta^\bx_{j})}{\cosh(2\zeta^\bx_{j} + 2\zeta^\bx_{j+1\to j})} < 2\cosh(2\zeta^\bx_{j+1\to j}) \le 2\cosh(2\beta)$ for each $1\le j < l$, by \eqref{eq:boundutov}.
So by \eqref{eq:def-U} and \eqref{eq:def-V} we have
\begin{equation}  \label{eq:bound-U}
\prod_{1\le j < l} U_j(\bx)
<
\prod_{1\le j < l}
\frac{2\cosh^2(\beta)\cosh(2\zeta^\bx_{j})}{\cosh(2\zeta^\bx_{j+1} + 2\zeta^\bx_{j+2\to j+1})}
<
\cosh(2\zeta_1^\bx)
(4\cosh^2(\beta)\cosh(2\beta))^{l-1},    
\end{equation}
and similarly
\[
\prod_{r < j \le n} V_j(\bx) 
<
\prod_{r < j \le n}
\frac{2\cosh^2(\beta)\cosh(2\zeta^\bx_{j})}{\cosh(2\zeta^\bx_{j-1} + 2\zeta^\bx_{j-2\to j-1})}
<
\cosh(2\zeta_n^\bx)
(4\cosh^2(\beta)\cosh(2\beta))^{n-r}.
\]
Thus the second line of \eqref{eq:cov-chain4-pf} can be bounded by
\[
\begin{split}
&(4\cosh^2(\beta)\cosh(2\beta))^{2n}
\E\left[
e^{2|\zeta^{\bY_{t_-}}_l| + 2|\zeta^{\bY_{t_-}}_{l+1\to l}|
+2|\zeta^{\bY_{t_+}}_r| + 2|\zeta^{\bY_{t_+}}_{r-1\to r}|
+2|\zeta^{\bY_{t_-}}_1|
+2|\zeta^{\bY_{t_+}}_n|
}
\right] \\
<&
C^n \E\left[e^{4\beta + 8t + 2|B_{t_{-}}(v_l)|
+ 2|B_{t_{-}}(v_1)|
+ 2|B_{t_{+}}(v_n)|
+ 2|B_{t_{+}}(v_r)| + 8(d-1)\beta
}\right]
<C^{n+1}e^{25t+d}.
\end{split}
\]
Now we conclude that \eqref{eq:cov-chain4} can be bounded by
\[
n^2
C^{n+k}(C\theta)^{\bd(u_-,u_+)+\sum_{i=1}^k\bd(u_i,u_{i-1})}
e^{25t+d}.
\]
Then our conclusion follows in each case. 
\end{proof}

\subsection{Inductive construction of coupling}\label{subsec:ind-coup:ind-coup}
In this and the next subsection we prove Proposition~\ref{p:chain-reformulate}, by developing an inductive coupling scheme to study the partition functions in the main inequality.

We consider the space $\JJ=\{\pm 1\}^{\TT_R}\times C([0,\infty),\R)^{\TT_R}$ with coordinates $(\btau, \bB)$, and regard $\bY$ as a function of $\btau, \bB$ defined via \eqref{eq:bYlaw} on $\JJ$.
We let $\mu$ be the probability measure on the  space $\JJ$ such that for $(\btau, \bB) \sim \mu$, $\btau$ is sampled from $\pi=\pi^{\gamma,\boo}$ and $\bB$ are independent Brownian motions.
Let $1\le l \le r\le n$ be given, and set $\mu_\pm$ be the measures on $\JJ$ with $d\mu_\pm = \don[\pm\zeta^{\bY_{t_-}}_l\ge 0] d\mu$.

We let $\mu^*$ be another measure on $\JJ$, defined as 
\[
d\mu^* = 
\frac{
|\sinh(2\zeta^{\bY_{t_-}}_{l})|
C_l(\bY_{t_-})C_r(\bY_{t_+})
}
{\cosh(2\zeta^{\bY_{t_+}}_n)
\prod_{i=0}^{k+1}\oZ_{u_i,u_{i-1}}(\bY_{t_i})^2}
\prod_{1\le j < l} U_j(\bY_{t_-})
\prod_{r < j \le n} V_j(\bY_{t_+})
d\mu.
\]
We let $\mu^*_\pm$ be the measure on $\JJ$ with $d\mu^*_\pm=\don[\pm\zeta^{\bY_{t_-}}_l\ge 0] d\mu^*$.
To prove Proposition \ref{p:chain-reformulate}, it now suffices to bound
\begin{equation}  \label{eq:chain-refor-mu-diff}
\int \sinh(2\zeta^{\bY_{t_+}}_{r})
\cosh(2\zeta^{\bY_{t_+}}_{n})
d\mu^*_+ - \int \sinh(2\zeta^{\bY_{t_+}}_{r})
\cosh(2\zeta^{\bY_{t_+}}_{n})
d\mu^*_-.    
\end{equation}
We leave out the factor of $\sinh(2\zeta^{\bY_{t_+}}_{r})
\cosh(2\zeta^{\bY_{t_+}}_{n})$ for technical reasons.
To bound \eqref{eq:chain-refor-mu-diff}, we construct a coupling $\Gamma^*$ of $\mu_+^*$ and $\mu_-^*$.
In words, denote the coordinates of $\JJ^2$ as $(\btau_-,\bB_-), (\btau_+,\bB_+)$, we will construct a measure $\Gamma^*$ on it, such that for any measurable function $f$ on $\JJ$ we have
$\int f(\btau_-, \bB_-) d\Gamma^* = \int f d\mu^*_-$ and $\int f(\btau_+, \bB_+) d\Gamma^* = \int f d\mu^*_+$.
Denote $\bY^-, \bY^+$ as the two copies of $\bY$,
regarded as functions of $(\btau_-, \bB_-)$ and $(\btau_+, \bB_+)$.
With such a coupling we would like to bound
\begin{equation}  \label{eq:couple-bound}
\int \left\{\sinh(2\zeta^{\bY_{t_+}^+}_{r}) \cosh(2\zeta^{\bY_{t_+}^+}_{n})
- \sinh(2\zeta^{\bY_{t_+}^-}_{r}) \cosh(2\zeta^{\bY_{t_+}^-}_{n})
\right\}d\Gamma^*.    
\end{equation}

We construct such coupling $\Gamma^*$ inductively.
For each $l\le \ell < r$, we denote $\GG^{(\ell)}$ as the subgraph obtained by breaking the edge $(v_\ell,v_{\ell+1})$ in $\TT_R$, and taking the connected component containing $v_1$.
We also let $\GG^{(r)}=\TT_R$. See Figure \ref{fig:indu-G} for an illustration.

\begin{figure}
    \centering
\begin{tikzpicture}[line cap=round,line join=round,>=triangle 45,x=8cm,y=10cm]
\clip(-0.4,0.63) rectangle (1.7,1.15);

\draw [line width=1pt] (0.2,1) -- (1.2,1);

\draw [line width=1pt] (0.2,1.) -- (0,0.9);
\draw [line width=1pt] (1.2,1.) -- (1.4,0.9);

\foreach \i in {1,...,6}
{
\draw [line width=1pt] (\i*0.2,.9) -- (\i*0.2,1.);
}

\foreach \i in {0,...,7}
{
\draw [line width=1pt] (\i*0.2,.9) -- (\i*0.2-0.05,0.8);
\draw [line width=1pt] (\i*0.2,.9) -- (\i*0.2+0.05,0.8);
}

\foreach \i in {-1,...,14}
{
\draw [line width=1pt] (\i*0.1+0.075,0.7) -- (\i*0.1+0.05,0.8);
\draw [line width=1pt] (\i*0.1+0.025,0.7) -- (\i*0.1+0.05,0.8);
}

\foreach \i in {1,...,6}
{
\draw [fill=white] (\i*0.2,1) circle (2.0pt);
}

\foreach \i in {0,...,7}
{
\draw [fill=white] (\i*0.2,.9) circle (2.0pt);
}

\foreach \i in {0,...,7}
{
\draw [fill=white] (\i*0.2-0.05,0.8) circle (2.0pt);
\draw [fill=white] (\i*0.2+0.05,0.8) circle (2.0pt);
}

\foreach \i in {-1,...,14}
{
\draw [fill=white] (\i*0.1+0.075,0.7) circle (2.0pt);
\draw [fill=white] (\i*0.1+0.025,0.7) circle (2.0pt);
}

\draw [dashed] plot coordinates {(0.3,1.04) (0.3,0.67) (-0.1,0.67) (-0.1,1.04) (0.3,1.04)};
\draw [dashed] plot coordinates {(0.5,1.1) (0.5,0.65) (-0.12,0.65) (-0.12,1.1) (0.5,1.1)};

\begin{scriptsize}
\draw (-0.0675,0.7) node[anchor=south east]{$v_{l-3}$};
\draw (-0.05,0.8) node[anchor=south east]{$v_{l-2}$};
\draw (0,0.9) node[anchor=south east]{$v_{l-1}$};
\draw (0.2,1) node[anchor=south]{$v_l$};
\draw (0.4,1) node[anchor=south]{$v_{l+1}$};
\draw (0.6,1) node[anchor=south]{$v_{l+2}$};
\draw (0.8,1) node[anchor=south]{$v_{l+3}$};
\draw (0.05,0.8) node[anchor=east]{$u$};
\draw (0.4,0.9) node[anchor=east]{$v$};
\draw (1.2,1) node[anchor=south]{$v_r$};

\draw (0,1.1) node[anchor=north]{$\GG^{(l+1)}$};
\draw (0,1.04) node[anchor=north]{$\GG^{(l)}$};

\draw (0.7,0.64) node[anchor=south]{$\cdots$};

\end{scriptsize}

\end{tikzpicture}
\caption{
An illustration of the path from $v_l$ to $v_r$, and the subgraphs $\GG^{(\ell)}$.
The vertex $u$ has $\rho(u)=l$ and the vertex $v$ has $\rho(v)=l+1$.
The distance $\bbd_{l+3}^r(u,v)=2$ since $\rho(u)=l$ and $\rho(v)=l+1$.
}  
\label{fig:indu-G}
\end{figure}
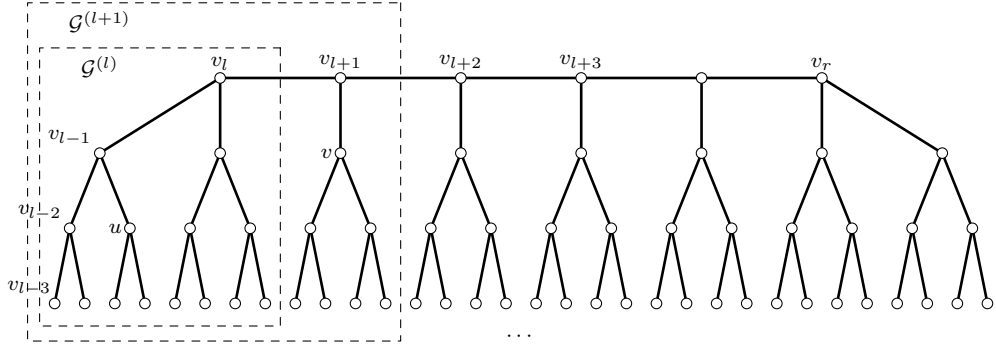

Our general strategy is (1) to define a version of $\mu_\pm^*$ that only depends on the information on $\GG^{(\ell)}$, denoted as $\mu_\pm^{(\ell)}$; (2) then we shall construct a coupling $\Gamma^{(\ell)}$ of $\mu_+^{(\ell)}$ and $\mu_-^{(\ell)}$ inductively in $\ell$, and $\Gamma^{(r)}$ would be the desired $\Gamma^*$.

\noindent\textbf{The notation $W_j$.} We introduce the notation $W_j$ for $r<j\le n$, which is a slightly different version of $V_j$.
For any $\bx \in \R^{\TT_R}$ we denote
\[
W_{r+1}(\bx) := \frac{2\cosh^2(\beta_{v_{r+1},v_{r}})}{\cosh(2\beta_{v_{r+1},v_{r}})+\cosh(2\zeta^\bx_{r} + 2\zeta^\bx_{r-1\to r})},\]
and for each $r+1<j\le n$ we denote 
\[
W_j(\bx) := \frac{2\cosh(2\zeta^{\bx}_{j-1})\cosh^2(\beta_{v_j,v_{j-1}})}{\cosh(2\beta_{v_j,v_{j-1}})+\cosh(2\zeta^\bx_{j-1} + 2\zeta^\bx_{j-2\to j-1})}.
\]
We note that this is slightly different from $V_j(\bx)$; in fact, we have $W_{r+1}(\bx) = (\cosh(2\zeta^{\bx}_{r+1}))^{-1}V_{r+1}(\bx)$ and 
$W_j(\bx) = (\cosh(2\zeta^{\bx}_j))^{-1}\cosh(2\zeta^{\bx}_{j-1})V_j(\bx)$ for $r+1<j\le n$.
Such $W_j$ is easier to use (than $V_j$) in the inductive arguments, since it can be bounded by a constant: $\forall r<j\le n$, we have
\begin{equation}  \label{eq:bound-W}
W_j(\bx)<\frac{2\cosh^2(\beta)\cosh(2\zeta^\bx_{j-1})}{\cosh(2\zeta^\bx_{j-1} + 2\zeta^\bx_{j-2\to j-1})}
<4\cosh^2(\beta)\cosh(2\zeta^\bx_{j-2\to j-1})
\le 4\cosh^2(\beta)\cosh(2\beta)
,
\end{equation}
since $|\zeta^\bx_{j-2\to j-1}| \le \beta$ (by \eqref{eq:boundutov}).
Now we can write
\[
d\mu^* = 
\frac{
|\sinh(2\zeta^{\bY_{t_-}}_{l})|
C_l(\bY_{t_-})C_r(\bY_{t_+})
}
{
\prod_{i=0}^{k+1}\oZ_{u_i,u_{i-1}}(\bY_{t_i})^2}
\prod_{1\le j < l} U_j(\bY_{t_-})
\prod_{r < j \le n} W_j(\bY_{t_+})
d\mu.
\]

\noindent\textbf{Definition of $\mu_\pm^{(\ell)}$.} For each $l\le \ell \le r$ and any $\bx\in\R^{\TT_R}$, 
let $\bx^{(\ell)}\in\R^{\TT_R}$ be the vector given by $\bx^{(\ell)}(v')=\bx(v')$ for $v'\in\GG^{(\ell)}$, and $\bx^{(\ell)}(v')=0$ for $v'\not\in\GG^{(\ell)}$.
Then we define $\oZ_{u,v}^{(\ell)}$ as the function such that

\begin{itemize}
    \item if at least one of $u,v \in \GG^{(\ell)}$, then $\oZ_{u,v}^{(\ell)}(\bx)=\oZ_{u,v}(\bx^{(\ell)})$,
    \item if $u,v \notin \GG^{(\ell)}$, then $\oZ_{u,v}^{(\ell)}(\bx)\equiv 1$.
\end{itemize}
Also, we define $C_l^{(\ell)}(\bx) := C_l(\bx^{(\ell)})$, and $U_j^{(\ell)}(\bx):=U_j(\bx^{(\ell)})$, for any $1\le j < l$;
and 
$C_r^{(\ell)}(\bx) := C_r(\bx^{(\ell)})$, and $W_j^{(\ell)}(\bx):=W_j(\bx^{(\ell)})$, for any $r < j \le n$.
We write $\mu^{(\ell)}$ as the measure on $\JJ$, with 
\[
d\mu^{(\ell)} =
\frac{
|\sinh(2\zeta^{\bY_{t_-}}_{l})|
C_l^{(\ell)}(\bY_{t_-})C_r^{(\ell)}(\bY_{t_+})
}
{
\prod_{i=0}^{k+1}\oZ_{u_i,u_{i-1}}^{(\ell)}(\bY_{t_i})^2}
\prod_{1\le j < l} U_j^{(\ell)}(\bY_{t_-})
\prod_{r < j \le n} W_j^{(\ell)}(\bY_{t_+})
d\mu,
\]
and we let $\mu^{(\ell)}_\pm$ be the measure on $\JJ$ with $d\mu^{(\ell)}_\pm=\don[\pm\zeta^{\bY_{t_-}}_l\ge 0] d\mu^{(\ell)}$.
From these definitions we have $\mu^{(r)}=\mu^*$ and $\mu^{(r)}_\pm=\mu^*_\pm$.

\vspace{2mm}

\noindent\textbf{Construction of $\Gamma$ and $\Gamma^{(l)}$.}
We start by constructing $\Gamma$, a probability measure on $\JJ^2$, and a coupling of $2\mu_\pm$ (recall that $d\mu_\pm = \don[\pm\zeta^{\bY_{t_-}}_l\ge 0] d\mu$; we multiplied the measures $\mu_{\pm}$ by the scalar 2 for convenience, since $2\mu_\pm$ are probability measures).

We define $\Gamma$ via defining random variables $\btau^\pm, \bB^\pm$, as follows.
We first take $(\tau_{v}^+, (B_t^+(v))_{t\ge 0})_{v\in \GG^{(l)}}$ as sampled from $2\mu_+$;
and we take $\tau_{v}^- = - \tau_{v}^+, B_t^-(v) = -B_t^+(v)$ for each $v\in\GG^{(l)}$ and $t\ge 0$.

Next we regard $\TT_R$ as a tree rooted at $v_l$ and inductively define $\tau_{v}^\pm$ for each $v \not\in \GG^{(l)}$.
For any $v \not\in \GG^{(l)}$, denote $v'$ as its parent in this $v_l$-rooted tree.
Given $\tau_{v'}^\pm$, we need to take $\tau_{v}^{\pm}$ to be $\tau_{v}^{\pm}=\tau_{v'}^{\pm}$ with probability $\frac{e^{\beta_{v,v'}}}{e^{\beta_{v,v'}}+e^{-{\beta_{v,v'}}}}$, according to the law of the Ising model.
To couple $\tau_{v}^+$ and $\tau_{v}^-$, if $\tau_{v'}^-=\tau_{v'}^+$ we set  $\tau_{v}^-=\tau_{v}^+$ with probability $1$;
and if $\tau_{v'}^-\neq\tau_{v'}^+$, we have $\tau_{v'}^- = \tau_{v}^-\neq\tau_{v}^+ = \tau_{v'}^+$
with probability $\tanh(\beta_{v,v'})$, and otherwise $\tau_{v}^{\pm} = -$ or $\tau_{v}^{\pm} = +$, each with probability $\frac{1}{2}(1-\tanh( \beta_{v,v'}) )$. 
We let $(B_t^-(v))_{t\ge 0}=(B_t^+(v))_{t\ge 0}$
be the same Brownian motion (but independent for each $v\not\in \GG^{(l)}$).

From this we have that the marginal distributions of $(\btau^-, \bB^-)$ and $(\btau^+, \bB^+)$ are given by $2\mu_-$ and $2\mu_+$, respectively.

To get $\Gamma^{(l)}$, we reweight $\Gamma$.
Note under $\Gamma$, almost surely we have $\sinh(2\zeta^{\bY_{t_-}^-}_{l})=-\sinh(2\zeta^{\bY_{t_-}^+}_{l})$, 
$C_l^{(l)}(\bY_{t_-}^-)=C_l^{(l)}(\bY_{t_-}^+)$, 
$C_r^{(l)}(\bY_{t_-}^-)=C_r^{(l)}(\bY_{t_-}^+)$;
and 
$U_j^{(l)}(\bY_{t_-}^-)=U_j^{(l)}(\bY_{t_-}^+)$ for each $1\le j <l$,
$W_j^{(l)}(\bY_{t_-}^-)=W_j^{(l)}(\bY_{t_-}^+)$ for each $r< j \le n$.
So we can define $\Gamma^{(l)}$ as
\[
\begin{split}
d\Gamma^{(l)}
&=
\frac{1}{2}
|\sinh(2\zeta^{\bY_{t_-}^-}_{l})|
C_l^{(l)}(\bY_{t_-}^-)C_r^{(l)}(\bY_{t_+}^-)
\prod_{1\le j < l} U_j^{(l)}(\bY_{t_-}^-)
\prod_{r< j \le n} W_j^{(l)}(\bY_{t_+}^-)
d\Gamma \\
&=
\frac{1}{2}
|\sinh(2\zeta^{\bY_{t_-}^+}_{l})|
C_l^{(l)}(\bY_{t_-}^+)C_r^{(l)}(\bY_{t_+}^+)
\prod_{1\le j < l} U_j^{(l)}(\bY_{t_-}^+)
\prod_{r< j \le n} W_j^{(l)}(\bY_{t_+}^+)
d\Gamma.
\end{split}
\]

\vspace{2mm}

\noindent\textbf{From $\Gamma^{(\ell)}$ to $\Gamma^{(\ell+1)}$.}
For $l\le \ell < r$, we assume that we have a measure $\Gamma^{(\ell)}$ on the space of $\JJ^2$, such that it is a coupling of $\mu^{(\ell)}_\pm$.
In fact, we can make the following stronger assumption of $\Gamma^{(\ell)}$.
Let $\FF^\pm_\ell$ be the sigma algebra of $\JJ$, generated by $(\tau^\pm_v)_{v\in\GG^{(\ell)}}, (B^\pm_t(v))_{t\ge 0, v\in\GG^{(\ell)}}$; and $\FF_\ell$ be the sigma algebra of $\JJ^2$, generated by $\{S_+\times \JJ:S_+\in\FF^+_\ell\} \cup  \{\JJ\times S_-:S_-\in\FF^-_\ell\}$.
Then $\Gamma^{(\ell)}$ satisfies the following conditions:
\begin{itemize}
    \item For any $S_+\in\FF^+_\ell$ we have $\Gamma^{(\ell)}(S_+\times\JJ) = \mu_+^{(\ell)}(S_+)$;
and for any $S_-\in\FF^-_\ell$ we have $\Gamma^{(\ell)}(\JJ\times S_-) = \mu_-^{(\ell)}(S_-)$.
    \item Letting $\oGamma^{(\ell)}:=\frac{\Gamma^{(\ell)}}{|\Gamma^{(\ell)}|}$, which is a probability measure,
the conditional distribution $\oGamma^{(\ell)}(\cdot\mid \FF_\ell)$ is the same as $\Gamma(\cdot\mid \FF_\ell )$.
\end{itemize}
Now we construct $\Gamma^{(\ell+1)}$. Denote
\begin{equation} \label{eq:defn-R}
R_\ell^\pm := 
\prod_{i=0}^{k+1}\frac{\oZ_{u_i,u_{i-1}}^{(\ell)}(\bY_{t_i}^\pm)^2}
{\oZ_{u_i,u_{i-1}}^{(\ell+1)}(\bY_{t_i}^\pm)^2}
\frac{C_l^{(\ell+1)}(\bY_{t_-}^\pm)}{C_l^{(\ell)}(\bY_{t_-}^\pm)}
\frac{C_r^{(\ell+1)}(\bY_{t_-}^\pm)}{C_r^{(\ell)}(\bY_{t_-}^\pm)}
\prod_{1\le j <l}
\frac{U_j^{(\ell+1)}(\bY_{t_-}^\pm)}{U_j^{(\ell)}(\bY_{t_-}^\pm)}
\prod_{r< j \le n}
\frac{W_j^{(\ell+1)}(\bY_{t_-}^\pm)}{W_j^{(\ell)}(\bY_{t_-}^\pm)}
.    
\end{equation}
Then we have $d\mu^{(\ell+1)}_\pm = R_\ell^\pm d\mu^{(\ell)}_\pm$; i.e., $R_\ell^\pm$ is the `reweight' from $\mu^{(\ell)}_\pm$ to $\mu^{(\ell+1)}_\pm$.
We also have that $R_\ell^\pm$ is $\FF_{\ell+1}^{\pm}$ measurable, and also $\FF_{\ell+1}$ measurable.

Let $P_\ell = R_\ell^- \wedge R_\ell^+$, and define $\Gamma^{(\ell+1)}$ as the measure with $d\Gamma^{(\ell+1)}=P_\ell d\Gamma^{(\ell)}+d\Xi^{(\ell)}$, where $\Xi^{(\ell)}$ is a measure on
$\JJ^2$, given by the following conditions.
For any $S_+\in \FF^+_{\ell+1}$ and $S_-\in \FF^-_{\ell+1}$, we have
\[
\Xi^{(\ell)} (S_+\times \JJ) = \int (R_\ell^+ - P_\ell)\don[S_+\times \JJ] d\Gamma^{(\ell)},\quad
\Xi^{(\ell)} (\JJ\times S_-) = \int (R_\ell^- - P_\ell)\don[\JJ\times S_-] d\Gamma^{(\ell)}.
\]
We also require that the conditional distribution $\oXi^{(\ell)}(\cdot\mid \FF_{\ell+1})$ is the same as $\Gamma(\cdot\mid \FF_{\ell+1} )$, where $\oXi^{(\ell)}:=\frac{\Xi^{(\ell)}}{|\Xi^{(\ell)}|}$ is a probability measure on $\JJ^2$.
Such $\Xi^{(\ell)}$ exists (although not unique) since $R_\ell^+ - P_\ell, R_\ell^- - P_\ell\ge 0$, and
\[
\int (R_\ell^+ - P_\ell) d\Gamma^{(\ell)} = \int R_\ell^+ d\mu^{(\ell)}_+ -\int P_\ell d\Gamma^{(\ell)} = \int R_\ell^- d\mu^{(\ell)}_- -\int P_\ell d\Gamma^{(\ell)} = \int (R_\ell^- - P_\ell) d\Gamma^{(\ell)}.
\]

We now check $\Gamma^{(\ell+1)}$ satisfies the desired conditions.
From the definitions, for any $S_+\in\FF^+_{\ell+1}$ we have
\[
\begin{split}
\Gamma^{(\ell+1)}(S_+\times\JJ)
=&
\int P_\ell \don[S_+\times\JJ] d\Gamma^{(\ell)} + \Xi^{(\ell)}(S_+\times\JJ)
=
\int R^+_\ell \don[S_+\times\JJ] d\Gamma^{(\ell)}
\\
=&
\int R^+_\ell \don[S_+] d\mu_+^{(\ell)}
=
\int \don[S_+] d\mu_+^{(\ell+1)}
=
\mu_+^{(\ell+1)}(S_+).
\end{split}
\]
Similarly,  for any $S_-\in\FF^-_{\ell+1}$ we have $\Gamma^{(\ell+1)}(\JJ\times S_-) = \mu_-^{(\ell+1)}(S_+)$.

Also, since $P_\ell$ is $\FF_{\ell+1}$ measurable, and $\oXi^{(\ell)}(\cdot\mid \FF_{\ell+1})$ and $\oGamma^{(\ell)}(\cdot\mid \FF_{\ell+1})$ are the same as $\Gamma(\cdot\mid \FF_{\ell+1} )$, we have that $\oGamma^{(\ell+1)}(\cdot\mid \FF_{\ell+1})$ is the same as $\Gamma(\cdot\mid \FF_{\ell+1} )$ as well.
By principle of induction, we have defined $\Gamma^{(\ell)}$ for any $l\le \ell \le r$.

Finally, we just let $\Gamma^*=\Gamma^{(r)}$.

\subsection{Key estimates on coupling and inductive expansion}\label{subsec:ind-coup:estims}
It now remains to bound \eqref{eq:couple-bound} for the $\Gamma^*$ we constructed.
For this we need the following key estimate, to state which we set up some further notations.

Take $E_\ell = |R^+_\ell-R^-_\ell|$ for $l\le \ell < r$, where $R^\pm_\ell$ are the reweights \eqref{eq:defn-R}.
For any vertex $u\in \TT_R$, let $\rho(u)\in\{l,\ldots,r\}$, such that $u\in\GG^{(\rho(u))}$ but $u\not\in\GG^{(\rho(u)-1)}$ (assuming that $\GG^{(l-1)}=\emptyset$).
In other words, $\rho(u)=\argmin_{l\le\ell\le r}\bd(v_\ell, u)$; i.e. $\rho(u)$ is the `projection' of $u$ onto $[v_l, v_r]$.
For any $u,v \in \TT_R$, and $l\le a\le b \le r$, we let $\bbd_a^b(u, v)$ be the distance between the intervals $[a, b]$ and $[\rho(u), \rho(v)]$ (if $\rho(u)\le\rho(v)$) or $[\rho(v), \rho(u)]$ (if $\rho(v)\le\rho(u)$).
We also denote $\bbd_a(u,v):=\bbd_a^a(u,v)$.
See Figure \ref{fig:indu-G} for an illustration of these notations.
Recall that $P_\ell = R_\ell^- \wedge R_\ell^+$ for $l\le \ell < r$.
\begin{proposition}  \label{prop:key-esti}
Let $l-1\le a < b < r$, and $f$ be a non-negative function of $(B^\pm_t(v))_{t\ge 0, v\not\in\GG^{(a+1)}}$, and let $\Lambda$ be any probability measure on $\JJ^2$, such that $\Lambda(\cdot\mid \FF_{a+1}) = \Gamma(\cdot\mid \FF_{a+1})$.
Then we have
\begin{equation}  \label{eq:key-est1}
\int fE_b\prod_{a<\ell<b}P_\ell d\Lambda
<
M_a^b \E_\Gamma[f],
\end{equation}
\begin{equation}  \label{eq:key-est2}
\int f\don[\tau_{v_{b+1}}^- \neq \tau_{v_{b+1}}^+]\prod_{a<\ell\le b}P_\ell d\Lambda<M_a^b \E_\Gamma[f], 
\end{equation}
where
\[
M_a^b:=C(2\theta)^{b-a}
\prod_{i=0}^{k+1}
(1 + C(2\theta)^{\bbd_{a+1}^{b+1}(u_i,u_{i-1})})
(1+C(2\theta)^{a+2-l})^{l}
(1+C(2\theta)^{r-b-1})^{n-r+1}
\]
and $C$ is an absolute constant.
In particular, this implies that we have
\begin{equation} \label{eq:key-est3}
\begin{split}
  \left.\begin{aligned}
  \int fE_b\prod_{a<\ell<b}P_\ell d\Xi^{(a)} & \\
  \int f\don[\tau_{v_{b+1}}^- \neq \tau_{v_{b+1}}^+]\prod_{a<\ell\le b}P_\ell d\Xi^{(a)} &
  \end{aligned}\right\rbrace    
  &<M_a^b \E_\Gamma[f] \Xi^{(a)}(\JJ^2) \\& \le M_a^b \E_\Gamma[f] \int E_{a}d\Gamma^{(a)} = M_a^b\int fE_{a}d\Gamma^{(a)},
\end{split}
\end{equation}
when $a\ge l$; and similarly
\begin{equation} \label{eq:key-est5}
  \left.\begin{aligned}
  \int fE_b\prod_{a<\ell<b}P_\ell d\Gamma^{(a+1)} & \\
  \int f\don[\tau_{v_{b+1}}^- \neq \tau_{v_{b+1}}^+]\prod_{a<\ell\le b}P_\ell d\Gamma^{(a+1)} &
  \end{aligned}\right\rbrace    
  <M_a^b \E_\Gamma[f] \Gamma^{(a+1)}(\JJ^2)  = M_a^b\int fd\Gamma^{(a+1)}.
\end{equation}
Here we used that $\oXi^{(a)}(\cdot\mid \FF_{a+1})$ (when $a\ge l$), $\oXi^{(a+1)}(\cdot\mid \FF_{a+1})$ and $\oGamma^{(a)}(\cdot\mid \FF_{a+1} )$ are the same as $\Gamma(\cdot\mid \FF_{a+1} )$,
and that $E_a$ is $\FF_{a+1}$ measurable.
\end{proposition}
We leave the proof of this proposition to Appendix \ref{sec:appb}.
Now we use it to prove Proposition \ref{p:chain-reformulate}, by expansion on the inductive coupling.

\begin{proof}[Proof of Proposition \ref{p:chain-reformulate}]
From the construction of $\Gamma^*=\Gamma^{(r)}$ we need to bound \eqref{eq:couple-bound}.
Denote
\[
\LL:= 
2e^{4t_{+}+4(d-1)\beta + 2|B^+_{t_{+}}(v_r)|+2|B^+_{t_{+}}(v_n)|}.
\]
Under $\Gamma^{(r)}$, when $\tau_{v_{r}}^- = \tau_{v_{r}}^+$ we must have that $2\zeta^{\bY_{t_+}^-}_{r}=2\zeta^{\bY_{t_+}^+}_{r}$ and $2\zeta^{\bY_{t_+}^-}_{n}=2\zeta^{\bY_{t_+}^+}_{n}$.
Thus we have the bound
\begin{equation} \label{eq:hypo-to-diff}
\int \sinh(2\zeta^{\bY_{t_+}^+}_{r}) \cosh(2\zeta^{\bY_{t_+}^+}_{n})
- \sinh(2\zeta^{\bY_{t_+}^-}_{r}) \cosh(2\zeta^{\bY_{t_+}^-}_{n})
d\Gamma^{(r)}
\le
\int \LL \don[\tau_{v_{r}}^- \neq \tau_{v_{r}}^+] d\Gamma^{(r)}.    
\end{equation}
Now we inductively expand the RHS using the relation $d\Gamma^{(\ell+1)}=P_\ell d\Gamma^{(\ell)}+d\Xi^{(\ell)}$ for $l\le \ell < r$, and Proposition \ref{prop:key-esti}.
Recall that $E_\ell=|R_\ell^- - R_\ell^+|$ and $P_\ell = R_\ell^- \wedge R_\ell^+$, and $R^\pm_\ell$ are the reweights \eqref{eq:defn-R}, for any $l\le \ell < r$.
We first assume that $r>l$, 
and in this case we claim that for each $1\le m\le r-l$,
\begin{equation}\label{eq:induct-expan-step}
\begin{split}  
\int \LL \don[\tau_{v_{r}}^- \neq \tau_{v_{r}}^+] d\Gamma^{(r)}
&\le
\int \LL\don[\tau_{v_{r}}^- \neq \tau_{v_{r}}^+]\prod_{r-m\le \ell< r}P_\ell d\Gamma^{(r-m)}
+
\int \LL E_{r-1}\prod_{r-m\le \ell< r-1}P_\ell d\Gamma^{(r-m)}
\\
& \ \ + 
\sum_{s\ge 2}\sum_{r-m\le a_1<\cdots<a_s=r-1}
\prod_{i=1}^{s-1} M_{a_i}^{a_{i+1}} \int 2\LL E_{a_1}\prod_{r-m\le \ell < a_1}P_\ell
d\Gamma^{(r-m)}    .
\end{split}    
\end{equation}
We prove this by induction in $m$.
For $m=1$, we have 
\[
\begin{split}
\int \LL \don[\tau_{v_{r}}^- \neq \tau_{v_{r}}^+] d\Gamma^{(r)}&
=
\int \LL \don[\tau_{v_{r}}^- \neq \tau_{v_{r}}^+]P_{r-1} d\Gamma^{(r-1)}
+
\int \LL \don[\tau_{v_{r}}^- \neq \tau_{v_{r}}^+] d\Xi^{(r-1)} \\
&
\le \int \LL \don[\tau_{v_{r}}^- \neq \tau_{v_{r}}^+]P_{r-1} d\Gamma^{(r-1)}
+
\int \LL E_{r-1} d\Gamma^{(r-1)},
\end{split}
\]
where the inequality is due to the following reason. 
Note that $\LL$ depends only on $B^+_{t_{+}}(v_r)$ and $B^+_{t_{+}}(v_n)$, and that $d\mu_+^{(r)}=R_{r-1}^+ d\mu_+^{(r-1)}$,
so we have $\int \LL d\Gamma^{(r)}=\int \LL R_{r-1}^+ d\Gamma^{(r-1)}$, and $\int \LL  d\Xi^{(r-1)}
=
\int \LL d\Gamma^{(r)}- \int \LL P_{r-1} d\Gamma^{(r-1)}
=
\int \LL (R_{r-1}^+-P_{r-1}) d\Gamma^{(r-1)}$.
Thus we have 
\[\int \LL \don[\tau_{v_{r}}^- \neq \tau_{v_{r}}^+] d\Xi^{(r-1)} 
\le 
\int \LL  d\Xi^{(r-1)}
=
\int \LL (R_{r-1}^+-P_{r-1}) d\Gamma^{(r-1)} \le \int \LL E_{r-1} d\Gamma^{(r-1)}.\]
Now we assume that \eqref{eq:induct-expan-step} holds for some $1\le m <r-l$. We study each term in the RHS.
By $d\Gamma^{(r-m)}=P_{r-m-1} d\Gamma^{(r-m-1)}+d\Xi^{(r-m-1)}$ and \eqref{eq:key-est3}, we have
\[
\begin{split}
&\int \LL\don[\tau_{v_{r}}^- \neq \tau_{v_{r}}^+]\prod_{r-m\le \ell< r}P_\ell d\Gamma^{(r-m)}\\
=&
\int \LL\don[\tau_{v_{r}}^- \neq \tau_{v_{r}}^+]\prod_{r-m-1\le \ell< r}P_\ell d\Gamma^{(r-m-1)}
+
\int \LL\don[\tau_{v_{r}}^- \neq \tau_{v_{r}}^+]\prod_{r-m\le \ell< r}P_\ell d\Xi^{(r-m-1)}\\
\le &
\int \LL\don[\tau_{v_{r}}^- \neq \tau_{v_{r}}^+]\prod_{r-m-1\le \ell< r}P_\ell d\Gamma^{(r-m-1)}
+
M_{r-m-1}^{r-1}\int \LL E_{r-m-1} d\Gamma^{(r-m-1)},
\end{split}
\]
and by \eqref{eq:key-est3},
\[
\begin{split}
&\int \LL E_{r-1}\prod_{r-m\le \ell< r-1}P_\ell d\Gamma^{(r-m)}\\
=&
\int \LL E_{r-1}\prod_{r-m-1\le \ell< r-1}P_\ell d\Gamma^{(r-m-1)}
+
\int \LL E_{r-1}\prod_{r-m\le \ell< r-1}P_\ell d\Xi^{(r-m-1)}\\
\le &
\int \LL E_{r-1}\prod_{r-m-1\le \ell< r-1}P_\ell d\Gamma^{(r-m-1)}
+
M_{r-m-1}^{r-1}\int \LL E_{r-m-1} d\Gamma^{(r-m-1)},
\end{split}
\]
and for each $s\ge 2$ and $r-m\le a_1<\cdots<a_s=r-1$, by \eqref{eq:key-est3},
\[
\begin{split}
&\prod_{i=1}^{s-1} M_{a_i}^{a_{i+1}} \int 2\LL E_{a_1}\prod_{r-m\le \ell < a_1}P_\ell
d\Gamma^{(r-m)}\\
=&
\prod_{i=1}^{s-1} M_{a_i}^{a_{i+1}} \int 2\LL E_{a_1}\prod_{r-m-1\le \ell < a_1}P_\ell
d\Gamma^{(r-m-1)}
+
\prod_{i=1}^{s-1} M_{a_i}^{a_{i+1}} \int 2\LL E_{a_1}\prod_{r-m\le \ell < a_1}P_\ell
d\Xi^{(r-m-1)}\\
\le &
\prod_{i=1}^{s-1} M_{a_i}^{a_{i+1}} \int 2\LL E_{a_1}\prod_{r-m-1\le \ell < a_1}P_\ell
d\Gamma^{(r-m-1)}
+
M_{r-m-1}^{a_1}\prod_{i=1}^{s-1} M_{a_i}^{a_{i+1}} \int 2\LL E_{r-m-1}
d\Gamma^{(r-m-1)}.
\end{split}
\]
Summing up these inequalities we get \eqref{eq:induct-expan-step} for $m+1$.
Thus \eqref{eq:induct-expan-step} holds for all $1\le m \le r-l$.
Take $m=r-l$ in \eqref{eq:induct-expan-step},
we get 
\[
\begin{split}  
\int \LL \don[\tau_{v_{r}}^- \neq \tau_{v_{r}}^+] d\Gamma^{(r)}
&\le
\int \LL\don[\tau_{v_{r}}^- \neq \tau_{v_{r}}^+]\prod_{l\le \ell< r}P_\ell d\Gamma^{(l)}+
\int \LL E_{r-1}\prod_{l\le \ell< r-1}P_\ell d\Gamma^{(l)}
\\
& \ \ + 
\sum_{s\ge 2}\sum_{l\le a_1<\cdots<a_s=r-1}
\prod_{i=1}^{s-1} M_{a_i}^{a_{i+1}} \int 2\LL E_{a_1}\prod_{l\le \ell < a_1}P_\ell
d\Gamma^{(l)}    .
\end{split}    
\]
Apply \eqref{eq:key-est5} to the terms in the RHS, we get
\begin{equation}  \label{eq:L-diff-bd-final}
\int \LL \don[\tau_{v_{r}}^- \neq \tau_{v_{r}}^+] d\Gamma^{(r)}
\le
\sum_{s\ge 2}\sum_{l-1= a_1<\cdots<a_s=r-1}
\prod_{i=1}^{s-1} M_{a_i}^{a_{i+1}} \int 2\LL 
d\Gamma^{(l)}    .    
\end{equation}
We consider the summation:
there are $2^{r-l-1}$ terms; and for each $s\ge 2$ and $l-1= a_1<\cdots<a_s=r-1$, we have 
\[
\begin{split}
\prod_{i=1}^{s-1} M_{a_i}^{a_{i+1}}
&=
C^{s-1}(2\theta)^{r-l}
\prod_{i=1}^{s-1}
\prod_{j=0}^{k+1}(1 + C(2\theta)^{\bbd_{a_i+1}^{a_{i+1}+1}(u_j,u_{j-1})})
(1+C(2\theta)^{a_i+2-l})^{l}
(1+C(2\theta)^{r-a_{i+1}-1})^{n-r+1}
\\
&\le 
(2C\theta)^{r-l}
\prod_{j=0}^{k+1} C'^{1+\bd(u_j,u_{j-1})}
C'^{l+n-r+1}
\end{split}
\]
for some absolute constant $C'$.
By plugging this into \eqref{eq:L-diff-bd-final} we get
\begin{equation}  \label{eq:L-diff-bd-final1}
\int \LL \don[\tau_{v_{r}}^- \neq \tau_{v_{r}}^+] d\Gamma^{(r)}
\le
2^{r-l}(2C\theta)^{r-l}
\prod_{j=0}^{k+1} C'^{1+\bd(u_j,u_{j-1})}
C'^{l+n-r+1} \int \LL 
d\Gamma^{(l)}  .    
\end{equation}
In the case where $l=r$, we have $\int \LL \don[\tau_{v_{r}}^- \neq \tau_{v_{r}}^+] d\Gamma^{(r)}
= \int \LL 
d\Gamma^{(l)}$, so \eqref{eq:L-diff-bd-final1} also holds.
Now it remains to bound $\int \LL 
d\Gamma^{(l)}$.
For any $\bx\in\R^{\TT_R}$, we have $C_l^{(l)}(\bx), C_r^{(l)}(\bx) \le \cosh(2\beta)$, by \eqref{eq:boundutov}.
Then with \eqref{eq:bound-U} and \eqref{eq:bound-W} we have
\[
\begin{split}
\int \LL d\Gamma^{(l)} = &\int 
\frac{1}{2}
|\sinh(2\zeta^{\bY_{t_-}^+}_{l})|
C_l^{(l)}(\bY_{t_-}^+)C_r^{(l)}(\bY_{t_+}^+)
\prod_{1\le j < l} U_j^{(l)}(\bY_{t_-}^+)
\prod_{r< j \le n} W_j^{(l)}(\bY_{t_+}^+) \LL d\Gamma^{(l)}
\\
<&
\int |\sinh(2\zeta^{\bY_{t_-}^+}_{l})| \cosh(2\zeta^{\bY_{t_-}^+}_{1})
\cosh^2(2\beta)
(4\cosh^2(\beta)\cosh(2\beta))^{n-r+l-1}
\\
& \times
e^{4t_{+}+4(d-1)\beta + 2|B^+_{t_{+}}(v_r)|+2|B^+_{t_{+}}(v_n)|}
d\Gamma    \\
< &
\int 
e^{4t_-+4t_{+}+8(d-1)\beta +
2|B^+_{t_{-}}(v_1)|+2|B^+_{t_{-}}(v_l)| +
2|B^+_{t_{+}}(v_r)|+2|B^+_{t_{+}}(v_n)|}
(2\cosh(2\beta))^{2n}
d\Gamma\\
<&
C''^{2n} e^{17t+d}
\end{split}
\]
for some absolute constant $C''$.
Using this and \eqref{eq:hypo-to-diff}, \eqref{eq:L-diff-bd-final1}, 
we conclude that
\[
\begin{split}
&
\E\left[
\frac{
\sinh(2\zeta^{\bY_{t_-}}_{l})
\sinh(2\zeta^{\bY_{t_+}}_{r})
C_l(\bY_{t_-}) C_r(\bY_{t_+})
}
{
\prod_{i=0}^{k+1}\oZ_{u_i,u_{i-1}}(\bY_{t_i})^2}
\prod_{1\le j < l} U_j(\bY_{t_-})
\prod_{r < j \le n} V_j(\bY_{t_+})
\right]    
\\
\le & 
C''^{2n} e^{17t+d}
2^{r-l}(2C\theta)^{r-l}
\prod_{j=0}^{k+1} C'^{1+\bd(u_j,u_{j-1})}
C'^{l+n-r+1}
\end{split}
\]
which gives the desired bound.
\end{proof}

\section*{Acknowledgment}
AS would like to thank Russ Lyons and Yuval Peres for many discussions about the problem.  DN is supported by a Samsung Scholarship. AS is supported by NSF grants DMS-1352013 and DMS-1855527, Simons Investigator grant and a MacArthur Fellowship.
The authors also thank anonymous referees for reading this paper carefully and providing valuable comments.

\bibliography{fiidref}

\newpage

\appendix

\section{Computations of the covariances} \label{sec:appa}
We prove Lemma \ref{L:cov-formula} in this appendix.
\begin{proof}[Proof of Lemma \ref{L:cov-formula}]
We first compute $\cov{\sigma_u}{\sigma_v}_{\pi^{\bx}}$.
For any $\sigma,\sigma'\in\{\pm 1\}$, we denote $L_{\sigma,\sigma'}=\tZ_{[u,v]}^{\{u\mapsto \sigma, v\mapsto \sigma'\}}$.
Then $\cov{\sigma_u}{\sigma_v}_{\pi^{\bx}}$ equals
\[
\begin{split}
&\iprod{\sigma_u\sigma_v}_{\pi^{\bx}}
-
\iprod{\sigma_u}_{\pi^{\bx}}\iprod{\sigma_v}_{\pi^{\bx}}\\
=&
\frac{L_{1,1}+L_{-1,-1}-L_{1,-1}-L_{-1,1}}{\tZ_{u,v}(\bx)}
-\frac{L_{1,1}-L_{-1,-1}+L_{1,-1}-L_{-1,1}}{\tZ_{u,v}(\bx)}
\frac{L_{1,1}-L_{-1,-1}-L_{1,-1}+L_{-1,1}}{\tZ_{u,v}(\bx)}
\\
=&
\frac{4L_{1,1}L_{-1,-1}-4L_{1,-1}L_{-1,1}}{\tZ_{u,v}(\bx)^2}.
\end{split}
\]
We claim that $4L_{1,1}L_{-1,-1}-4L_{1,-1}L_{-1,1}$ is independent of $\bx$.
Then we can get the conclusion by recalling that $\cov{\sigma_u}{\sigma_v}_{\pi} = A_{u,v}$.
We can expand it as
\[
\begin{split}
&\sum_{\substack{(\sigma_{v'})_{v'\in[u,v]}, (\osigma_{v'})_{v'\in[u,v]}\\ \sigma_u=\sigma_v=1, \osigma_u=\osigma_v=-1}} \exp\left( \sum_{\substack{(u',v')\in E_R,\\ u',v'\in[u,v]}} \beta_{u',v'}(\sigma_{u'}\sigma_{v'}+\osigma_{u'}\osigma_{v'}) + \sum_{v'\in[u,v]} (x(v')+\zeta_\GG^\bx(v'))(\sigma_{v'}+\osigma_{v'}) \right)
\\
&-
\sum_{\substack{(\sigma_{v'})_{v'\in[u,v]}, (\osigma_{v'})_{v'\in[u,v]}\\ \sigma_u=\osigma_v=1, \osigma_u=\sigma_v=-1}} \exp\left( \sum_{\substack{(u',v')\in E_R,\\ u',v'\in[u,v]}} \beta_{u',v'}(\sigma_{u'}\sigma_{v'}+\osigma_{u'}\osigma_{v'}) + \sum_{v'\in[u,v]} (x(v')+\zeta_\GG^\bx(v'))(\sigma_{v'}+\osigma_{v'}) \right).
\end{split}
\]
Take any $(\sigma_{v'})_{v'\in[u,v]}$, $(\osigma_{v'})_{v'\in[u,v]}$ such that $\sigma_u=\sigma_v=1, \osigma_u=\osigma_v=-1$, let $w\in[u,v]$ be the vertex with the smallest $\bd(w,u)$, such that $\sigma_w=\osigma_w$ (if such vertex exists).
We then exchange $\sigma_{v'}$ and $\osigma_{v'}$ for all $v'\in [w,v]$.
Then we obtain some $(\sigma_{v'})_{v'\in[u,v]}, (\osigma_{v'})_{v'\in[u,v]}$ such that $\sigma_u=\osigma_v=1, \osigma_u=\sigma_v=-1$;
and this construction is a bijection for those $(\sigma_{v'})_{v'\in[u,v]}$, $(\osigma_{v'})_{v'\in[u,v]}$ where $w$ exists.
Thus in the above sums we only need to consider $(\sigma_{v'})_{v'\in[u,v]}$, $(\osigma_{v'})_{v'\in[u,v]}$ where $\sigma_{v'}\neq \osigma_{v'}$ for each $v'\in[u,v]$; and these terms are independent of $\bx$.

For $\cov{\sigma_u\sigma_{u'}}{\sigma_v}_{\pi^{\bx}}$ with $v\neq u'$, we begin with an observation on the conditional expectation of $\sigma_{u'}$ given $\sigma_u$. The following identity can be verified from a straight-forward computation and we omit the detail:
	\begin{equation*}
	\mathbb{E}_{\pi^{\bx}} [ \sigma_{u'} | \sigma_{u}] = \frac{\sinh(2x(u')) + \sinh(2\gamma) \sigma_u }{\cosh(2x(u'))+\cosh(2\gamma)}.
	\end{equation*}
 This implies that
\begin{equation*}
\begin{split}
\cov{\sigma_u\sigma_{u'}}{\sigma_v}_{\pi^\bx}& = \mathbb{E}_{\pi^{\bx}} \left[\sigma_u\sigma_v \E_{\pi^{\bx}}[\sigma_{u'}|\sigma_u] \right]
-
\E_{\pi^{\bx}} \left[ \sigma_u \E_{\pi^{\bx}}[\sigma_{u'}|\sigma_u] \right] \E_{\pi^{\bx}}[\sigma_v]\\
&= 
\frac{\sinh(2x(u')) }{\cosh(2x(u'))+\cosh(2\gamma)} \cov{\sigma_u}{\sigma_v}_{\pi^{\bx}}.
\end{split}
\end{equation*}
Moreover, note that $Z_{u',v}(\boo) = 2\cosh(\gamma) Z_{u,v}(\boo)$. Thus, to get \eqref{eq:cov:threepts:nonendptcase}, from \eqref{eq:cov:twopts} it suffices to show that 
\begin{equation}\label{eq:cov:threepts:goal}
\tZ_{u',v}(\bx)^2 =2 \tZ_{u,v}(\bx)^2 (\cosh(2x(u'))+\cosh(2\gamma)).
\end{equation}
To this end, we write $\zeta_{u'\to u}^\bx$ as
\begin{equation}\label{eq:cov:threepts:rewrite}
\zeta_{u'\to u}^\bx = \frac{1}{2}\log \left(\frac{\cosh(x(u')+\gamma)}{\cosh(x(u')-\gamma)} \right),
\end{equation}
since $u'$ is a leaf of $\TT_R$.
Furthermore, recall the definition of $L_{\sigma, \sigma'}$ above, and set 
\[
L_1 := L_{1, 1} + L_{1,-1}, \quad L_{-1} := L_{-1,1} + L_{-1,-1}.
\]
Then, we can decompose $\tZ_{u',v}(\bx)$ into two cases of either $\sigma_u = 1$ or $\sigma_u = -1$, and express it as follows.
\begin{equation}\label{eq:cov:threepts:1}
\begin{split}
\tZ_{u',v}(\bx) &= L_1 e^{-\zeta_{u'\to u}^\bx} \left(e^{x(u')+\gamma} + e^{-x(u')-\gamma}\right) 
+ L_{-1} e^{\zeta_{u'\to u}^\bx} \left( e^{x(u')-\gamma} + e^{-x(u') + \gamma} \right) \\
&=2(L_1 + L_{-1}) \left( \cosh(x(u')+\gamma) \cosh(x(u')-\gamma) \right)^{1/2}.
\end{split}
\end{equation}
In the first identity, we reweighted $L_1$ (resp.~$L_{-1}$) by $e^{-\zeta_{u'\to u}^\bx}$ (resp.~$e^{\zeta_{u'\to u}^\bx}$) since $L_1$ and $L_{-1}$ are the partition functions including the effect of the induced field $\zeta_{u'\to u}^\bx$ already. The second equality uses \eqref{eq:cov:threepts:rewrite}.
We then obtain the conclusion by noticing that \eqref{eq:cov:threepts:goal} is equivalent to \eqref{eq:cov:threepts:1}. 

Finally, for $\cov{\sigma_u\sigma_{u'}}{\sigma_v}_{\pi^{\bx}}$ with $v\neq u'$, we remark that \eqref{eq:cov:threepts:endptcase} can be obtained from an analogous calculation, where the only difference is that we need to include the induced external field $\zeta_{[u,u']}^{\bx}(u)$ rather than $\zeta_{u'\to u}^\bx$.
\end{proof}

\section{Computations for the key estimate} \label{sec:appb}
This appendix is devoted to proving Proposition \ref{prop:key-esti}.
Recall the setup from Section \ref{subsec:ind-coup:prelim}: we work on $\TT_R$ for fixed $R$, and let $\bbeta \in \R^{E_R}$ be the inverse temperature, such that $\beta_{u, v}=\beta$ if $(u, v) \not\in \partial E_R$, and $\beta_{u, v}=\gamma$ if $(u, v) \in \partial E_R$. Also recall the Belief Propagation message $m_{u\to v}^\bx$ and induced external field $\zeta_{u\to v}^\bx$ for $(u,v)\in E(\TT_R)$, and $\theta=\tanh\beta$.

Much of our computation is based on the following result on the induced external field, which will be repeatedly used.
\begin{lemma}  \label{lemma:bp-eq}
Let $\bx_-, \bx_+\in \R^{\TT_R}$,
and $(u,v)\in E(\TT_R)$.
Then we have
\[
\tanh(|\zeta_{u\to v}^{\bx_-}-\zeta_{u\to v}^{\bx_+}|)
\le 
2\theta \tanh\left(|x_-(u)-x_+(u)+\sum_{w\sim u, w\neq v} \zeta_{w\to u}^{\bx_-}-\zeta_{w\to u}^{\bx_+} |\right).
\]
\end{lemma}
\begin{proof}
Note that $\tanh(\zeta_{u\to v}^{\bx_\pm}) = \tanh(\beta_{u,v})\tanh(x_\pm(u)+\sum_{w\sim u, w\neq v} \zeta_{w\to u}^{\bx_\pm})$.
Then we have
\[
\begin{split}
\tanh&(|\zeta_{u\to v}^{\bx_-}-\zeta_{u\to v}^{\bx_+}|)
=
\frac{|\tanh(\zeta_{u\to v}^{\bx_-})-\tanh(\zeta_{u\to v}^{\bx_+})|}{1-\tanh(\zeta_{u\to v}^{\bx_-})\tanh(\zeta_{u\to v}^{\bx_+})}
\\
&=
\frac{\tanh(\beta_{u,v})|\tanh(x_-(u)+\sum_{w\sim u, w\neq v} \zeta_{w\to u}^{\bx_-})-\tanh(x_+(u)+\sum_{w\sim u, w\neq v} \zeta_{w\to u}^{\bx_+})|}{1-\tanh(x_-(u)+\sum_{w\sim u, w\neq v} \zeta_{w\to u}^{\bx_-})\tanh(x_+(u)+\sum_{w\sim u, w\neq v} \zeta_{w\to u}^{\bx_+})\tanh^2(\beta_{u,v})}
\\
&\le
\frac{2\theta|\tanh(x_-(u)+\sum_{w\sim u, w\neq v} \zeta_{w\to u}^{\bx_-})-\tanh(x_+(u)+\sum_{w\sim u, w\neq v} \zeta_{w\to u}^{\bx_+})|}{1-\tanh(x_-(u)+\sum_{w\sim u, w\neq v} \zeta_{w\to u}^{\bx_-})\tanh(x_+(u)+\sum_{w\sim u, w\neq v} \zeta_{w\to u}^{\bx_+})}
\\
&=
2\theta \tanh\left(|x_-(u)-x_+(u)+\sum_{w\sim u, w\neq v} \zeta_{w\to u}^{\bx_-}-\zeta_{w\to u}^{\bx_+} |\right).
\end{split}
\]
Here we used the basic fact that $\tanh(|a-b|)=\frac{|\tanh(a)-\tanh(b)|}{1-\tanh(a)\tanh(b)}$ for any $a,b\in \R$, and $\tanh(\beta_{u,v})\le \theta$, and the basic fact that $\frac{1}{1-ab}\le \frac{2}{1-a}$ for any $a\in(-1,1)$ and $b\in [0,1)$.
\end{proof}

\subsection{Bounds for $C_l, C_r,  U_j, W_j$ terms}
In this subsection, we provide main estimates for the quantities $C_l, C_l^{(a)}, C_r, C_r^{(a)}, U_j, U_j^{(a)}, W_j$ and $W_j^{(a)}$ introduced in Sections \ref{subsec:ind-coup:covchain} and \ref{subsec:ind-coup:ind-coup}.

\begin{lemma}\label{lemma:key-est-UVC}
Take any $l\le a \le r$ and vector $\bx_-,\bx_+\in\R^{\TT_R}$.
If $x_-(v)=x_+(v)$ for any $v\in\GG^{(a)}$, we must have
\[
1 - C(2\theta)^{a+1-l} < 
\frac{C_l(\bx_-)}{C_l(\bx_+)},\;
\frac{U_{j}(\bx_-)}{U_{j}(\bx_+)}
< 1 + C(2\theta)^{a+1-l},
\]
for any $1\le j < l$.
If $x_-(v)=x_+(v)$ for any $v\not\in\GG^{(a-1)}$, we must have
\[
1 - C(2\theta)^{r-a+1} < 
\frac{C_r(\bx_-)}{C_r(\bx_+)},\;
\frac{W_{j}(\bx_-)}{W_{j}(\bx_+)}
< 1 + C(2\theta)^{r-a+1},
\]
for any $r< j \le n$.
\end{lemma}
\begin{proof}
By symmetry, it suffices to prove the upper bounds.
From the definition of $C_l$ in Proposition \ref{p:chain-reformulate}, we have
\begin{equation}  \label{eq:ratioCpm}
\frac{C_l(\bx_-)}{C_l(\bx_+)} < e^{|2\zeta_{l+1\to l}^{\bx_-}-2\zeta_{l+1\to l}^{\bx_+}|}.   
\end{equation}
If $x_-(v)=x_+(v)$ for any $v\in\GG^{(a)}$,
for each $1\le j <l$ we have $\zeta_j^{\bx_-}=\zeta_j^{\bx_+}$ and $\zeta_{j+1}^{\bx_-}=\zeta_{j+1}^{\bx_+}$;
then from the definition of $U_j$ in \eqref{eq:def-U} we have 
\begin{equation}  \label{eq:ratioUpm}
\frac{U_{j}(\bx_-)}{U_{j}(\bx_+)} < e^{|2\zeta_{j+2\to j+1}^{\bx_-}-2\zeta_{j+2\to j+1}^{\bx_+}|}, \;\forall 1\le j <l.
\end{equation}
By Lemma \ref{lemma:bp-eq}, for $l\le \ell < a$ we have $\tanh(|\zeta_{\ell+1\to \ell}^{\bx_-}-\zeta_{\ell+1\to \ell}^{\bx_+}|) \le 2\theta \tanh(|\zeta_{\ell+2\to \ell+1}^{\bx_-}-\zeta_{\ell+2\to \ell+1}^{\bx_+}|)$, and $\tanh(|\zeta_{a+1\to a}^{\bx_-}-\zeta_{a+1\to a}^{\bx_+}|) \le 2\theta$.
By multiplying these together we get
\[\tanh(|\zeta_{l+1\to l}^{\bx_-}-\zeta_{l+1\to l}^{\bx_+}|) 
< (2\theta)^{a+1-l}.\]
This with \eqref{eq:ratioCpm} and \eqref{eq:ratioUpm} gives the desired upper bounds for $\frac{C_l(\bx_-)}{C_l(\bx_+)},
\frac{U_{j}(\bx_-)}{U_{j}(\bx_+)}$.
Similarly, when $x_-(v)=x_+(v)$ for any $v\not\in\GG^{(a-1)}$, we have
\[
\frac{C_r(\bx_-)}{C_r(\bx_+)} < e^{|2\zeta_{r-1\to r}^{\bx_-}-2\zeta_{r-1\to r}^{\bx_+}|}, \quad
\frac{W_{j}(\bx_-)}{W_{j}(\bx_+)} < e^{|2\zeta_{j-2\to j-1}^{\bx_-}-2\zeta_{j-2\to j-1}^{\bx_+}|}, \;\forall r< j \le n.
\]
Using Lemma \ref{lemma:bp-eq} we can similarly get
$\tanh(|\zeta_{r-1\to r}^{\bx_-}-\zeta_{r-1\to r}^{\bx_+}|) 
< (2\theta)^{r-a+1}$;
thus the desired upper bounds for $\frac{C_r(\bx_-)}{C_r(\bx_+)},
\frac{W_{j}(\bx_-)}{W_{j}(\bx_+)}$ also hold.
\end{proof}

\begin{lemma}  \label{lemma:key-est-prep1-UWC}
For any $l-1\le a<b<r$, vector $\bx\in\R^{\TT_R}$, we have
\begin{equation}  \label{eq:key-prep1-2}
1 - C(2\theta)^{a+2-l} < 
\frac{C_l^{(b+1)}(\bx)}{C_l^{(a+1)}(\bx)},\;
\frac{U^{(b+1)}_{j}(\bx)}{U^{(a+1)}_{j}(\bx)}
< 1 + C(2\theta)^{a+2-l},    
\end{equation}
for any $1\le j <l$;
and
\begin{equation}  \label{eq:key-prep1-3}
1 - C(2\theta)^{r-b-1} < 
\frac{C_r^{(b+1)}(\bx)}{C_r^{(a+1)}(\bx)},\;
\frac{W^{(b+1)}_{j}(\bx)}{W^{(a+1)}_{j}(\bx)}
< 1 + C(2\theta)^{r-b-1},    
\end{equation}
for any $r<j\le n$. 
Here $C$ is an absolute constant.
\end{lemma}
\begin{proof}
For $\bx^{(a+1)}$ and $\bx^{(b+1)}$, we have that
$\bx^{(a+1)}(v)=\bx^{(b+1)}(v)$ for $v\in\GG^{(a+1)}$, or $v\not\in \GG^{(b+1)}$.
Thus by Lemma \ref{lemma:key-est-UVC}, we get \eqref{eq:key-prep1-2}, and \eqref{eq:key-prep1-3} when $b<r-1$.

Now we prove \eqref{eq:key-prep1-3} under the case where $b=r-1$. By taking $C>1$ it suffices to prove the upper bound in \eqref{eq:key-prep1-3}, i.e. to upper bound $\frac{C_r^{(r)}(\bx)}{C_r^{(a+1)}(\bx)}$ and 
$\frac{W^{(r)}_{j}(\bx)}{W^{(a+1)}_{j}(\bx)}$ by a constant.
Note that $C_r^{(r)}(\bx)=C_r^{(r-1)}(\bx)$, so we have $\frac{C_r^{(r)}(\bx)}{C_r^{(a+1)}(\bx)}=\frac{C_r^{(r-1)}(\bx)}{C_r^{(a+1)}(\bx)} < 1+C(2\theta)<1+C$.
By \eqref{eq:bound-W}, we have that $W^{(r)}_{j}(\bx)$ is upper bounded by a constant.
Note that as $a+1<r$ and $j>r$, $\zeta_{j-1}^{\bx^{(a+1)}}=1$, so we have
$W_j^{(a+1)}(\bx)\ge \frac{2}{\cosh(2\beta)+\cosh(2\zeta^{\bx^{(a+1)}}_{j-2\to j-1})}$, which is lower bounded by a constant.
Thus $\frac{W^{(r)}_{j}(\bx)}{W^{(a+1)}_{j}(\bx)}$ is upper bounded by a constant.
\end{proof}

\subsection{Bounds for normalized partition functions}

We move on to the study of the normalized partition functions $\oZ^{(\ell)}_{u,v}(\bx)$ defined in Section \ref{subsec:ind-coup:ind-coup}.
Recall that $\oZ_{u,v}^{(\ell)}(\bx)=\oZ_{u,v}(\bx^{(\ell)})$ if at least one of $u,v \in \GG^{(\ell)}$, and $\oZ_{u,v}^{(\ell)}(\bx)\equiv 1$ otherwise.
Also recall that $\rho(u)$ is the `projection' of $u$ onto $[v_l, v_r]$ for each $u\in\TT_R$;
and for any $u,v \in \TT_R$, and $l\le a\le b \le r$, $\bbd_a^b(u, v)$ is the distance between the intervals $[a, b]$ and $[\rho(u), \rho(v)]$ (if $\rho(u)\le\rho(v)$) or $[\rho(v), \rho(u)]$ (if $\rho(v)\le\rho(u)$).
\begin{lemma}  \label{lemma:key-est-prep1}
For any $l-1\le a<b<r$, vector $\bx\in\R^{\TT_R}$, and $u, v\in\TT_R$, we have
\begin{equation}  \label{eq:B4}
1 - C(2\theta)^{\bbd_{a+2}^{b+1}(u,v)} < 
\frac{\oZ^{(a+1)}_{u,v}(\bx)^2}{\oZ^{(b+1)}_{u,v}(\bx)^2}
< 1 + C(2\theta)^{\bbd_{a+2}^{b+1}(u,v)},    
\end{equation}
where $C$ is an absolute constant.
\end{lemma}

To prove this lemma, we need the following notion that extends \eqref{eq:def:ZGh}.
Take any $\HH\subset \GG$, $\GG\subset \GG'$, and $h\in \{\pm 1\}^\HH$, and define
 \begin{equation}\label{eq:def:ZGh:prime}
 \tZ_{\GG}^{\GG',h}(\bx) := \sum_{\substack{(\sigma_v)_{v\in\GG}\\ \sigma_v=h(v),\forall v\in\HH}} \exp\left( \sum_{(u,v)\in E(\GG)} \beta_{u,v}\sigma_u\sigma_v + \sum_{v\in \GG} (x(v)+\zeta_{\GG'}^\bx(v))\sigma_v \right).
 \end{equation}
 The only difference compared to \eqref{eq:def:ZGh} is that we use $\zeta_{\GG'}^{\bx}(v)$ in the RHS instead of $\zeta_{\GG}^{\bx}$. Recalling the definition \eqref{eq:def:zetaG}, this means that we purposefully avoid considering the induced external fields coming from some branches outside of $\GG$ by setting $\GG'\supsetneq \GG$ if necessary.

\begin{proof}[Proof of Lemma \ref{lemma:key-est-prep1}]
We first consider the case where $u,v \not\in \GG^{(a+1)}$. We have that $\oZ^{(a+1)}_{u,v}(\bx) = 1$ and $\oZ^{(b+1)}_{u,v}(\bx) \ge 1$, and the second inequality of \eqref{eq:B4} holds.
If $u,v \not\in \GG^{(b+1)}$, then $\oZ^{(b+1)}_{u,v}(\bx) = 1$ and the first inequality of \eqref{eq:B4} holds; if at least one of $u, v$ is in $\GG^{(b+1)}$, we must have that $\bbd_{a+2}^{b+1}(u,v) = 0$, and the first inequality of \eqref{eq:B4} also holds by taking $C\ge 1$.

We next consider the case where at least one of $u, v$ is in $\GG^{(a+1)}$; and by symmetry we assume that $u\in\GG^{(a+1)}$.
If $v\not\in\GG^{(a+1)}$, we have that $\bbd_{a+2}^{b+1}(u,v) = 0$, and the first inequality of \eqref{eq:B4} holds (by taking $C\ge 1$).
Also we must have $v_{a+1}, v_{a+2}\in[u,v]$. Let $\HH_1$ be the subgraph generated by vertices in $[u,v]\cap\GG^{(a+1)}$, and $\HH_2$ be the subgraph generated by vertices in $[u,v]\setminus\GG^{(a+1)}$.
We have that
\begin{equation}  \label{eq:tZ-ratio}
\begin{split}
&\frac{\oZ^{(a+1)}_{u,v}(\bx)}{\oZ^{(b+1)}_{u,v}(\bx)}
=
\frac{\tZ_{[u,v]}(\bx^{(a+1)})}{\tZ_{[u,v]}(\bx^{(b+1)
})}\\
&=
\frac{\sum_{\sigma, \sigma'\in\{\pm 1\}}\tZ_{\HH_1}^{[u,v];\{v_{a+1}\mapsto \sigma\}}(\bx^{(a+1)})\tZ_{\HH_2}^{[u,v];\{v_{a+2}\mapsto \sigma'\}}(\bx^{(a+1)})e^{\beta_{v_{a+1},v_{a+2}}\sigma\sigma'}}
{\sum_{\sigma, \sigma'\in\{\pm 1\}}\tZ_{\HH_1}^{[u,v];\{v_{a+1}\mapsto \sigma\}}(\bx^{(b+1)})\tZ_{\HH_2}^{[u,v];\{v_{a+2}\mapsto \sigma'\}}(\bx^{(b+1)})e^{\beta_{v_{a+1},v_{a+2}}\sigma\sigma'}}.
\end{split}
\end{equation}
Here the second equality is by expanding $\tZ_{[u,v]}(\bx^{(a+1)})$ and $\tZ_{[u,v]}(\bx^{(b+1)})$ in terms of the spin at $v_{a+1}$ and $v_{a+2}$.
We first consider the numerator. 
Since $x^{(a+1)}(v)=0$ for any $v\not\in\GG^{(a+1)}$, we have 
\[
2\tZ_{\HH_2}^{[u,v];\{v_{a+2}\mapsto \sigma'\}}(\bx^{(a+1)}) =
\tZ_{\HH_2}^{[u,v]}(\boo)
\]
for any $\sigma'\in\{\pm 1\}$.
Thus the numerator (of the last line in \eqref{eq:tZ-ratio}) equals
\begin{equation}  \label{eq:tZ-ratio-nu}
\tZ_{\HH_2}^{[u,v]}(\boo)
\cosh(\beta_{v_{a+1},v_{a+2}})\sum_{\sigma\in\{\pm 1\}}\tZ_{\HH_1}^{[u,v];\{v_{a+1}\mapsto \sigma\}}(\bx^{(a+1)}).    
\end{equation}
For the denominator (of the last line in \eqref{eq:tZ-ratio}), by Cauchy-Schwartz we have
\[
4\tZ_{\HH_2}^{[u,v];\{v_{a+2}\mapsto 1\}}(\bx^{(b+1)})\tZ_{\HH_2}^{[u,v];\{v_{a+2}\mapsto -1\}}(\bx^{(b+1)})
\ge \tZ_{\HH_2}^{[u,v]}(\boo)^2 .
\]
Thus the denominator (of the last line in \eqref{eq:tZ-ratio}) is at least
\begin{equation}  \label{eq:tZ-ratio-de}
\tZ_{\HH_2}^{[u,v]}(\boo)\sum_{\sigma\in\{\pm 1\}}\tZ_{\HH_1}^{[u,v];\{v_{a+1}\mapsto \sigma\}}(\bx^{(b+1)}).
\end{equation}
Since $x^{(a+1)}(v)=x^{(b+1)}(v)$ for $v\in\GG^{(a+1)}$, we have $\tZ_{\HH_1}^{[u,v];\{v_{a+1}\mapsto \sigma\}}(\bx^{(a+1)})=\tZ_{\HH_1}^{[u,v];\{v_{a+1}\mapsto \sigma\}}(\bx^{(b+1)})$ for each $\sigma\in\{\pm 1\}$.
Then by taking the ratio of \eqref{eq:tZ-ratio-nu} over \eqref{eq:tZ-ratio-de}, and using \eqref{eq:tZ-ratio}, we conclude that
\[
\frac{\oZ^{(a+1)}_{u,v}(\bx)}{\oZ^{(b+1)}_{u,v}(\bx)}
\le
 \cosh(\beta_{v_{a+1},v_{a+2}}) \le \cosh(\beta),
\]
and this is bounded by constant.
So the second inequality of \eqref{eq:B4} holds.

Finally, we study the case where both $u, v \in \GG^{(a+1)}$.
Let $v_* = \argmin_{w\in [u,v]}\bd(w, v_{a+2})$.
Using the definitions we have
\begin{equation}  \label{eq:ratio-Z-up-lo-bd}
e^{-|\zeta_{[u,v]}^{\bx^{(a+1)}}(v_*) - \zeta_{[u,v]}^{\bx^{(b+1)}}(v_*)|}
\le
\frac{\oZ^{(a+1)}_{u,v}(\bx)}{\oZ^{(b+1)}_{u,v}(\bx)}
=
\frac{\tZ_{[u,v]}(\bx^{(a+1)})}{\tZ_{[u,v]}(\bx^{(b+1)
})}
\le e^{|\zeta_{[u,v]}^{\bx^{(a+1)}}(v_*) - \zeta_{[u,v]}^{\bx^{(b+1)}}(v_*)|}.
\end{equation}
By using Lemma \ref{lemma:bp-eq} on each edge in the path $[v_*, v_{a+2}]$, and the fact that $x^{(a+1)}(w)=x^{(b+1)}(w)$ for each $w\in \GG^{(a+1)}$, we have
\[
\begin{split}
&\tanh(|\zeta_{[u,v]}^{\bx^{(a+1)}}(v_*) - \zeta_{[u,v]}^{\bx^{(b+1)}}(v_*)|) \\
\le & (2\theta)^{\bd(v_*, v_{a+2})} \tanh(
|\sum_{w\sim v_{a+2}, w\neq v_{a+1}} \zeta_{w\to v_{a+2}}^{\bx^{(b+1)}} + x(v_{a+2}) - \zeta_{w\to v_{a+2}}^{\bx^{(a+1)}}|
)\\
< & (2\theta)^{\bd(v_*, v_{a+2})} \le (2\theta)^{\bbd_{a+2}^{b+1}(u,v)}.
\end{split}
\]
Thus with \eqref{eq:ratio-Z-up-lo-bd} we have $\left|\frac{\oZ^{(a+1)}_{u,v}(\bx)^2}{\oZ^{(b+1)}_{u,v}(\bx)^2} - 1\right| < \frac{2(2\theta)^{\bbd_{a+2}^{b+1}(u,v)}}{1-(2\theta)^{\bbd_{a+2}^{b+1}(u,v)}}$, and our conclusion follows.
\end{proof}
For $\bx_-, \bx_+ \in \R^{\TT_R}$, and $l\le a\le b\le r$, we let
\begin{equation}\label{eq:def:kappa}
\kappa_a^b(\bx_-,\bx_+):=|\{\ell: a\le\ell\le b, x_-(v)=x_+(v), \forall v\in \TT_R\setminus\GG^{(\ell-1)}\}|.
\end{equation}
We also write $\kappa^b(\bx_-,\bx_+):=\kappa_l^b(\bx_-,\bx_+)$.
We note that $\kappa_a^b(\bx_-,\bx_+)$ is the distance from $b-1$ to the set $\{\rho(v): v\in \TT_R, x_-(v)\neq x_+(v)\}$ (recall that $\rho(v)=\argmin_{l\le\ell\le r}\bd(v_\ell, v)$).
\begin{lemma}\label{lemma:key-est-prep2}
For any $l\le b<r$, vector $\bx_-,\bx_+\in\R^{\TT_R}$, and $u, v\in\TT_R$, assuming $\kappa^{b+1}(\bx_-,\bx_+)\ge 1$ we have
\begin{equation}   \label{eq:key-est-prep2}
\frac{\oZ^{(b+1)}_{u,v}(\bx_-)^2\oZ^{(b)}_{u,v}(\bx_+)^2}{\oZ^{(b)}_{u,v}(\bx_-)^2\oZ^{(b+1)}_{u,v}(\bx_+)^2}
<
1 + C(2\theta)^{\kappa^{b+1}(\bx_-,\bx_+)},    
\end{equation}
where $C$ is an absolute constant.
\end{lemma}
\begin{proof}
Denote $\chi:=b+1-\kappa^{b+1}(\bx_-,\bx_+)$, then $x_-(w)=x_+(w)$ for any $w\in \TT_R\setminus \GG^{(\chi)}$.
We first study the case where $u,v\not\in\GG^{(b)}$.
If in addition $u,v\not\in\GG^{(b+1)}$, then the LHS of \eqref{eq:key-est-prep2} equals $1$ and the statement holds.
If at least one of $u, v$ is in $\GG^{(b+1)}$, we have
\[
\frac{\oZ^{(b+1)}_{u,v}(\bx_-)^2\oZ^{(b)}_{u,v}(\bx_+)^2}{\oZ^{(b)}_{u,v}(\bx_-)^2\oZ^{(b+1)}_{u,v}(\bx_+)^2}
=
\frac{\tZ_{[u,v]}(\bx_-^{(b+1)})^2}{\tZ_{[u,v]}(\bx_+^{(b+1)})^2}
\le
e^{|\zeta_{[u,v]}^{\bx_-^{(b+1)}}(v_*) - \zeta_{[u,v]}^{\bx_+^{(b+1)}}(v_*)|}
\]
where $v_*=\argmin_{w\in[u,v]}\bd(w,v_{\chi})$.
As $x_-^{(b+1)}(w)=x_+^{(b+1)}(w)$ for any $w\not\in \GG^{(\chi)}$, by using Lemma \ref{lemma:bp-eq} for each edge in the path $[v_*, v_\chi]$ we have
\[
\tanh(|\zeta_{[u,v]}^{\bx_-^{(b+1)}}(v_*) - \zeta_{[u,v]}^{\bx_+^{(b+1)}}(v_*)|)
\le (2\theta)^{\bd(v_*,v_\chi)} \le (2\theta)^{\kappa^{b+1}(\bx_-,\bx_+)},
\]
and this implies \eqref{eq:key-est-prep2} (in the case where $u,v\not\in\GG^{(b)}$).

We next study the case where at least one of $u,v$ is in $\GG^{(b)}$. 
Let $w_1,\cdots, w_s$ denote all the vertices in $[u,v]$.
Consider the graph of $\TT_R$ removing the edges in $[u,v]$, 
and let the graphs $\HH_1, \cdots, \HH_s$ be its connected components that contain $w_1, \cdots, w_s$, respectively. We also recall the notation \eqref{eq:def:ZGh:prime}.
For $1\le j \le s$ and any $\bx\in \R^{\TT_R}$, note that 
\[
\zeta_{[u,v]}^{\bx}(w_j) + x(w_j) = \frac{1}{2} \log \left( \frac{ \tZ_{\HH_j}^{\TT_R,\{w_j\mapsto 1\}}(\bx)}{\tZ_{\HH_j}^{\TT_R,\{w_j\mapsto -1\}}(\bx)} \right).
\]
Thus, by plugging this into \eqref{eq:def:restrictedpartitionfunction:plain} with $\GG=[u,v]$, we can write
\begin{equation}  \label{eq:key-est-prep2:1}
\tZ_{u,v}(\bx)^2 = \frac{\tZ(\bx)^2}{\prod_{j=1}^s \tZ_{\HH_j}^{\TT_R,\{w_j\mapsto 1\}}(\bx)\tZ_{\HH_j}^{\TT_R,\{w_j\mapsto -1\}}(\bx)}.
\end{equation}
Plugging this formula into the LHS of \eqref{eq:key-est-prep2}, we then need to bound
\begin{equation}  \label{eq:key-est-prep2:2}
\frac{\tZ_{\HH_j}^{\TT_R,\{w_j\mapsto \sigma\}}(\bx_-^{(b)})\tZ_{\HH_j}^{\TT_R,\{w_j\mapsto \sigma\}}(\bx_+^{(b+1)})}{\tZ_{\HH_j}^{\TT_R,\{w_j\mapsto \sigma\}}(\bx_-^{(b+1)})\tZ_{\HH_j}^{\TT_R,\{w_j\mapsto \sigma\}}(\bx_+^{(b)})} .
\end{equation}
for each $j=1,\cdots, s$ and $\sigma\in\{-1,1\}$; and
\begin{equation}  \label{eq:key-est-prep2:21}
\frac{\tZ(\bx_-^{(b+1)})\tZ(\bx_+^{(b)})}{\tZ(\bx_-^{(b)})\tZ(\bx_+^{(b+1)})}
\end{equation}
We note that $x_-^{(b)}(w)=x_+^{(b)}(w)$ and $x_-^{(b+1)}(w)=x_+^{(b+1)}(w)$ for any $w\not\in \GG^{(\chi)}$;
and $x_-^{(b)}(w)=x_-^{(b+1)}(w)$ and $x_+^{(b)}(w)=x_+^{(b)}(w)$ for any $w\in \GG^{(b)}$ or $w\not\in \GG^{(b+1)}$.
This means that for any $\HH_j$ and $\sigma\in\{\pm 1\}$, \eqref{eq:key-est-prep2:2} equals $1$, 
unless $\HH_j$ intersects both $\GG^{(\chi)}$ and $\GG^{(b+1)}\setminus \GG^{(b)}$. However, this means that $v_b, v_{b+1} \in \HH_j$, which happens to at most one of $\HH_1, \cdots, \HH_s$ (since they are disjoint).

We hold on estimating \eqref{eq:key-est-prep2:2} for a particular $j$, and bound \eqref{eq:key-est-prep2:21} first.
Similar to \eqref{eq:key-est-prep2:1}, we have
\begin{equation}   \label{eq:key-est-prep2:1s}
\tZ_{v_b, v_{b+1}}(\bx)^2 = \frac{\tZ(\bx)^2}{ \tZ_{\GG^{(b)}}^{\TT_R,\{v_b\mapsto 1\}}(\bx)\tZ_{\GG^{(b)}}^{\TT_R,\{v_b\mapsto -1\}}(\bx)
\tZ_{\TT_R\setminus\GG^{(b)}}^{\TT_R,\{v_{b+1}\mapsto 1\}}(\bx)\tZ_{\TT_R\setminus\GG^{(b)}}^{\TT_R,\{v_{b+1}\mapsto -1\}}(\bx)}.    
\end{equation}
We note that for $\GG^{(b)}$ and $\TT_R\setminus\GG^{(b)}$, neither of them intersects both $\GG^{(\chi)}$ and $\GG^{(b+1)}\setminus \GG^{(b)}$.
So for $f\in \{\tZ_{\GG^{(b)}}^{\TT_R,\{v_b\mapsto 1\}}, \tZ_{\GG^{(b)}}^{\TT_R,\{v_b\mapsto -1\}}, 
\tZ_{\TT_R\setminus\GG^{(b)}}^{\TT_R,\{v_{b+1}\mapsto 1\}}, \tZ_{\TT_R\setminus\GG^{(b)}}^{\TT_R,\{v_{b+1}\mapsto -1\}}\}$, we have
\[
\frac{f(\bx_-^{(b+1)})f(\bx_+^{(b)})}{f(\bx_-^{(b)})f(\bx_+^{(b+1)})}=1.
\]
Thus by \eqref{eq:key-est-prep2:1s} we have
\[
\frac{\tZ(\bx_-^{(b+1)})\tZ(\bx_+^{(b)})}{\tZ(\bx_-^{(b)})\tZ(\bx_+^{(b+1)})}
=
\frac{\tZ_{v_b,v_{b+1}}(\bx_-^{(b+1)})\tZ_{v_b,v_{b+1}}(\bx_+^{(b)})}{\tZ_{v_b,v_{b+1}}(\bx_-^{(b)})\tZ_{v_b,v_{b+1}}(\bx_+^{(b+1)})}
\]
The RHS can be expanded as
\begin{equation}  \label{eq:key-est-prep2:1expr}
\frac{\cosh \left\{\zeta_{v_b\to v_{b+1}}^{\bx_-}+x_-(v_{b+1}) + 
\sum_{w\sim v_{b+1}, w\neq v_{b}} \zeta_{w\to v_{b+1}}^{\bx_-}
 \right\} \cosh(\zeta_{v_b\to v_{b+1}}^{\bx_+})}{\cosh(\zeta_{v_b\to v_{b+1}}^{\bx_-})\cosh\left\{\zeta_{v_b\to v_{b+1}}^{\bx_+}+x_+(v_{b+1}) + \sum_{w\sim v_{b+1}, w\neq v_{b}} \zeta_{w\to v_{b+1}}^{\bx_+}\right\}}.    
\end{equation}
By our assumption of $\kappa^{b+1}(\bx_-,\bx_+)\ge 1$, we have $x_-(v_{b+1})=x_+(v_{b+1})$ and $\sum_{w\sim v_{b+1}, w\neq v_{b}} \zeta_{w\to v_{b+1}}^{\bx_-} = \sum_{w\sim v_{b+1}, w\neq v_{b}} \zeta_{w\to v_{b+1}}^{\bx_+}$. Thus by the expression \eqref{eq:key-est-prep2:1expr} we have
\begin{equation}  \label{eq:key-est-prep2:bdexp}
\frac{\tZ(\bx_-^{(b+1)})\tZ(\bx_+^{(b)})}{\tZ(\bx_-^{(b)})\tZ(\bx_+^{(b+1)})}
\le e^{2|\zeta_{v_b\to v_{b+1}}^{\bx_-} - \zeta_{v_b\to v_{b+1}}^{\bx_+}|}.    
\end{equation}
By using Lemma \ref{lemma:bp-eq} for each edge in $[v_{b+1}, v_\chi]$, we get
\[
\tanh(|\zeta_{v_b\to v_{b+1}}^{\bx_-} - \zeta_{v_b\to v_{b+1}}^{\bx_+}|) \le (2\theta)^{\kappa^{b+1}(\bx_-,\bx_+)}.
\]
This with \eqref{eq:key-est-prep2:bdexp} implies that
\begin{equation} \label{eq:key-est-prep2:3}
\frac{\tZ(\bx_-^{(b+1)})\tZ(\bx_+^{(b)})}{\tZ(\bx_-^{(b)})\tZ(\bx_+^{(b+1)})}<1+C'(2\theta)^{\kappa^{b+1}(\bx_-,\bx_+)}    
\end{equation}
for some constant $C'$.

Now we go back to bound \eqref{eq:key-est-prep2:2}, for some $j$ such that $\HH_j$ intersects both $\GG^{(\chi)}$ and $\GG^{(b+1)}\setminus \GG^{(b)}$.
Such bound can be directly obtained from \eqref{eq:key-est-prep2:3}, by exchanging $\bx_-$ and $\bx_+$, and taking the following special case of $\bx_\pm$: first set $x_-(v)=x_+(v)=0$ for $v\not\in\HH_j$, then send $x_-(w_j)=x_+(w_j)$ to $\infty$ (if $\sigma=1$) or $-\infty$ (if $\sigma=-1$).
So we conclude that \eqref{eq:key-est-prep2:2} for such particular $j$ is also bounded by $1+C'(2\theta)^{\kappa^{b+1}(\bx_-,\bx_+)}$.
Then the LHS of \eqref{eq:key-est-prep2} is bounded by $(1+C'(2\theta)^{\kappa^{b+1}(\bx_-,\bx_+)})^4$, and our conclusion follows.
\end{proof}

\subsection{Proof of the key estimate}

We conclude this section by establishing Proposition \ref{prop:key-esti}.

\begin{proof}[Proof of Proposition \ref{prop:key-esti}]
Recall the definition of $\kappa^{b'}_{a'}$ \eqref{eq:def:kappa} and $\kappa^{b'}$ for $l\le a'\le b' \le r$. 

We first prove  \eqref{eq:key-est1}.
We note that for each $l\le a'\le b'\le r$, and $t\ge 0$, we have $\kappa_{a'}^{b'}(\btau^-, \btau^+) = \kappa_{a'}^{b'}(\bY^-_t, \bY^+_t)$.
Thus in this proof, we write $\kappa_{a'}^{b'} = \kappa_{a'}^{b'}(\btau^-, \btau^+)$ and $\kappa^{b'}=\kappa^{b'}(\btau^-, \btau^+)$.

First consider the case where $\kappa_{a+2}^{b+1}=0$.
We have $E_b\le R_b^+\vee R_b^-$ and $\prod_{a<\ell<b}P_\ell \le \prod_{a<\ell<b}R_\ell^+ \vee \prod_{a<\ell<b}R_\ell^-$, so
\begin{equation}  \label{eq:EPze-ca}
\begin{split}
&
E_b\prod_{a<\ell<b}P_\ell \le
\prod_{a<\ell\le b} R_\ell^+ \vee \prod_{a<\ell\le b} R_\ell^-
\\
=&
\prod_{i=0}^{k+1}\frac{\oZ_{u_i,u_{i-1}}^{(a+1)}(\bY_{t_i}^+)^2}
{\oZ_{u_i,u_{i-1}}^{(b+1)}(\bY_{t_i}^+)^2}
\frac{C_l^{(b+1)}(\bY_{t_-}^+)}{C_l^{(a+1)}(\bY_{t_-}^+)}
\frac{C_r^{(b+1)}(\bY_{t_+}^+)}{C_r^{(a+1)}(\bY_{t_+}^+)}
\prod_{1\le j <l}
\frac{U_j^{(b+1)}(\bY_{t_-}^+)}{U_j^{(a+1)}(\bY_{t_-}^+)}
\prod_{r< j \le n}
\frac{W_j^{(b+1)}(\bY_{t_+}^+)}{W_j^{(a+1)}(\bY_{t_+}^+)}
\\
& \vee
\prod_{i=0}^{k+1}\frac{\oZ_{u_i,u_{i-1}}^{(a+1)}(\bY_{t_i}^-)^2}
{\oZ_{u_i,u_{i-1}}^{(b+1)}(\bY_{t_i}^-)^2}
\frac{C_l^{(b+1)}(\bY_{t_-}^-)}{C_l^{(a+1)}(\bY_{t_-}^-)}
\frac{C_r^{(b+1)}(\bY_{t_+}^-)}{C_r^{(a+1)}(\bY_{t_+}^-)}
\prod_{1\le j <l}
\frac{U_j^{(b+1)}(\bY_{t_-}^-)}{U_j^{(a+1)}(\bY_{t_-}^-)}
\prod_{r< j \le n}
\frac{W_j^{(b+1)}(\bY_{t_+}^-)}{W_j^{(a+1)}(\bY_{t_+}^-)}
.    
\end{split}    
\end{equation}
By Lemma \ref{lemma:key-est-prep1-UWC} and \ref{lemma:key-est-prep1}, 
we have
\begin{equation}  \label{eq:key-est-pf}
\prod_{a<\ell\le b} R_\ell^+,\; \prod_{a<\ell\le b} R_\ell^-
<\prod_{i=0}^{k+1}
(1 + C(2\theta)^{\bbd_{a+2}^{b+1}(u_i,u_{i-1})})
(1+C(2\theta)^{a+2-l})^{l}
(1+C(2\theta)^{r-b-1})^{n-r+1}.
\end{equation}

We next consider the case where $\kappa_{a+2}^{b+1}\ge 1$.
Without loss of generality, we assume that $R_b^+\ge R_b^-$.
We have
\begin{equation}  \label{eq:EP-ini-bd}
E_b\prod_{a<\ell<b}P_\ell \le
E_b\prod_{a<\ell<b}R_\ell^- =
\left(\frac{R_b^+}{R_b^-}-1\right)
\prod_{a<\ell\le b}R_\ell^- .    
\end{equation}
For factor $\prod_{a<\ell\le b}R_\ell^-$, it is again bounded using \eqref{eq:key-est-pf}.
It remains to bound $\frac{R_b^+}{R_b^-}$.
Recall the definition of $R_b^{\pm}$ from \eqref{eq:defn-R}.
We consider the factors one by one.
For each $0\le i \le k+1$, by Lemma \ref{lemma:key-est-prep2} we have
\begin{equation} \label{eq:4Zratio-bd-1}
\frac{\oZ_{u_i,u_{i-1}}^{(b)}(\bY_{t_i}^+)^2\oZ_{u_i,u_{i-1}}^{(b+1)}(\bY_{t_i}^-)^2}
{\oZ_{u_i,u_{i-1}}^{(b+1)}(\bY_{t_i}^+)^2\oZ_{u_i,u_{i-1}}^{(b)}(\bY_{t_i}^-)^2}
< 1 + C(2\theta)^{\kappa^{b+1}}
\le 1 + C(2\theta)^{\kappa_{a+2}^{b+1}}.
\end{equation}
Recall that for any $l\le a' \le r$, we denote $\bbd_{a'}=\bbd_{a'}^{a'}$.
If $\bbd_{b+1}(u_i,u_{i-1})\ge 1$, by Lemma \ref{lemma:key-est-prep1}, we also have
\begin{equation}  \label{eq:4Zratio-bd-2}
\frac{\oZ_{u_i,u_{i-1}}^{(b)}(\bY_{t_i}^+)^2\oZ_{u_i,u_{i-1}}^{(b+1)}(\bY_{t_i}^-)^2}
{\oZ_{u_i,u_{i-1}}^{(b+1)}(\bY_{t_i}^+)^2\oZ_{u_i,u_{i-1}}^{(b)}(\bY_{t_i}^-)^2}
< 
\frac{1 + C(2\theta)^{\bbd_{b+1}(u_i,u_{i-1})}}{1 - C(2\theta)^{\bbd_{b+1}(u_i,u_{i-1})}}
<
1 + C'(2\theta)^{\bbd_{b+1}(u_i,u_{i-1})},
\end{equation}
for some constant $C'>C$. Thus by combining \eqref{eq:4Zratio-bd-1} and \eqref{eq:4Zratio-bd-2}, we have
\begin{equation}  \label{eq:4Zratio-bd}
\frac{\oZ_{u_i,u_{i-1}}^{(b)}(\bY_{t_i}^+)^2\oZ_{u_i,u_{i-1}}^{(b+1)}(\bY_{t_i}^-)^2}
{\oZ_{u_i,u_{i-1}}^{(b+1)}(\bY_{t_i}^+)^2\oZ_{u_i,u_{i-1}}^{(b)}(\bY_{t_i}^-)^2}
<
1 + C'(2\theta)^{\bbd_{b+1}(u_i,u_{i-1})\vee \kappa_{a+2}^{b+1}}.    
\end{equation}
If $u_i, u_{i-1} \not\in \GG^{(b+1)}$, the LHS equals $1$; otherwise, we have
\[
(\bbd_{b+1}(u_i,u_{i-1})- \kappa_{a+2}^{b+1}) \vee 0
\ge (\bbd_{b+1}(u_i,u_{i-1})- (b-a)) \vee 0
= \bbd_{a+1}^{b+1}(u_i,u_{i-1}).
\]
By plugging this into \eqref{eq:4Zratio-bd} we have
\begin{equation}  \label{eq:key-est-pf3}
\frac{\oZ_{u_i,u_{i-1}}^{(b)}(\bY_{t_i}^+)^2\oZ_{u_i,u_{i-1}}^{(b+1)}(\bY_{t_i}^-)^2}
{\oZ_{u_i,u_{i-1}}^{(b+1)}(\bY_{t_i}^+)^2\oZ_{u_i,u_{i-1}}^{(b)}(\bY_{t_i}^-)^2}
<
1 + C'(2\theta)^{\bbd_{a+1}^{b+1}(u_i,u_{i-1}) +  \kappa_{a+2}^{b+1}}.    
\end{equation}

By Lemma \ref{lemma:key-est-prep1-UWC},
we have
$\frac{C_l^{(b+1)}(\bY_{t_-}^+)}{C_l^{(b)}(\bY_{t_-}^+)} < 1+C(2\theta)^{b+1-l}$, 
$\frac{C_l^{(b)}(\bY_{t_-}^-)}{C_l^{(b+1)}(\bY_{t_-}^-)} < (1-C(2\theta)^{b+1-l})^{-1}$;
and $\frac{U_j^{(b+1)}(\bY_{t_-}^+)}{U_j^{(b)}(\bY_{t_-}^+)} < 1+C(2\theta)^{b+1-l}$,
$\frac{U_j^{(b)}(\bY_{t_-}^-)}{U_j^{(b+1)}(\bY_{t_-}^-)} < (1-C(2\theta)^{b+1-l})^{-1}$,
for each $1\le j <l$.
Thus we have
\begin{equation}  \label{eq:key-est-pf4}
\frac{C_l^{(b+1)}(\bY_{t_-}^+)}{C_l^{(b)}(\bY_{t_-}^+)}\frac{C_l^{(b)}(\bY_{t_-}^-)}{C_l^{(b+1)}(\bY_{t_-}^-)}, \;
\frac{U_j^{(b+1)}(\bY_{t_-}^+)}{U_j^{(b)}(\bY_{t_-}^+)}
\frac{U_j^{(b)}(\bY_{t_-}^-)}{U_j^{(b+1)}(\bY_{t_-}^-)}< 1+C'(2\theta)^{a+1-l+\kappa_{a+2}^{b+1}}.
\end{equation}

By Lemma \ref{lemma:key-est-UVC}, each of $\frac{C_r^{(b+1)}(\bY_{t_-}^+)}{C_r^{(b+1)}(\bY_{t_-}^-)}$, $\frac{C_r^{(b)}(\bY_{t_-}^-)}{C_r^{(b)}(\bY_{t_-}^+)}$
and $\frac{W_j^{(b+1)}(\bY_{t_-}^+)}{W_j^{(b+1)}(\bY_{t_-}^-)}$, $\frac{W_j^{(b)}(\bY_{t_-}^-)}{W_j^{(b)}(\bY_{t_-}^+)}$, $r< j \le n$, is bounded by $1+C(2\theta)^{r-b-1+\kappa_{a+2}^{b+1}}$.
So we have
\begin{equation}  \label{eq:key-est-pf5}
\frac{C_r^{(b+1)}(\bY_{t_-}^+)}{C_r^{(b)}(\bY_{t_-}^+)}\frac{C_r^{(b)}(\bY_{t_-}^-)}{C_r^{(b+1)}(\bY_{t_-}^-)},\;
\frac{W_j^{(b+1)}(\bY_{t_-}^+)}{W_j^{(b)}(\bY_{t_-}^+)}
\frac{W_j^{(b)}(\bY_{t_-}^-)}{W_j^{(b+1)}(\bY_{t_-}^-)}
< 1+C'(2\theta)^{r-b-1+\kappa_{a+2}^{b+1}}.    
\end{equation}
By putting together \eqref{eq:key-est-pf3}, \eqref{eq:key-est-pf4}, \eqref{eq:key-est-pf5}, we have
\[
\begin{split}
\frac{R_b^+}{R_b^-}
&\le \prod_{i=0}^{k+1}(1 + C'(2\theta)^{\bbd_{a+1}^{b+1}(u_i,u_{i-1})+\kappa_{a+2}^{b+1}})
(1+C'(2\theta)^{a+1-l+\kappa_{a+2}^{b+1}})^{l}
(1+C'(2\theta)^{r-b-1+\kappa_{a+2}^{b+1}})^{n-r+1} \\
&
\le
1+
(2\theta)^{\kappa_{a+2}^{b+1}}\prod_{i=0}^{k+1}(1 + C'(2\theta)^{\bbd_{a+1}^{b+1}(u_i,u_{i-1})})
(1+C'(2\theta)^{a+1-l})^{l}
(1+C'(2\theta)^{r-b-1})^{n-r+1}.
\end{split}
\]
Thus with \eqref{eq:EP-ini-bd} and the bound of $\prod_{a<\ell\le b}R_\ell^-$ by \eqref{eq:key-est-pf}, 
we conclude that
\[
E_b\prod_{a<\ell<b}P_\ell \le
(2\theta)^{\kappa_{a+2}^{b+1}}
\prod_{i=0}^{k+1}(1 + C''(2\theta)^{\bbd_{a+1}^{b+1}(u_i,u_{i-1})})
(1+C''(2\theta)^{a+1-l})^{l}
(1+C''(2\theta)^{r-b-1})^{n-r+1},
\]
where $C''$ is another constant.
Note that this also holds when $\kappa_{a+2}^{b+1}=0$ (by \eqref{eq:EPze-ca}).
Thus we always have
\begin{multline}  \label{eq:esti-key:2}
\int fE_b\prod_{a<\ell<b}P_\ell d\Lambda \\
\le
\prod_{i=0}^{k+1}(1 + C''(2\theta)^{\bbd_{a+1}^{b+1}(u_i,u_{i-1})})
(1+C''(2\theta)^{a+1-l})^{l}
(1+C''(2\theta)^{r-b-1})^{n-r+1}\int (2\theta)^{\kappa_{a+2}^{b+1}}fd\Lambda.
\end{multline}
From the construction of $\Gamma$ we have $\P_\Gamma(b-a-\kappa_{a+2}^{b+1} \ge j) \le \theta^j$ for any $j\in \Z_+$.
Since $\Gamma(\cdot\mid \FF_{a+1} )=\Lambda(\cdot\mid \FF_{a+1} )$ we would also have $\P_\Lambda(b-a-\kappa_{a+2}^{b+1} \ge j) \le \theta^j$ for any $j\in \Z_+$.
Also note that $f$ is a function of $(B^\pm_t(v))_{t\ge 0, v\not\in \GG^{(a+1)}}$, which is independent of $(\tau^\pm_v)_{v\not\in \GG^{(a+1)}}$.
Thus we have
\begin{equation}  \label{eq:esti-key:3}
\int f(2\theta)^{\kappa_{a+2}^{b+1}}d\Lambda
< 2(2\theta)^{b-a}\E_\Gamma[f]
\end{equation}
By plugging \eqref{eq:esti-key:3} into \eqref{eq:esti-key:2} we get \eqref{eq:key-est1}.

For \eqref{eq:key-est2}, using \eqref{eq:key-est-pf} we have
\[
\begin{split}
&\int f\don[\tau_{v_{b+1}}^- \neq \tau_{v_{b+1}}^+]\prod_{a<\ell\le b}P_\ell d\Lambda
\le \int f\don[\tau_{v_{b+1}}^- \neq \tau_{v_{b+1}}^+]\prod_{a<\ell\le b}R_\ell^+ d\Lambda \\
\le &
\prod_{i=0}^{k+1}
(1 + C(2\theta)^{\bbd_{a+2}^{b+1}(u_i,u_{i-1})})
(1+C(2\theta)^{a+2-l})^{l}
(1+C(2\theta)^{r-b-1})^{n-r+1}
\int f\don[\tau_{v_{b+1}}^- \neq \tau_{v_{b+1}}^+] d\Lambda.
\end{split}
\]
Again, using that $\Gamma(\cdot\mid \FF_{a+1} )=\Lambda(\cdot\mid \FF_{a+1} )$, and $\P_\Gamma(\tau_{v_{b+1}}^- \neq \tau_{v_{b+1}}^+ \mid \FF_{a+1}) \le \theta^{b-a}$,
and the independence of $(\tau^\pm_v)_{v\not\in \GG^{(a+1)}}$ and $(B^\pm_t(v))_{t\ge 0, v\not\in \GG^{(a+1)}}$ under $\Gamma(\cdot\mid \FF_{a+1})$, 
we can bound this by
\[
\theta^{b-a}
\prod_{i=0}^{k+1}
(1 + C(2\theta)^{\bbd_{a+2}^{b+1}(u_i,u_{i-1})})
(1+C(2\theta)^{a+2-l})^{l}
(1+C(2\theta)^{r-b-1})^{n-r+1}
\E_\Gamma[f],
\]
and \eqref{eq:key-est2} follows.
\end{proof}

\end{document}